\documentclass[12pt]{article}
\usepackage{amsthm,amsfonts,amssymb,epsfig,graphics,amsmath,amsbsy,tikz}

\usepackage{mathtools}
\usepackage{enumerate}
\mathtoolsset{showonlyrefs=true}

\newtheorem{theorem}{Theorem}[section] 

\newtheorem{proposition}[theorem]{Proposition}
\newtheorem{remark}[theorem]{Remark}
\newtheorem{definition}[theorem]{Definition}

\usepackage{upquote,textcomp} 

\usepackage[colorlinks = true,
            linkcolor = blue,
            urlcolor  = blue,
            citecolor = blue,
            anchorcolor = blue]{hyperref} 

\newcommand{\sn}{\operatorname{sn}}

\numberwithin{equation}{section}

\DeclareMathOperator*{\argmin}{arg\,min}

\usepackage{listings}
\definecolor{dkgreen}{rgb}{0,0.6,0}
\definecolor{gray}{rgb}{0.5,0.5,0.5}

\title{Comparison of Coarsening Dynamics for the Cahn--Hilliard and Burgers--Cahn--Hilliard equations
}

\author{Peter Howard
\thanks{Department of Mathematics, 
    Texas A\&M University, 
    College Station, 
    TX 77843, USA, 
    phoward@math.tamu.edu}
\and
Adam Larios
\thanks{Department of Mathematics, 
        University of Nebraska--Lincoln,
        Lincoln, NE 68588-0130, USA, 
        alarios@unl.edu}
\and
Quyuan Lin
\thanks{School of Mathematical and Statistical Sciences,              
    Clemson University, 
    Clemson, SC 29634, USA, 
    quyuanl@clemson.edu}
}

\date{May 20, 2024}

\usepackage{placeins} 

\usepackage{subcaption}

\usepackage{todonotes}

\begin{document}

\maketitle

\begin{abstract} 
We consider coarsening dynamics associated with a Burgers--Cahn--Hilliard
system modeling a two-phase flow in one space 
dimension. Our emphasis is on the effect that coupling 
between the phase and fluid dynamics has on coarsening rates, and on the 
mechanisms driving this effect. We start with a detailed
examination of coarsening dynamics for the uncoupled Cahn--Hilliard 
equation, comparing numerically generated rates with two analytic 
methods, and then we consider how these dynamics are affected by 
appropriate coupling with a viscous Burgers equation. In order to keep 
the analysis as self-contained as possible, we establish the global well-posedness of the system  under consideration. 
\end{abstract}


\section{Introduction}\label{introduction}

Coarsening dynamics for solutions to single uncoupled Cahn--Hilliard equations
are relatively well understood, but less is known about the effect on 
coarsening when a Cahn--Hilliard equation is coupled with an appropriate
fluids model to generate two-phase flow. Our primary goal in the current 
analysis is to consider such couplings, and to make some baseline 
observations for future study. For this initial investigation, we would 
especially like to emphasize comparisons between 
analytical methods and computational results, and robust analytical methods
of coarsening dynamics only seem to be available in the one-dimensional (1D) case. 
Accordingly, we will restrict our considerations to the one-dimensional setting, 
bearing in mind that this restriction places natural limitations on our choice 
of fluid model. In particular, since the only incompressible flows in 1D are trivial 
(i.e., piecewise constant), we do not impose incompressibility, and rather proceed 
by modeling fluid velocity with the viscous Burgers equation. 
While somewhat artificial, this one-dimensionalization paradigm has proven fruitful in numerous
related contexts, 
and we believe it has potential to shed substantial insight into the setting of phase-field/fluid coupling.  

Having settled on a fluids model, we turn to the choice of coupling, 
and for this we are guided by the well-known ``model H" from \cite{HH77} (see also \cite{diegel2017convergence, 
gal2010asymptotic, giorgini2019uniqueness, GPV96}). In particular, ``model H" provides our 
coupling term for the momentum equation, and further energy considerations determine 
our coupling for the phase-field equation. (See Section \ref{model_section} 
below for details.)

With the above considerations in mind, we propose to study the following system, 
which we will refer to as the Burgers--Cahn--Hilliard (BCH) system,
\begin{equation} \label{pre-main}
\left\{\begin{aligned}
\phi_t + v \phi_x &= (M (\phi)\mu_{x})_x, \\
\mu &= -\kappa \phi_{xx} + F' (\phi), \\
v_t + v v_x &= \nu v_{xx} + K \mu \phi_x,
\end{aligned}\right.
\end{equation}
where specifications for the constants $\nu$, $\kappa$, and $K$, 
and the functions $M(\phi)$ and $F(\phi)$ will be given following 
system \eqref{main_sys} below. We observe that 
when $\phi \equiv 0$, we formally recover the viscous Burgers equation, 
and when $v \equiv 0$ and the last equation is eliminated, 
we formally recover the usual Cahn--Hilliard equation.
We view this model as a one-dimensional analogue for the 2D or 3D
two-phase Cahn--Hilliard--Navier--Stokes (CHNS) system, given by
\begin{equation} \label{main_sys}
\left\{\begin{aligned}
\phi_t + v \cdot \nabla \phi &=
\nabla \cdot (M (\phi) \nabla \mu), \\
\mu &= -\kappa \Delta \phi + F' (\phi), \\
\rho(v_t + (v\cdot \nabla)v)
&= \nu \Delta v - \nabla p + g + K \mu \nabla \phi, \\ 
\nabla \cdot v &= 0,
\end{aligned}\right.
\end{equation} 
over some open domain $\Omega\subset\mathbb{R}^d$, $d = 2$ or $d = 3$,
and some time interval $[0,T]$ with $T>0$.  
Denoting $Q:=\Omega \times [0, T]$,  $\phi:Q\rightarrow \mathbb{R}$ is a phase variable, 
$v:Q\rightarrow \mathbb{R}^d$
is the fluid velocity, $M:\mathbb{R}\rightarrow \mathbb{R}$ is molecular mobility, 
$\mu:Q\rightarrow \mathbb{R}$ is the chemical potential, $\rho > 0$ is
(constant) fluid density, $\kappa>0$ is a (constant) measure of interfacial 
energy, $F:\mathbb{R}\rightarrow \mathbb{R}$ is bulk free energy density, 
$\nu>0$ is the (constant) coefficient of shear viscosity, $p:Q\rightarrow \mathbb{R}$
is the fluid pressure, $g:Q\rightarrow \mathbb{R}^d$ is a given body force, 
and $K \in \mathbb{R}$ is a coupling constant. 
System \eqref{main_sys} is precisely ``model H" from 
\cite{HH77}, though expressed here in the form of system (54)
from \cite{GPV96}. This system has been the subject of numerous 
studies, including 
\cite{
diegel2017convergence, 
gal2010asymptotic, 
giorgini2019uniqueness, 
GPV96, 
HH77, 
NovickCohen_2000,
SHH76}, 
and closely related models have been considered in 
\cite{AF08, Bray94, Chen2020, gal2011instability, HW2015, LS03, LT98, you2022continuous, Zhao2021}.
(See Section 3 of the current analysis for additional information on how 
these models are related.) In the references mentioned so far, 
the two fluid phases are assumed to have the same constant density, 
but models allowing for different densities have also been 
introduced and studied (see, e.g., \cite{Brummelen2015, Guo2014, hintermuller2018goal,kay2007efficient}).
In addition, the Navier--Stokes system has been coupled with phase-field models
for which three or more phases are present \cite{BLMPQ2010,Kim2005,Kim2007,KL2005, Minjeaud2013}. 
Coarsening rates for the uncoupled Cahn--Hilliard equation are 
discussed in \cite{Chen2020, H11, KO02, L71}, and coarsening rates for 
Cahn--Hilliard equations coupled with a Navier--Stokes system are discussed 
in \cite{Chen2020, HW2015, Zhao2021}, though with a different choice
of coupling than the one taken here. We also mention 
that coarsening rates for CHNS are related to coarsening rates
for the convective Cahn--Hilliard equation, and these are analyzed 
in \cite{watson2003coarsening}. See also 
\cite{EB1996, Leung1988, Leung1990, M-P2016, M-PK2013, Witelski1996} for additional 
work on the convective Cahn--Hilliard equation. 

The primary contributions of this paper are as follows. First, 
we review and compare two analytical coarsening methods for the 
uncoupled Cahn--Hilliard equation, taken from \cite{L71} and 
\cite{H11}. The method of \cite{L71} is based on late-stage 
coarsening, but provides a straightforward expression for 
characteristic coarsening length that can readily 
be computed at all times (see equation \eqref{langer_period} below).
The approach of \cite{H11} adapts the method of \cite{L71} so that earlier 
times are more readily incorporated, yet interestingly the two methods
give extremely similar results up to a scaling factor that will be 
discussed in detail in Section \ref{ch-coarsening}. 
In Section \ref{sec-computational-results} 
we numerically generate coarsening rates based on randomized data, 
and observe that the two analytic methods and the numerically generated rates align 
quite well. In particular, it is well-known (see, e.g., 
\cite{L71, watson2003coarsening}) that coarsening rates are logarithmic
in this case, and our calculations bear this out quantitatively. 
An important aspect of this part of the study is the identification 
of a systematic method for associating a natural length scale $\ell$ with 
any solution to the uncoupled Cahn--Hilliard equation with energy below 
the energy of a natural homogeneous configuration $\phi_0$
described below. 

Having established this firm baseline of coarsening behavior for the 
uncoupled Cahn--Hilliard equation, we turn next in Section \ref{coupled-system-section} 
to our primary overall goal, which is to understand the effect of 
coupling on these coarsening rates. Our main observation for this part of the 
study is that coarsening for the coupled system (with coupling constant
$K = 1$) is orders of magnitude faster than coarsening for the uncoupled system 
(see Figure \ref{fig_periods_all}). Indeed, for the two specific cases 
depicted in Figure \ref{fig_periods_all}, the coarseness achieved by the coupled 
system at times $t_c = 0.25$ and $t_c = 0.34$ (respectively for the two cases)  
is not achieved by the uncoupled system until $t_u = 76.0$, and the order of magnitude 
multiplier between $t_c$ and $t_u$ increases still further at later stages 
of coarseness. In short, gradients in the velocity profile $v(x,t)$ tend to either push transitions 
layers together or pull them apart, and in either case the phase profile becomes more 
coarse than in the uncoupled case. (These dynamics are discussed in detail in 
Section \ref{coupled-system-section}). Nonetheless, it is interesting to 
note that the coarsening rate in both cases appears logarithmic, albeit
with substantially different rate constants. 

{\it Plan of the paper.} In Section \ref{model_section}, we discuss 
considerations that led to our selection of the system \eqref{main_sys} 
as the starting point for this study, along with the reduction to 
\eqref{pre-main}. In Section \ref{ch-coarsening}, we describe the 
analytical methods of coarsening from \cite{L71} and \cite{H11} for 
the uncoupled Cahn--Hilliard equation, and in Section \ref{sec-computational-results} 
we turn to computational results, beginning with an overview 
of our numerical methods (Section \ref{overview-numerical-section}). 
In Section \ref{uncoupled-system-section}, 
we compare coarsening rates obtained analytically with those 
obtained by computation, and in Section \ref{coupled-system-section}
we study the effect of coupling on these rates. 
In Section \ref{chns_section}, 
we establish the global well-posedness of our system \eqref{pre-main}, and in Section \ref{conclusions-section}
we further summarize our results and discuss future directions of inquiry.

\section{Models of Two-Phase Flow} \label{model_section}

Several models have been proposed for coupling Cahn--Hilliard
equations and systems with a momentum equation for fluid dynamics, 
and in this section we discuss our selection of \eqref{main_sys} as 
a starting point for the current study. Aside from our 
explicit incorporation of the mobility function $M(\phi)$, 
\eqref{main_sys} is taken from system (54) from \cite{GPV96},
in which the authors have expressed the well-known ``Model H'' 
from \cite{HH77} in a convenient form. 

\subsection{The pressure and one-dimensionalization}

One approach toward introducing a non-trivial one-dimensional 
analogue to \eqref{main_sys} is to replace 
all spatial derivatives in the momentum equation with their 
one-dimensional counterparts, and to drop the incompressibility 
constraint $\nabla \cdot v = 0$ (which requires a constant or 
piecewise constant flow in one space dimension). Since the main 
role of the pressure is to act as a Lagrange multiplier to enforce 
this constraint, we also drop the pressure gradient $\nabla p$ from 
the system (as Burgers and Bateman did in deriving Burgers equation; 
see \cite{Bateman_1915_burgers,Burgers_1948}), although one should 
be mindful when doing this that in the current setting the 
variable $p$ is subject to various interpretations. Most 
directly for the present work, in the careful derivation of 
\cite{GPV96}, the authors show that in the momentum equation 
of \eqref{main_sys} the pressure $p$ is more properly given by
\begin{equation} \label{gpv96pressure}
    p = p_{o} - \frac{\kappa}{6} |\nabla \phi|^2 + F(\phi),
\end{equation}
where $p_{o}$ denotes {\it original} pressure, as described
in \cite{GPV96} (see system (54) and the adjoining discussion 
in that reference). 
In this way, when we drop the pressure from 
our system, we are really dropping the right-hand side of 
\eqref{gpv96pressure}. 

More generally, models with coupling terms other than the 
one in \eqref{main_sys} have been proposed, and it is 
worth noting how $p$ should be interpreted in some such 
cases. For example, in \cite{Bray94} the author 
proposes a momentum equation of the form 
\begin{equation} \label{bray94momentum}
    \rho (v_t + (v\cdot\nabla)v)
    = \eta \Delta v - \nabla p - \phi \nabla \mu,
\end{equation}
(equation (47) in \cite{Bray94})
which agrees with the 
coupling in \eqref{main_sys} (with $K = 1$) if $p$ is replaced by 
\begin{equation}
    p = p_o - \mu \phi.
\end{equation}
Likewise, in \cite{LS03},
the authors derive the 
momentum equation 
\begin{equation} \label{LS03momentum}
    \rho (v_t + (v\cdot\nabla)v)
    = \eta \Delta v - \nabla p 
    - c \nabla \cdot (\nabla \phi \otimes \nabla \phi)
\end{equation}
(equation (2.1) of \cite{LS03}).
Noting the identity 
\begin{equation} \label{noted-identity}
\mu \nabla \phi = 
\rho \left(-\kappa \nabla \cdot (\nabla \phi \otimes \nabla \phi) 
+ \nabla \left(\frac{\kappa}{2} |\nabla \phi|^2 + F (\phi)\right) \right),
\end{equation}
we see that the momentum equation in \eqref{main_sys} is 
recovered in this case by choosing $K$ so that 
$c = K \rho \kappa$ and replacing $p$ with 
\begin{equation}
    p = p_0 - \frac{c}{\kappa} \left(\frac{\kappa}{2} |\nabla \phi|^2 + F(\phi)\right).
\end{equation}

The point that we would like to emphasize with this discussion 
is that given our choice of reducing to the one-dimensional 
case by dropping $p$ from some three-dimensional system,
various reasonable coupling terms are possible. Our choice to 
work with the coupling term in \eqref{main_sys} is taken both 
because \eqref{main_sys} is a commonly used model, and because 
it provides a coupling term that is relatively convenient
for analysis (probably not entirely unrelated facts).

We close this subsection by noting that in addition to the 
models of incompressible two-phase flows described just 
above, several models for compressible two-phase flows
have been proposed. See, for example, \cite{AF08, LT98}. 

\subsection{Divergence-form coupling}
Since our study is in one spatial dimension, we are not assuming the velocity $v$ is divergence-free to avoid trivial cases, as discussed above. Clearly then, \eqref{pre-main} does not preserve mass (unlike in the divergence-free case), since in general, $\frac{d}{dt}\int_{-L}^{+L}\phi\,dx =-\int_{-L}^{+L}v\phi_x\,dx\neq0$.   Hence, one might also consider a divergence form of the advection; namely, with the advective term $v\phi_x$ replaced by $(v\phi)_x$, resulting in the following system.
\begin{equation}\label{pre-main_div_form1}
\left\{\begin{aligned}
\phi_t + (v \phi)_x &= (M (\phi)\mu_{x})_x, \\
\mu &= -\kappa \phi_{xx} + F' (\phi), \\
v_t + v v_x &= \nu v_{xx} + K \mu \phi_x.
\end{aligned}\right.
\end{equation}
However, as pointed out in Section 
\ref{energy_balance_section} below (see equation \eqref{energy_balance_div_form}), this form does not conserve energy.  One can modify the coupling in the phase equation to get an equation that conserves both mass and energy, such as in the following system.
\begin{equation}\label{pre-main_div_form2}
\left\{\begin{aligned}
\phi_t + (v \phi)_x &= (M (\phi)\mu_{x})_x, \\
\mu &= -\kappa \phi_{xx} + F' (\phi), \\
v_t + v v_x &= \nu v_{xx} - K \mu_x \phi.
\end{aligned}\right.
\end{equation}
However, this system has a significantly different coupling in the momentum equation than that of the CHNS system \eqref{main_sys}, and hence strays from our goal of finding a suitable one-dimensional analogue of \eqref{main_sys}.  Therefore, we make the choice to focus on system \eqref{pre-main},  giving up on mass conservation but retaining  energy conservation and the form of the momentum coupling. However, in Section \ref{sec-computational-results}, we briefly explore these divergence-form systems and demonstrate that their solutions appear qualitatively similar to those of \eqref{pre-main} in many respects.

\section{Coarsening Rates for Cahn--Hilliard Equations} 
\label{ch-coarsening}

In this section, we review coarsening dynamics for the uncoupled
Cahn--Hilliard equation 
\begin{equation} \label{ch}
\phi_t = 
\Big( M(\phi) (-\kappa \phi_{xx} + F' (\phi))_x \Big)_x,
\end{equation}
where $\kappa > 0$ and for this discussion we will make the 
following assumptions on $M$ and $F$. 

\medskip
\noindent
{\bf (A)} $M \in C^2 (\mathbb{R})$, and there exists a constant 
$m_0 > 0$ so that with $M (\phi) \ge m_0$ for all $\phi \in \mathbb{R}$;
$F \in C^4 (\mathbb{R})$ has a double-well form: there exist real
numbers $\alpha_1 < \alpha_2 < \alpha_3 < \alpha_4 < \alpha_5$ so
that $F$ is strictly decreasing on $(-\infty, \alpha_1)$ and
$(\alpha_3, \alpha_5)$ and strictly increasing on $(\alpha_1, \alpha_3)$
and $(\alpha_5, +\infty)$, and additionally $F$ is concave up on
$(-\infty, \alpha_2) \cup (\alpha_4, +\infty)$ and concave down on
$(\alpha_2,\alpha_4)$.  

\medskip

\begin{remark} With its origins in the work of John W. Cahn 
and John E. Hilliard, especially \cite{Cahn1961, CH58}, equation \eqref{ch}
is now well established as a foundational model of phase separation dynamics. 
Although a general review of references on \eqref{ch} is far 
beyond the scope of this discussion, we mention that our approach 
and methods are closely related to the work on initiation of 
phase separation in \cite{Grant91, Grant93}, the analyses of periodic 
solutions in \cite{H09, H11}, the analyses of kink and antikink solutions in 
\cite{BKT1999, CCO2001, H07, OW2014}, and the analyses of coarsening
rates in \cite{H11, L71, watson2003coarsening}.
\end{remark}

We observe at the outset that for each $F$ satisfying Assumptions {\bf (A)}, 
there exists a unique pair of values $\phi_1$ and $\phi_2$ 
(the {\it binodal} values) so that
\begin{equation}
F'(\phi_1) = \frac{F(\phi_2) - F(\phi_1)}{\phi_2 - \phi_1} = F'(\phi_2)
\end{equation}
and such that the line passing through
$(\phi_1, F(\phi_1))$ and $(\phi_2, F(\phi_2))$ lies entirely on or 
below $F$. (See Figure \ref{F-figure}.)
Also, we note that for any linear function 
$G(\phi) = A \phi + B$ we can replace $F(\phi)$ in \eqref{ch} with 
$H(\phi) = F(\phi) - G(\phi)$ without changing the equation in 
any way.  If we choose 
\begin{equation*}
G(\phi) := \frac{F(\phi_2) - F(\phi_1)}{\phi_2 - \phi_1} (\phi - \phi_1) + F(\phi_1),
\end{equation*}
then $H(\phi)$ has local minima at the binodal values, 
with $H(\phi_1) = H(\phi_2) = 0$, and a local maximum at the 
unique value $\phi_h$ for which 
\begin{equation*}
F'(\phi_h) = \frac{F(\phi_2) - F(\phi_1)}{\phi_2 - \phi_1}
\quad \textrm{and} \quad F''(\phi_h) < 0.
\end{equation*}
Finally, upon replacing $\phi$ with $\phi+\phi_h$ 
we can shift $H$ so that the local maximum is located at 
$\phi_h = 0$. For the remainder of our analysis we will assume 
that these transformations have been carried out, and we 
will denote the resulting function $F$. The 
standard form that we will use for numerical computations 
and some specific analytical results is 
\begin{equation} \label{quarticF}
F(\phi) = \frac{1}{4} \alpha \phi^4 - \frac{1}{2} \beta \phi^2
+ \frac{1}{4} \frac{\beta^2}{\alpha}
= \frac{\alpha}{4} \left(\phi^2 - \frac{\beta}{\alpha}\right)^2,
\end{equation} 
which clearly satisfies our general assumptions for all 
$\alpha, \beta > 0$.

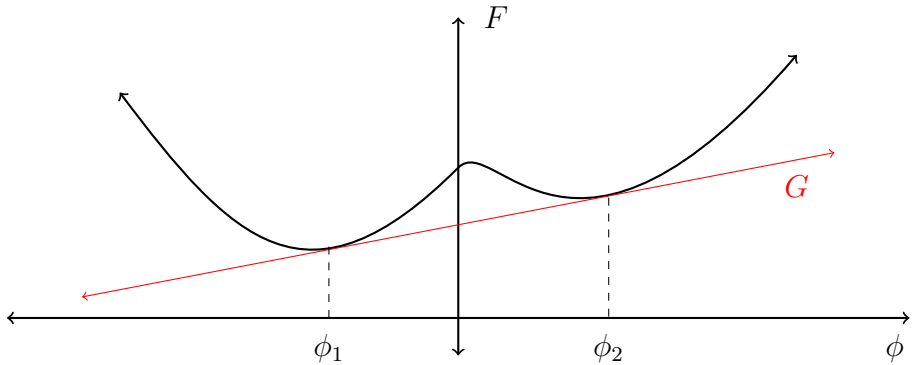
\begin{figure}[ht]
\begin{center}
\begin{tikzpicture}
\draw[thick, <->] (-6,0) -- (6,0);
\node at (5.8,-.4) {$\phi$};
\draw[thick,<->] (0,-.5) -- (0,4);
\node at (.5,4) {$F$};
\draw[thick,<-] (-4.5,3) .. controls (-3,1) and (-2,0) .. (0,2);
\draw[thick,->] (0,2) .. controls (.5,2.5) and (1.5,0) .. (4.5,3.5); 
\draw[<->,red] (-5,.28) -- (5,2.2);
\node at (4.5,1.75) {$\color{red} G$};
\draw[dashed] (-1.72,.95) -- (-1.72,0);
\node at (-1.72,-.4) {$\phi_1$};
\draw[dashed] (2,1.6) -- (2,0);
\node at (2,-.4) {$\phi_2$};
\end{tikzpicture}
\end{center}
\caption{The bulk free energy $F$ along with its supporting line.} 
\label{F-figure}
\end{figure}

\begin{remark} \label{F-remark}
For convenient reference, we summarize the properties that 
$F$ will have after the transformations described above:
$F \in C^4 (\mathbb{R})$, and there exist values 
$\phi_1 < \phi_3 < 0 < \phi_4 < \phi_2$ so that 
$F$ is strictly decreasing on $(-\infty, \phi_1)$ and
$(0, \phi_2)$ and strictly increasing on $(\phi_1, 0)$
and $(\phi_2, +\infty)$; $F$ is convex on
$(-\infty, \phi_3)$ and $(\phi_4, +\infty)$, and concave on
$(\phi_3,\phi_4)$; $F$ has local minima at $\phi_1$ and 
$\phi_2$ with $F(\phi_1) = F (\phi_2) = 0$, and $F$ has 
a local maximum at $\phi = 0$. 
\end{remark}

\subsection{The Cahn--Hilliard energy and characteristic length scale}

Equations of form \eqref{ch} are based on the energy
functional 
\begin{equation} \label{ch_energy}
E (\phi (\cdot,t)) = \int_{-L}^{+L} F(\phi) + \frac{\kappa}{2} (\phi_x)^2 dx,
\end{equation}
introduced for this context in \cite{CH58} (see also \cite{GNRW03, KO02}). 
In particular, \eqref{ch} can be expressed as a conservation law, 
\begin{equation*}
\phi_t + J_x = 0,
\end{equation*}
with flux 
\begin{equation*}
J = - M (\phi) \partial_x \frac{\delta E}{\delta \phi},
\end{equation*}
where $\frac{\delta E}{\delta \phi}$ denotes the usual variational 
gradient, and we have restricted the expressions from \cite{CH58} 
to one space dimension. A straightforward calculation shows that if $\phi$ evolves
according to \eqref{ch} then the energy $E (\phi)$ will generally 
dissipate as it approaches a global minimum value, corresponding with a 
stationary solution of \eqref{ch}. (This is clear from the energy balance
calculation carried out below in Section \ref{energy_balance_section} below, 
reduced to the uncoupled setting.) One important aspect of 
such solutions is the rate at which they coarsen from an 
initial near-homogeneous configuration to a distinctive
non-homogeneous configuration corresponding with an asymptotic
limit. In principle, we would like to gauge such coarsening by 
evolving some designated length scale $\ell (t)$, but in practice 
it is problematic to consistently assign such a scale to an arbitrary 
solution $\phi(x,t)$ of \eqref{ch}. Instead, it is often more convenient
to gauge coarsening by evolution of the
energy \eqref{ch_energy}, and one goal of the current analysis 
is to relate the evolution of $E(\phi (\cdot, t))$ to 
the evolution of a length scale $\ell (t)$ in a consistent manner. 

\subsection{Periodic solutions of the Cahn--Hilliard equation}

For many phase separation processes, including the well-studied
process of spinodal decomposition, we expect 
$\phi (x,0) = \phi_0 (x)$ to be a small random perturbation of
a homogeneous state $\phi_h = \text{constant}$.  We can understand 
the initiation of phase separation by linearizing \eqref{ch} about 
this state ($\phi = \phi_h + v$) to obtain the linear perturbation equation
\begin{equation*} 
v_t = - M(\phi_h) \kappa v_{xxxx} + M(\phi_h) F''(\phi_h) v_{xx}.
\end{equation*}
If we look for solutions of the form $v(x,t) = e^{\lambda t + i\xi x}$ 
we obtain the dispersion relation 
\begin{equation*}
\lambda (\xi) = - \kappa M(\phi_h) \xi^4 - M(\phi_h) F''(\phi_h) \xi^2,
\end{equation*}
with leading eigenvalue 
\begin{equation} \label{Leaders}
\lambda_s = \frac{M(\phi_h) F''(\phi_h)^2}{4 \kappa} > 0; \quad
\text{at} \quad \xi_s = \sqrt{-\frac{F''(\phi_h)}{2 \kappa}},
\end{equation}
and associated period
\begin{equation} \label{spinodal-period}
p_s = 2\pi \sqrt{\frac{2\kappa}{- F''(\phi_h)}},    
\end{equation}
where we emphasize that under our assumptions on $F$,
$F''(\phi_h) < 0$.
Accordingly, we expect solutions of \eqref{ch}, initialized
by small random perturbations of $\phi_h$, to rapidly move 
toward a periodic solution with period $p_s$.  Indeed, for 
the case of \eqref{ch} posed on a bounded domain in $\mathbb{R}$,
this expectation has been rigorously verified by Grant 
\cite{Grant91,Grant93}. It is natural to view these initial 
dynamics as culminating once solutions are nearly periodic 
with period $p_s$, and we refer to the dynamics up to this 
time as the {\it spinodal phase} of the process. Correspondingly,
we refer to $p_s$ as the {\it spinodal period}. 

At the end of the spinodal phase, solutions will generally be near a stationary 
periodic solution with period $p_s$, and in order to understand
the next phase of the dynamics, we consider the behavior of 
solutions initialized as perturbations of stationary 
periodic solutions. As a starting point for this, we have 
from \cite{H09, H11} that there exists a continuum 
of such periodic solutions $\bar{\phi} (x)$ to \eqref{ch} in the 
following sense: if $\phi_1$ and $\phi_2$ denote the binodal values
and $\phi_{\min}$ and $\phi_{\max}$ are any values so that 
$\phi_1 < \phi_{\min} < \phi_{\max} < \phi_2$ with additionally
\begin{equation*}
F'(\phi_{\min}) > \frac{F(\phi_{\max}) - F(\phi_{\min})}{\phi_{\max} - \phi_{\min}} > F'(\phi_{\max}),
\end{equation*}  
then there exists a periodic solution to \eqref{ch}
with minimum value $\phi_{\min}$ and maximum value $\phi_{\max}$
(see Theorem 1.5 in \cite{H09}). Moreover (again from Theorem 
1.5 in \cite{H09}), we have that if 
$F$ is as described in Remark \ref{F-remark}, and if 
additionally $F$ is an even function (such as \eqref{quarticF}), 
then for every 
amplitude $a \in (0, \phi_2)$ there exists precisely one
(up to translation) periodic stationary solution with
amplitude $a$. More precisely, this solution, denoted here
$\bar{\phi} (x; a)$,  satisfies the relation
\begin{equation*}
- \kappa \bar{\phi}_{xx} + F'(\bar{\phi}) = 0 
\implies (\bar{\phi}_x)^2 = \frac{2}{\kappa} (F(\bar{\phi}) - F(a)).
\end{equation*}
If we select the shift so that $\bar{\phi} (0; a) = 0$, we find
the integral relation
\begin{equation} \label{integralu}
\int_0^{\bar{\phi} (x;a)} \frac{dy}{\sqrt{\frac{2}{\kappa} (F(y) - F(a))}} = x,
\end{equation}  
from which we see immediately (by symmetry) that $\bar{\phi}(x;a)$ 
has period
\begin{equation} \label{a-to-p}
p(a) = 4 \int_0^{a} \frac{dy}{\sqrt{\frac{2}{\kappa} (F(y) - F(a))}}.
\end{equation}

Returning briefly to the initiation of dynamics, we can use 
\eqref{a-to-p} to compute the minimum possible period, obtained
as the limit of $p(a)$ as $a$ tends to 0. For this calculation, 
we need only observe that for $a$ small and $y \in (0, a)$, 
we can Taylor expand $F(a)$ about $y$, and subsequently Taylor 
expand $F'(y)$ and $F''(y)$ about $0$ (and note
that $F'(0) = 0$) to write 
\begin{equation*}
    F(y) - F(a) = \frac{1}{2} F''(0) (y^2 - a^2) \Big(1 + \mathbf{O} (a)\Big).
\end{equation*}
We then have 
\begin{equation*}
    p (a) = 4 \sqrt{\frac{\kappa}{- F'' (0)}} \int_0^a \frac{1}{\sqrt{a^2 - y^2}} \Big(1 + \mathbf{O} (a) \Big) dy 
    = 2 \pi \sqrt{\frac{\kappa}{- F'' (0)}} + {\mathbf O} (a).
\end{equation*}
We see that the minimum period is 
\begin{equation} \label{minimim-period}
    p_{\min} := \lim_{a \to 0^+} p(a) 
    = 2 \pi \sqrt{\frac{\kappa}{- F'' (0)}}. 
\end{equation}

In the following proposition, we verify that the period $p(a)$
increases as the amplitude $a$ increases. 

\begin{proposition} \label{period-proposition}
Assume $F$ is as described in Remark \ref{F-remark}, and also 
that $F$ is an even function. Then for all $a \in (0, \phi_2)$ 
the period $p(a)$ specified in \eqref{a-to-p} satisfies 
\begin{equation} \label{period-derivative}
p'(a) = \frac{2 \sqrt{2 \kappa}}{\sqrt{F(0)-F(a)}} 
- \sqrt{2 \kappa} \int_0^{a} \frac{F'(y) - F'(a)}{(F(y) - F(a))^{3/2}} dy. 
\end{equation}
In addition, if $F'''(\phi) > 0$ for all $\phi \in (0, \phi_2)$ 
then $p'(a) > 0$ for all $a \in (0, \phi_2)$. 
\end{proposition}

\begin{proof}
In order to abbreviate notation, we will write 
\begin{equation*}
    p(a) = 2 \sqrt{2 \kappa} \tilde{p} (a),
    \quad \tilde{p} (a) = \int_0^a \frac{dy}{\sqrt{F(y) - F(a)}}. 
\end{equation*}
By setting $z = y/a$, we can express $\tilde{p} (a)$ as 
\begin{equation*}
    \tilde{p} (a) = \int_0^1 \frac{a}{\sqrt{F(az) - F(a)}} dz,
\end{equation*}
for which we can justify differentiating through the integral. This
gets us to the relation 
\begin{equation} \label{p-tilde-equation}
    \tilde{p}'(a)
    = \int_0^1 \frac{(F(az) - \frac{az}{2} F' (az)) - (F(a) - \frac{a}{2} F' (a))}{(F(az) - F(a))^{3/2}} dz.
\end{equation}
In order to establish positivity of $\tilde{p}'(a)$, we will set 
\begin{equation*}
G(y) := F(y) - \frac{y}{2} F' (y),
\end{equation*}
and show that $G(az) - G(a)$ is positive for all $z \in (0, 1)$.  
First, $G' (y) = \frac{1}{2} F'(y) - \frac{y}{2}F''(y)$, so in 
particular $G' (0) = \frac{1}{2} F'(0) = 0$. 
Next, by assumption, $G''(y) = - \frac{y}{2} F'''(y) < 0$
for all $y \in (0, a)$, so $G'(y)$ is a decreasing function,
and we must have $G' (y) < 0$ for all $y \in (0, a)$. It 
follows that $G (y)$ is a decreasing function, so $G (a) < G (az)$
for all $z \in (0, 1)$, giving the second claim.

For the first claim, it is convenient to return to $y = az$ and 
rearrange \eqref{p-tilde-equation} as 
\begin{equation} \label{p-tilde-equation2}
    \tilde{p}'(a)
    = \int_0^a \frac{\frac{1}{2} (F' (a) - F' (y)) + \frac{1}{a} (F(y) - F(a))}{(F(y) - F(a))^{3/2}} dy
    + \frac{1}{2} \lim_{\tau \to a^-} \int_0^{\tau} \frac{F'(y) - \frac{y}{a} F'(y)}{(F(y) - F(a))^{3/2}} dy,
\end{equation}
where the final integral has been expressed as a limit to justify 
integrating by parts. Integrating by parts, we find 
\begin{equation*}
    \begin{aligned}
\lim_{\tau \to a^-} \int_0^{\tau} \frac{F'(y) - \frac{y}{a} F'(y)}{(F(y) - F(a))^{3/2}} dy
&= - \int_0^a \frac{\frac{2}{a} (F(y) - F(a))}{(F(y) - F(a))^{3/2}} dy
+ \lim_{\tau \to a^-} \frac{-2 (1 - \frac{y}{a})}{\sqrt{F (y) - F(a)}} \Big|_0^\tau \\
&= - \int_0^a \frac{\frac{2}{a} (F(y) - F(a))}{(F(y) - F(a))^{3/2}} dy
+ \frac{2}{\sqrt{F (0) - F(a)}}. 
    \end{aligned}
\end{equation*}
Recalling the factor of $1/2$, we see that the first summand on the right-hand side of 
this last expression cancels with the second part of the first integral on the right-hand side of 
\eqref{p-tilde-equation2}, leaving 
\begin{equation*}
    \tilde{p}' (a)
    = \frac{1}{\sqrt{F (0) - F(a)}} 
    + \int_0^a \frac{\frac{1}{2} (F' (a) - F' (y))}{(F(y) - F(a))^{3/2}} dy. 
\end{equation*}
Recalling the specification $p(a) = 2 \sqrt{2 \kappa} \tilde{p} (a)$, we see 
that the proof is complete. 
\end{proof}

Since $p$ is monotonically increasing as a function of amplitude, we can 
uniquely specify the spinodal amplitude $a_s$ so that 
$p(a_s) = p_s$. Precisely, we obtain the relation 
\begin{equation*}
\frac{2 \pi \sqrt{2 \kappa}}{\sqrt{- F''(0)}}
= 4 \sqrt{\kappa} \int_0^{a_s} \frac{dy}{\sqrt{2 (F(y) - F(a_s))}},
\end{equation*}
and since $\sqrt{\kappa}$ can be divided out of both sides, 
we see that $a_s$ does not depend on $\kappa$.

Using \eqref{integralu}, we can identify the unique (up to shift, 
selected by $\bar{\phi} (0;a_s) = 0$)
periodic solution of \eqref{ch} with amplitude $a_s$, namely
$\bar{\phi} (x; a_s)$, and subsequently we define the spinodal 
energy $E_s$ as the energy associated with this periodic solution,
\begin{equation} \label{spinodal-energy}
    E_s := E(\bar{\phi} (\cdot; a_s)).
\end{equation} 
For the long-time models discussed below, we will typically 
think of initiating the dynamics once the energy has 
reduced to the spinodal value. 

Although the energy function $E (\phi (\cdot, t))$ can 
achieve any value attainable via functions $\phi (x, t)$
in its domain, the dynamics we have in mind involve 
solutions with energies less than the energy achieved 
by the homogeneous configuration $\phi_{0} \equiv 0$.
For a given interval $[-L, +L]$, this is easily computed 
to be 
\begin{equation} \label{maximum-energy}
    E_{\max} := E (0) = \int_{-L}^{+L} F(0) dx
    = 2LF(0). 
\end{equation}
As time increases, energies will decrease as solutions 
approach either a kink or anti-kink solution, respectively
$K(x)$ or $K(-x)$, where 
\begin{equation*}
    - \kappa K'' + F' (K) = 0,
    \quad \forall\, x \in \mathbb{R},
\end{equation*}
with also 
\begin{equation*}
    \lim_{x \to -\infty} K(x) = \phi_1,
    \quad    \lim_{x \to +\infty} K(x) = \phi_2.
\end{equation*}
(See \cite{H09} for existence of such solutions, and 
\cite{H07} for asymptotic stability; we note that according 
to convention (see, e.g., \cite{EB1996}), 
a {\it kink} solution is monotonically 
increasing, while an {\it anti-kink} solution is 
monotonically decreasing.) These solutions
provide us with a lower bound on the 
energy 
\begin{equation}
    E_{\min} := \int_{-L}^{+L}
    F(K (x)) + \frac{\kappa}{2} |K'(x)|^2 dx,
\end{equation}
which will be computed explicitly in Proposition 
\ref{specific-F-proposition} just below.

For specific implementations of our approach, we will use 
the family of bulk free energy densities specified in 
\eqref{quarticF}. Many of the preceding relations can 
be made explicit in this case, and for convenient reference
we summarize these in the following proposition. 

\begin{proposition} \label{specific-F-proposition}
    For \eqref{ch}, let $M$ satisfy the assumptions in 
    {\bf (A)}, and let $F$ be as in \eqref{quarticF}.
    Then the following hold:

    \medskip
    \noindent
    (i) {\bf Minimum period}. The minimum period computed in \eqref{minimim-period} 
    is $p_{\min} = 2 \pi \sqrt{\kappa/\beta}$, and the 
    associated energy computed in \eqref{maximum-energy}
    is $E_{\max} = L\beta^2/(2 \alpha)$. 

    \medskip
    \noindent
    (ii) {\bf Spinodal period}. The spinodal period computed 
    in \eqref{spinodal-period} is $p_s = 2 \pi \sqrt{2 \kappa/\beta}$.

    \medskip
    \noindent
    (iii) {\bf Periodic solutions}. For each $a \in (0, \sqrt{\beta/\alpha})$,
    the periodic solution $\bar{\phi} (x; a)$ specified in 
    \eqref{integralu} can be expressed as a Jacobi elliptic function 
    \begin{equation} \label{jacobiellipticu}
    \bar{\phi} (x; a) = a \operatorname{sn} \left(\sqrt{\frac{-2(F(a) - F(0))}{\kappa}} \frac{x}{a}, k\right),
    \end{equation}
    where 
    \begin{equation} \label{kdefined}
    k = \sqrt{- \frac{\alpha a^4}{4 (F(a) - F(0))}}.
    \end{equation}
  
    \medskip
    \noindent
    (iv) {\bf Kink solutions and the minimum energy}. The binodal values for $F$ are $\pm \sqrt{\beta/\alpha}$, 
    and there exists a unique (up to translation) kink solution
    of \eqref{ch}, 
    \begin{equation} \label{kink-solution}
        K(x) = \sqrt{\frac{\beta}{\alpha}} \tanh \Big(\sqrt{\frac{\beta}{2 \kappa}} x \Big).
    \end{equation}
    The energy associated with this kink solution on $[-L, +L]$ 
    is 
    \begin{equation}\label{E_kink}
        E_{\min} := E(K(x)) = \int_{-L}^{+L} F (K (x)) + \frac{\kappa}{2} K'(x)^2 dx
        = \sqrt{2 \kappa \alpha} K (L) \Big(\frac{\beta}{\alpha} - \frac{K (L)^2}{3} \Big),   
    \end{equation}
    and it follows that 
    \begin{equation*}
     E_{\min}^{\infty} := \lim_{L \to \infty} E_{\min}
     = \frac{2}{3} \frac{\beta^2}{\alpha} \sqrt{\frac{2 \kappa}{\beta}}.   
    \end{equation*}
\end{proposition}

\begin{proof}
    Items (i) and (ii) are clear from evaluation of the indicated formulas at 
    the specific family of bulk free energy densities \eqref{quarticF}. For Item (iii),
    we first observe that for $F$ as specified in \eqref{quarticF}, we can 
    write 
    \begin{equation*}
        F(y) - F(a)
        = \left(\frac{\alpha}{4} y^2 + \frac{F(a) - F(0)}{a^2}\right) (y^2 - a^2).
    \end{equation*}
    In this case, \eqref{integralu} becomes 
    \begin{equation*} 
    \int_0^{\bar{\phi} (x;a)} 
    \frac{dy}{\sqrt{\frac{2}{\kappa} (\frac{\alpha}{4} y^2 + \frac{F(a) - F(0)}{a^2}) (y^2 - a^2)}} = x.
    \end{equation*}  
    In order to express this as a Jacobi elliptic function, we first multiply both 
    sides by the factor $(1/a) \sqrt{- 2 (F(a) - F(0))/\kappa}$ to obtain 
    \begin{equation*} 
    \int_0^{\bar{\phi} (x;a)} 
    \frac{\frac{1}{a} \sqrt{\frac{-2 (F(a) - F(0))}{\kappa}}}
    {\sqrt{\frac{2}{\kappa} (\frac{\alpha}{4} y^2 + \frac{F(a) - F(0)}{a^2}) (y^2 - a^2)}} dy 
    = \frac{x}{a} \sqrt{\frac{-2 (F(a) - F(0))}{\kappa}},
    \end{equation*}  
    which we can rearrange to 
    \begin{equation*} 
    \int_0^{\bar{\phi} (x;a)} 
    \frac{dy}{\sqrt{(a^2 - y^2) (1 + \frac{\alpha a^2}{4(F(a) - F(0))}y^2)}}
    = \frac{x}{a} \sqrt{\frac{-2 (F(a) - F(0))}{\kappa}}.
    \end{equation*}  

    If we now set $y = a \sin \theta$ and specify $k$ as 
    in \eqref{kdefined}, we obtain the relation 
    \begin{equation*}
        \int_0^{\bar{\theta}(x;a)} \frac{d \theta}{\sqrt{1 - k^2 \sin^2 \theta}}
        = \frac{x}{a} \sqrt{\frac{- 2 (F(a) - F(0))}{\kappa}},
    \end{equation*}
    where $\sin \bar{\theta}  = \bar{\phi}/a$. By definition of the Jacobi 
    elliptic function $\operatorname{sn}$ (see Remark \ref{periodic-solution-details}
    below), 
    \begin{equation*}
    \operatorname{sn} \left(\sqrt{\frac{- 2 (F(a) - F(0))}{\kappa}} \frac{x}{a}, k\right)
    = \sin \bar{\theta} (x; a) = \frac{\bar{\phi} (x; a)}{a},
    \end{equation*}
    from which Item (iii) follows immediately.  

    Turning to Item (iv), it is straightforward to verify directly 
    that \eqref{kink-solution} is indeed a stationary solution 
    to \eqref{ch} with $F$ as in \eqref{quarticF}. For the 
    energy claim, we need to evaluate $E$ in \eqref{ch_energy} 
    at $K(x)$. For this, we first observe that $K (x)$ satisfies
    the equation 
    \begin{equation*}
        K' = \sqrt{\frac{2}{\kappa} F (K(x))},
    \end{equation*}
    so that 
    \begin{equation*}
        E(K(x)) = \int_{-L}^{+L} F (K (x)) + \frac{\kappa}{2} \cdot \frac{2}{\kappa} F (K (x)) dx
        = 2 \int_{-L}^{+L} F(K(x)) dx. 
    \end{equation*}
    We make the change of variables $y = K(x)$, so that 
    \begin{equation*}
        dy = K'(x) dx = \sqrt{\frac{2}{\kappa} F(K(x))} dx,
    \end{equation*}
    giving 
    \begin{equation*}
    \begin{aligned}
        E(K(x))
        &= \sqrt{2 \kappa} \int_{K (-L)}^{K (+L)} \sqrt{F(y)} dy
        = \sqrt{\frac{\kappa \alpha}{2}} \int_{K (-L)}^{K (+L)} \frac{\beta}{\alpha} - y^2 dy \\
        &= \sqrt{2 \kappa \alpha} K (L) \Big(\frac{\beta}{\alpha} - \frac{K (L)^2}{3} \Big),
    \end{aligned}
    \end{equation*}
    where in obtaining the final equality, we integrated directly 
    and observed the relation $K (-L) = -K(L)$. Finally, the 
    statement about $E_{\min}^{\infty}$ is clear from the limit 
    \begin{equation*}
        \lim_{L \to \infty} K(L) = \sqrt{\frac{\beta}{\alpha}}. 
    \end{equation*}
\end{proof}

\begin{remark} \label{periodic-solution-details}
For the periodic solution \eqref{jacobiellipticu}, 
$\text{sn} (y;k)$ denotes the Jacobi elliptic function, 
defined so that 
\begin{equation*}
\text{sn} (y;k) = \sin \theta; \quad \text{where} \quad
y = \int_0^{\theta} \frac{d \zeta}{\sqrt{1 - k^2 \sin^2 \zeta}}.
\end{equation*}
For implementations, we evaluate $\operatorname{sn} (y; k)$ with 
the MATLAB function {\it ellipj.m}.
The period in this case is 
\begin{equation} \label{jacobiperiod}
p (a) = \frac{4 a \sqrt{\kappa}}{\sqrt{-2(F(a) - F(0))}} \mathcal{K} (k),
\end{equation}
where $\mathcal{K}$ denotes the complete elliptic integral
\begin{equation*}
\mathcal{K}(k) = \int_0^1 \frac{ds}{\sqrt{(1-s^2)(1-k^2 s^2)}}.
\end{equation*}
For implementations, we evaluate $\mathcal{K} (k)$ with 
the MATLAB function {\it ellipke.m}. 
\end{remark}

\begin{remark}
    \textcolor{black}{
    We observe for consistency that the expression in (iii) for $\bar{\phi} (x; a)$, $a\in(0,\sqrt{\frac{\beta}{\alpha}})$, 
    matches the results when $a\to 0$ and $a\to \sqrt{\frac{\beta}{\alpha}}$. To see this, we first take $a\to 0$ and observe that
    \[
      \lim_{a\to0}\sqrt{\frac{-2(F(a)-F(0))}{a^2\kappa}} x = \sqrt{\frac\beta\kappa} x \quad \text{ and } \quad \lim_{a\to0} k = \lim_{a\to0} \sqrt{\frac{-\alpha a^4}{4(F(a)-F(0))}} = 0.
    \]
    As $\text{sn}(u;0) = \sin u$, one can conclude that 
    \[
     \lim_{a\to0} \bar\phi(x;a) = \lim_{a\to0} a \sin (\sqrt{\frac\beta\kappa} x) =0.
    \]
    Next, taking $a\to \sqrt{\frac{\beta}{\alpha}}$ we find that
    \[
      \lim_{a\to \sqrt{\frac{\beta}{\alpha}}}\sqrt{\frac{-2(F(a)-F(0))}{a^2\kappa}} x = \sqrt{\frac{\beta}{2\kappa}} x \quad \text{and} \quad \lim_{a\to \sqrt{\frac{\beta}{\alpha}}} k = \lim_{a\to \sqrt{\frac{\beta}{\alpha}}} \sqrt{\frac{-\alpha a^4}{4(F(a)-F(0))}} = 1.
    \]
    As $\text{sn}(u;1) = \tanh u$, we obtain
    \[
      \lim_{a\to \sqrt{\frac{\beta}{\alpha}}} \bar\phi(x;a) = \sqrt{\frac\beta\alpha} \text{sn}(\sqrt{\frac{\beta}{2\kappa}} x,1) = \sqrt{\frac\beta\alpha} \tanh(\sqrt{\frac{\beta}{2\kappa}} x) = K(x).
    \]
    which gives the kink solution.
    }
\end{remark}

As a working example for quantitative analysis, we will take 
$M (\phi) \equiv 1$ and use $F$ as in \eqref{quarticF} with 
$\alpha = \beta = 1$, leaving only $\kappa$ to vary. As a 
baseline case, we will take $\kappa = 0.001$. Using these values,  
periodic solutions
with three different periods are depicted in Figure 
\ref{pwaves-figure}. 

\begin{figure}[ht] 
\begin{center}
\includegraphics[width=12cm,height=8.2cm]{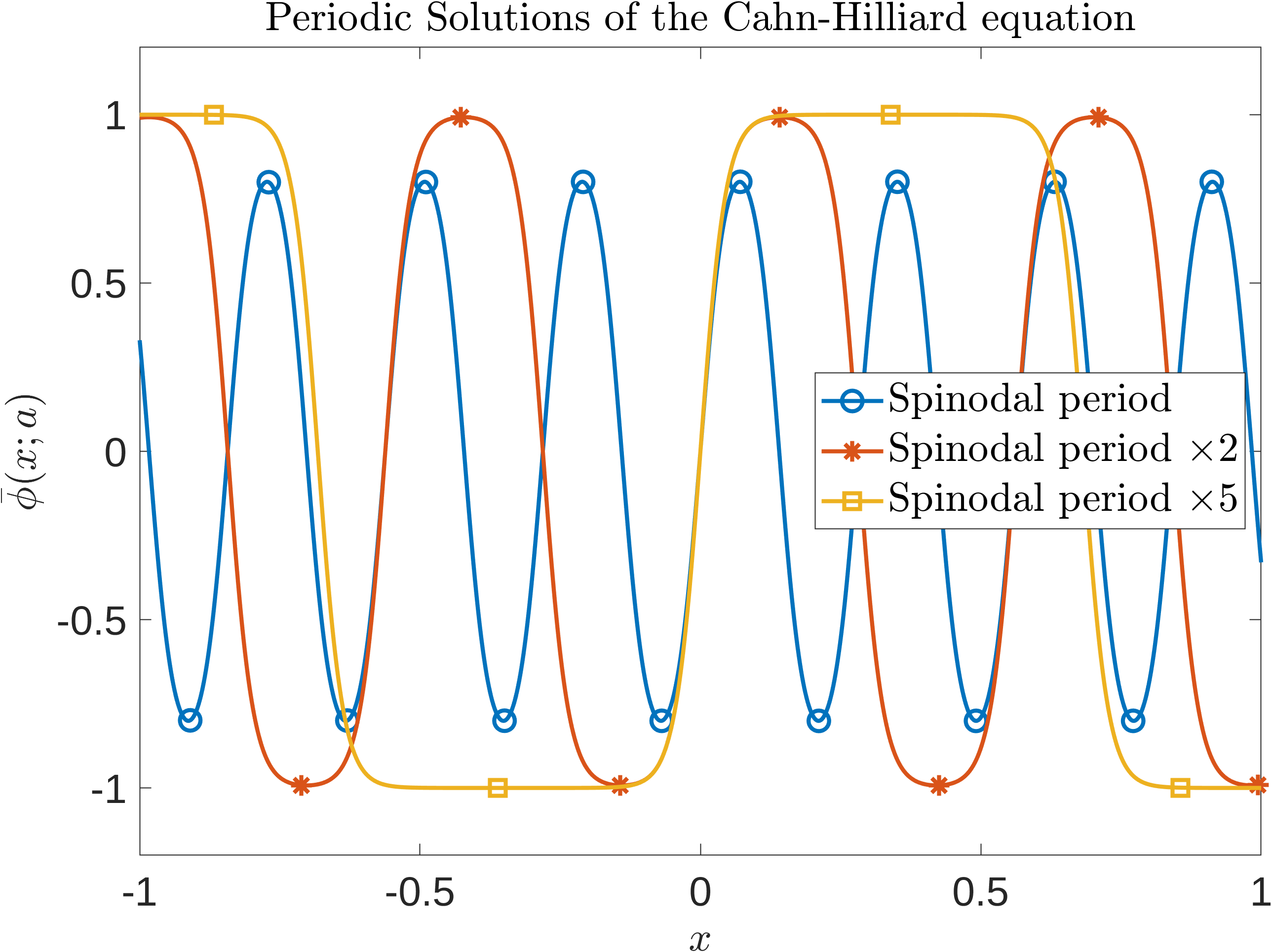}
\end{center}
\caption{Periodic solutions for the Cahn--Hilliard equation. \label{pwaves-figure}}
\end{figure}

\subsection{A Measure of Coarsening} 
\label{MeasureofCoursening}

We return now to our goal set out previously of relating 
each energy $e \in (E_{\min}, E_{\max})$ with an associated 
length scale $\ell$. Since solutions to \eqref{ch} tend to 
be near periodic solutions, our strategy will be to identify 
with a given energy $e$ the period $p$ for which 
\begin{equation*}
    E (\bar{\phi} (\cdot; a (p))) = e. 
\end{equation*}
One caveat for this approach is that the function 
\begin{equation} \label{p-to-e}
    \mathcal{E} (p) 
    := E (\bar{\phi} (\cdot; a (p)))
    = \int_{-L}^{+L} F (\bar{\phi} (x; a(p))) 
    + \frac{\kappa}{2} \bar{\phi}_x (x; a(p))^2 dx.
\end{equation}
is not invertible, a point we clarify in the next proposition. 

\begin{proposition} \label{EofPNotInvertible}
For \eqref{ch}, let $M$ satisfy the assumptions in 
{\bf (A)}, and let $F$ be as in \eqref{quarticF}.
Then the function $\mathcal{E} (p)$
defined in \eqref{p-to-e} is not invertible on $[p_{\min}, \infty)$. 
\end{proposition}

\begin{proof}
    First, since the period $p$ depends monotonically on the 
    amplitude $a$, we can work with the map 
\begin{equation} \label{energy-period}
    \mathbb{E} (a) := E (\bar{\phi} (\cdot; a))
    = \int_{-L}^{+L} F (\bar{\phi} (x; a)) + \frac{\kappa}{2} \bar{\phi}_x (x; a)^2 dx.
\end{equation}
We will prove the proposition by showing that $\mathbb{E}' (a)$ does not 
have a fixed sign. Upon differentiating $\mathbb{E}$ in $a$, we 
obtain 
\begin{equation} \label{energy-period-derivative}
\begin{aligned}
    \frac{d\mathbb{E}}{da} (a) 
    &= \int_{-L}^{+L} F' (\bar{\phi} (x; a)) \bar{\phi}_a (x; a) + \kappa \bar{\phi}_x (x; a) \bar{\phi}_{x a} (x; a) dx \\
    &\overset{\textrm{parts}}{=} \int_{-L}^{+L} F' (\bar{\phi} (x; a)) \bar{\phi}_a (x; a) - \kappa \bar{\phi}_{xx} (x; a) \bar{\phi}_{a} (x; a) dx \\
    &\quad \quad + \kappa \bar{\phi}_x (L; a) \bar{\phi}_{a} (L; a) - \kappa \bar{\phi}_x (-L; a) \bar{\phi}_{a} (-L; a) \\ 
    &= \kappa \bar{\phi}_x (L; a) \bar{\phi}_{a} (L; a) - \kappa \bar{\phi}_x (-L; a) \bar{\phi}_{a} (-L; a),
\end{aligned}    
\end{equation}
where the final equality follows because $- \kappa \bar{\phi}_{xx} (x; a) + F' (\bar{\phi} (x; a)) = 0$ for a periodic 
wave $\bar{\phi} (x; a)$. 

In order to understand the derivative $\bar{\phi}_a (x; a)$, we recall the explicit form for 
$\bar{\phi} (x; a)$, 
\begin{equation} \label{periodic-wave}
    \bar{\phi} (x; a) = a \sn (h(a)x; k), 
    \quad h(a) = \frac{1}{a} \sqrt{\frac{-2 (F(a) - F(0))}{\kappa}}.
\end{equation}
With $F$ as specified in \eqref{quarticF}, we find that 
\begin{equation*}
    F(a) - F(0) = a^2 (\frac{\alpha}{4} a^2 - \frac{\beta}{2}),
\end{equation*}
so 
\begin{equation*}
    h(a) = \sqrt{\frac{\alpha}{2\kappa}} \sqrt{\frac{2\beta}{\alpha} - a^2},
    \quad h'(a) = -\frac{\sqrt{\frac{\alpha}{2\kappa}}a}{\sqrt{\frac{2\beta}{\alpha} - a^2}}.
\end{equation*}
Now, 
\begin{equation*}
    \bar{\phi}_a (x; a) = \sn (h(a)x; k) + a \sn' (h(a)x; k) h'(a) x,
\end{equation*}
and likewise 
\begin{equation*}
    \bar{\phi}_x (x; a) = a \sn' (h(a)x; k) h(a),
\end{equation*}
from which we can write 
\begin{equation*}
    \bar{\phi}_a (x; a) = \frac{1}{a} \bar{\phi} (x; a) 
    + \frac{h'(a)x}{h(a)} \bar{\phi}_x (x; a)
    = \frac{1}{a} \bar{\phi} (x; a) - \frac{ax}{\frac{2\beta}{\alpha} - a^2} \bar{\phi}_x (x;a).
\end{equation*}

From these considerations, we see that 
\begin{equation*}
    \bar{\phi}_x (x; a) \bar{\phi}_a (x; a)
    = \frac{1}{a} \bar{\phi} (x; a) \bar{\phi}_x (x; a) 
    - \frac{ax}{\frac{2\beta}{\alpha} - a^2} \bar{\phi}_x (x;a)^2.
\end{equation*}
Since $\bar{\phi} (x; a)$ is an odd function (because $\sn (x;k)$ is odd), we 
see that the product $\bar{\phi}_x (x; a) \bar{\phi}_a (x; a)$ is an odd function, 
and from this we can conclude from \eqref{energy-period-derivative} that 
\begin{align} \label{energy-amplitude-derivative2}
    \frac{d\mathbb{E}}{da} (a) 
    & = 2 \kappa \bar{\phi}_x (L; a) \bar{\phi}_{a} (L; a)
    = \frac{2 \kappa}{a} \bar{\phi} (L; a) \bar{\phi}_x (L; a) 
    - \frac{2 \kappa aL}{\frac{2\beta}{\alpha} - a^2} \bar{\phi}_x (L;a)^2 \nonumber\\
    &= \frac{2 \kappa}{a} \bar{\phi}_x (L; a) \Big( \bar{\phi} (L; a) 
    - \frac{a^2L}{\frac{2\beta}{\alpha} - a^2} \bar{\phi}_x (L;a)\Big).
\end{align}
Recalling that our goal is to understand the sign of $\mathbb{E}' (a)$,
as $a$ varies, we begin by considering a configuration in which $\bar{\phi} (L; a) = 0$.
In this case, 
\begin{align*}
     \frac{d\mathbb{E}}{da} (a) 
     &= - \frac{2 \kappa aL}{\frac{2\beta}{\alpha} - a^2} \bar{\phi}_x (L;a)^2
     = - \frac{2 \kappa aL}{\frac{2\beta}{\alpha} - a^2} \frac{2}{\kappa}(F (0) - F(a)) \\
     &= - \frac{4 aL}{\frac{2\beta}{\alpha} - a^2} \frac{\alpha a^2}{4} \left(\frac{2\beta}{\alpha} - a^2\right)
     = - a^3 \alpha L < 0.
\end{align*}

With this configuration, we have either $\bar{\phi}_x (L; a) < 0$ or $\bar{\phi}_x (L; a) > 0$,
and for specificity we will focus on the former case. If we now think about increasing $a$ a small
amount, we know from the monotonic dependence of $p$ on $a$ that the period $p$ will increase 
a small amount and we will have $\bar{\phi} (L; a) > 0$. As the amplitude $a$ continues to 
increase the value of $\bar{\phi} (L; a)$ will increase, and correspondingly the value of 
$|\bar{\phi}_x (L; a)|$ will decrease. Nonetheless, for $\bar{\phi}_x (L; a) < 0$ and 
$\bar{\phi}(L;a) > 0$ we will continue to have $\mathbb{E}' (a) < 0$ until $\bar{\phi} (L;a)$
achieves its maximum value $\bar{\phi} (L;a) = a$, at which point $\mathbb{E}' (a) = 0$ (because 
$\bar{\phi}_x (L;a) = 0$).
As $a$ increases further, we will have $\bar{\phi}_x (L;a) > 0$, with $|\bar{\phi}_x (L; a)|$ 
sufficiently small so that 
\begin{equation*}
    \bar{\phi} (L; a) 
    - \frac{a^2L}{\frac{2\beta}{\alpha} - a^2} \bar{\phi}_x (L;a) > 0.
\end{equation*}
Since we are now in the setting with $\bar{\phi}_x (L; a) > 0$, we see from 
\eqref{energy-amplitude-derivative2} that we will have $\mathbb{E}' (a) > 0$. In this way, 
we conclude that $\mathbb{E} (a)$ is not monotonically decreasing as $a$ increases. 
Due to the monotonic dependence of $a$ on $p$, we can additionally conclude that 
the energy map $\mathcal{E} (p)$ is not monotonic in $p$, and so is not invertible 
for all values of $p$. 
\end{proof}

According to Proposition \ref{EofPNotInvertible}, we cannot invert 
$\mathcal{E} (p)$, but we will see below that the pseudoinverse
of $\mathcal{E} (p)$ nonetheless provides an effective measure of 
coarsening. 

\begin{definition} \label{pseudoinverse}
    Given a fixed energy $e \in (E_{\min}, E_{\max}]$, we define the associated
    period $p$ to be the pseudoinverse 
    \begin{equation} \label{pseudoinverse_eq}
        p = \inf \{\tilde p>0: \mathcal{E} (\tilde p) \le e\}.
    \end{equation}
\end{definition}

To better understand why the pseudoinverse works well in this case, 
we observe that $\mathcal{E} (p)$ is a decreasing function of $p$ 
except on some small intervals on which $\mathcal{E}' (p)$ is close
to 0. In Figure \ref{energy-period-figure}, we plot $\mathcal{E} (p)$
in the case $\alpha = \beta = 1$, $\kappa = 0.001$, and $L = 1$. We see
from this plot that as $p$ increases, sharp gradients with $\mathcal{E}' (p) < 0$
alternate with plateaus where $\mathcal{E}' (p) \cong 0$. The sharp 
gradients occur near values of $p$ for which $\bar{\phi} (L;a (p)) = 0$,
as in the proof of Proposition \ref{EofPNotInvertible}. Proposition
\ref{EofPPlateaus} addresses the nature of the plateaus. 

\begin{figure}[ht] 
\begin{center}
\includegraphics[width=12cm,height=8.2cm]{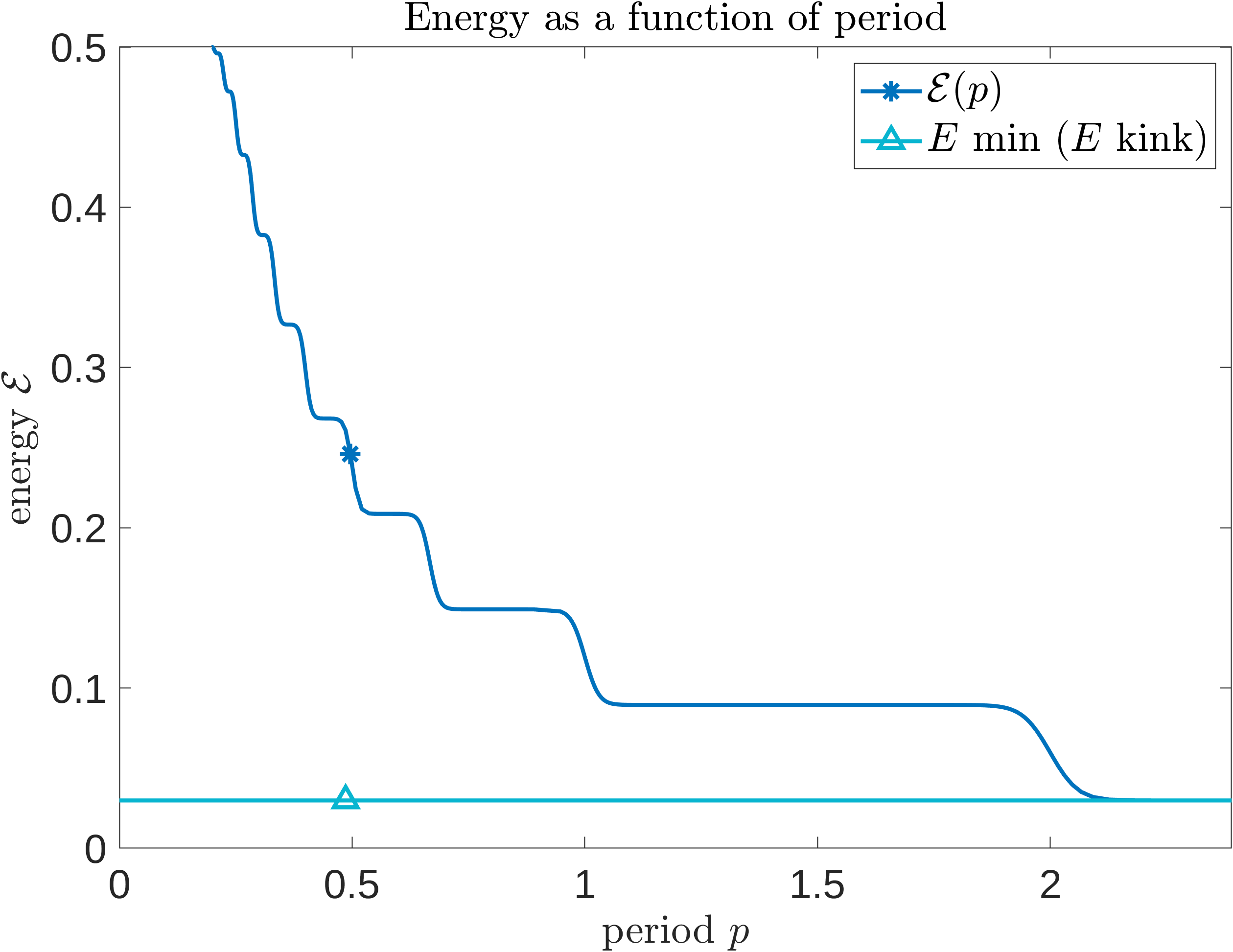}
\end{center}
\caption{Plot of $\mathcal{E} (p)$, computed with $\alpha = \beta = 1$, $\kappa = 0.001$, and $L = 1$. 
\label{energy-period-figure}}
\end{figure}

\begin{proposition} \label{EofPPlateaus}
For \eqref{ch}, let $M$ satisfy the assumptions in {\bf (A)}, and let $F$ be as in \eqref{quarticF}.
Then for any $p > p_{\min}$ for which $\mathcal{E}' (p) > 0$, we have the inequality 
\begin{equation*}
    \mathcal{E}' (p) 
    \le \frac{a^3 \sqrt{\alpha \kappa} (1+\delta)^{3/2} \delta^3 (\delta - 1)}{2L},
\end{equation*}
where $\delta = \sqrt{2\beta/(\alpha a^2)-1}$ and $a$ is the unique amplitude corresponding
with $p$. 
\end{proposition}

\begin{proof}
Recalling \eqref{energy-amplitude-derivative2} from the proof of Proposition 
\ref{EofPNotInvertible}, we see that if $\bar{\phi}_x (L;a) > 0$ then 
we can only have $\mathbb{E}' (a) > 0$ if 
\begin{equation} \label{barphiInequality}
    \frac{a^2L}{\frac{2\beta}{\alpha} - a^2} \bar{\phi}_x (L;a) < \bar{\phi} (L; a), 
\end{equation}
so that 
\begin{equation*}
    \bar{\phi}_x (L;a)
    \le \frac{\frac{2\beta}{\alpha} - a^2}{a^2L} \bar{\phi} (L; a). 
\end{equation*}
If we substitute this inequality into \eqref{energy-amplitude-derivative2}, we
see that 
\begin{equation} \label{EprimeEstimate}
    \frac{d \mathbb{E}}{da} (a) 
    \le \frac{2 \kappa}{a} \bar{\phi}_x (L; a) \bar{\phi} (L; a)
    < \frac{2 \kappa}{a} \frac{\frac{2\beta}{\alpha} - a^2}{a^2L} \bar{\phi} (L; a)^2
    \le \frac{2 \kappa}{a} \frac{\frac{2\beta}{\alpha} - a^2}{L},
\end{equation}
where in obtaining the final inequality we have observed that $|\bar{\phi} (L;a)| \le a$.

In order to translate this into information about the size of $\mathcal{E}' (p)$, we 
observe the relation 
\begin{equation} \label{ChainRule}
    \mathcal{E}' (p) = \mathbb{E}'(a)/p'(a).
\end{equation}
We have seen in the proof of Proposition \ref{period-proposition} that 
\begin{equation} \label{pprime}
    p' (a) = 2\sqrt{2\kappa} \int_0^1 \frac{G(az) - G(a)}{(F(az) - F(a))^{3/2}} \,dz,
\end{equation}
where 
\begin{equation*}
    G (y) = F (y) - \frac{y}{2} F'(y),
\end{equation*}
and moreover noted in the same proof that for all $a \in (0, \sqrt{\beta/\alpha})$ we 
have $p' (a) > 0$ (under the assumptions of that proposition). For our choice 
of $F$ \eqref{quarticF}, we have the relations 
\begin{equation*}
    \begin{aligned}
        F(az) - F(a) &= \frac{\alpha a^4}{4} (1 - z^2) (\delta^2 - z^2), \\
        G(az) - G(a) &= \frac{\alpha a^4}{4} (1-z^2) (1+z^2),
    \end{aligned}
\end{equation*}
where for notational brevity we have set 
\begin{equation*}
    \delta^2 := \frac{2 \beta}{\alpha a^2} - 1,
\end{equation*}
which is greater than 1 for all $a \in (0, \sqrt{\frac{\beta}{\alpha}})$. Upon 
substituting these relations into \eqref{pprime}, we arrive at the 
relation 
\begin{equation*}
    \begin{aligned}
    p' (a) &= \frac{4\sqrt{2 \kappa}}{a^2 \sqrt{\alpha}} 
    \int_0^1 \frac{1+z^2}{(1-z^2)^{1/2} (\delta^2 - z^2)^{3/2}} \,dz \\
    &=  \frac{4\sqrt{2 \kappa}}{a^2 \sqrt{\alpha}} 
     \int_0^1 \frac{1+z^2}{[(1-z)(1+z)]^{1/2} [(\delta - z)(\delta + z)]^{3/2}} \,dz. 
    \end{aligned}
\end{equation*}
We are primarily interested in a lower bound on this quantity, 
\begin{equation*}
    \begin{aligned}
    p'(a) &\ge \frac{4\sqrt{2 \kappa}}{a^2 \sqrt{\alpha}} 
    \frac{1}{\sqrt{2} (1+\delta)^{3/2}}
    \int_0^1 \frac{1}{(1-z)^{1/2} (\delta - z)^{3/2}} \,dz \\
    &\ge \frac{4\sqrt{\kappa}}{a^2 \sqrt{\alpha} (1+\delta)^{3/2}} 
    \int_0^1 \frac{1}{(\delta - z)^{2}} \,dz, 
    \end{aligned}
\end{equation*}
where the second inequality follows immediately from the observation that $\delta > 1$. 
Integrating directly, we now arrive at the inequality 
\begin{equation} \label{pprimeEstimate}
    p' (a) \ge \frac{4\sqrt{\kappa}}{a^2 \sqrt{\alpha} (1+\delta)^{3/2}} 
    \Big(\frac{1}{\delta (\delta - 1)} \Big).
\end{equation}
The stated estimate on $\mathcal{E}' (p)$ now follows from \eqref{ChainRule}, 
\eqref{EprimeEstimate}, and \eqref{pprimeEstimate}. 

Finally, we recall that these calculations have been carried out 
under the assumption $\bar{\phi}_x (L;a) > 0$, and note that the case 
with $\bar{\phi}_x (L;a) < 0$ is similar.
\end{proof}

\begin{remark}
As expected, we see from Proposition \ref{EofPPlateaus} that as the amplitude $a$ tends toward the 
maximum amplitude $\sqrt{\beta/\alpha}$, $p'(a)$ tends toward $\infty$ (since $\delta$ 
tends toward 1). I.e., for late-stage coarsening, $\mathcal{E}' (p)$ is much smaller 
than $\mathbb{E}' (a)$. For intermediate values of $a$, $p' (a)$ is proportional to 
$\sqrt{\kappa}$, so that $\mathcal{E}' (p)$ is (also) proportional to $\sqrt{\kappa}$. 

Although we are primarily interested in intermediate- to late-stage coarsening, it is perhaps
interesting to note the behavior of $\mathcal{E}' (p)$ as $a$ tends toward 0 (so $p$ tends toward
$p_{\min}$). As noted previously, in the event that $\mathbb{E}' (a) > 0$, we have the inequality 
\begin{equation*}
    \mathbb{E}' (a) < \frac{2\kappa}{a} \bar{\phi}_x (L; a) \bar{\phi} (L; a).
\end{equation*}
Recalling \eqref{periodic-wave}, we see that 
\begin{equation*}
    |\bar{\phi}_x (L; a)| 
    \le a \sqrt{\frac{\beta}{\kappa}},
\end{equation*}
where we have noted that $|\sn' (x; k)| \le 1$ for all $x \in \mathbb{R}$
(see, e.g., \cite{NIST:DLMF}). 
In this way, we conclude that 
\begin{equation*}
    \mathbb{E}' (a) \le 2a \sqrt{\beta \kappa}.
\end{equation*}
We can also obtain an alternative lower bound on $p' (a)$,
\begin{equation*}
    \begin{aligned}
        p' (a) &= \frac{4\sqrt{2 \kappa}}{a^2 \sqrt{\alpha}} 
     \int_0^1 \frac{1+z^2}{(1-z^2)^{1/2} (\delta^2 - z^2)^{3/2}} dz \\
     &\ge \frac{4\sqrt{2 \kappa}}{a^2 \delta^3 \sqrt{\alpha}} \int_0^1 \frac{1}{\sqrt{1-z^2}} dz
     = \frac{2 \pi\sqrt{2 \kappa}}{a^2 \delta^3 \sqrt{\alpha}}.
    \end{aligned}
\end{equation*}
Combining these observations, we see that 
\begin{equation*}
    \mathcal{E}' (p) = \mathbb{E}' (a)/p'(a)
    \le 2a \sqrt{\beta \kappa} \frac{a^2 \delta^3 \sqrt{\alpha}}{2 \pi\sqrt{2 \kappa}}
    = \frac{\sqrt{\alpha \beta}}{\pi \sqrt{2}} (\frac{2\beta}{\alpha} - a^2)^{3/2}.
\end{equation*}
We conclude that $\mathcal{E}' (p)$ is bounded above as $p$ approaches $p_{\min}$, though 
no longer proportional to $\sqrt{\kappa}$. 
\end{remark}

As an illustration of the considerations discussed in this section, we consider again 
the specific case of \eqref{ch} with $\kappa = 0.001$, $M (\phi) \equiv 1$, and $F$ as in \eqref{quarticF} 
with $\alpha = 1$ and $\beta = 1$. Referring to Figure \ref{energy-period-figure}, we pick the 
plateau between about 0.29 and 0.32, and 
in Figure \ref{EofPzoomed}
we zoom 
in on this flat section to see that in fact there is a small interval of periods for 
which monotonicity is lost.

\begin{figure}[ht] 
\includegraphics[width=1\textwidth]{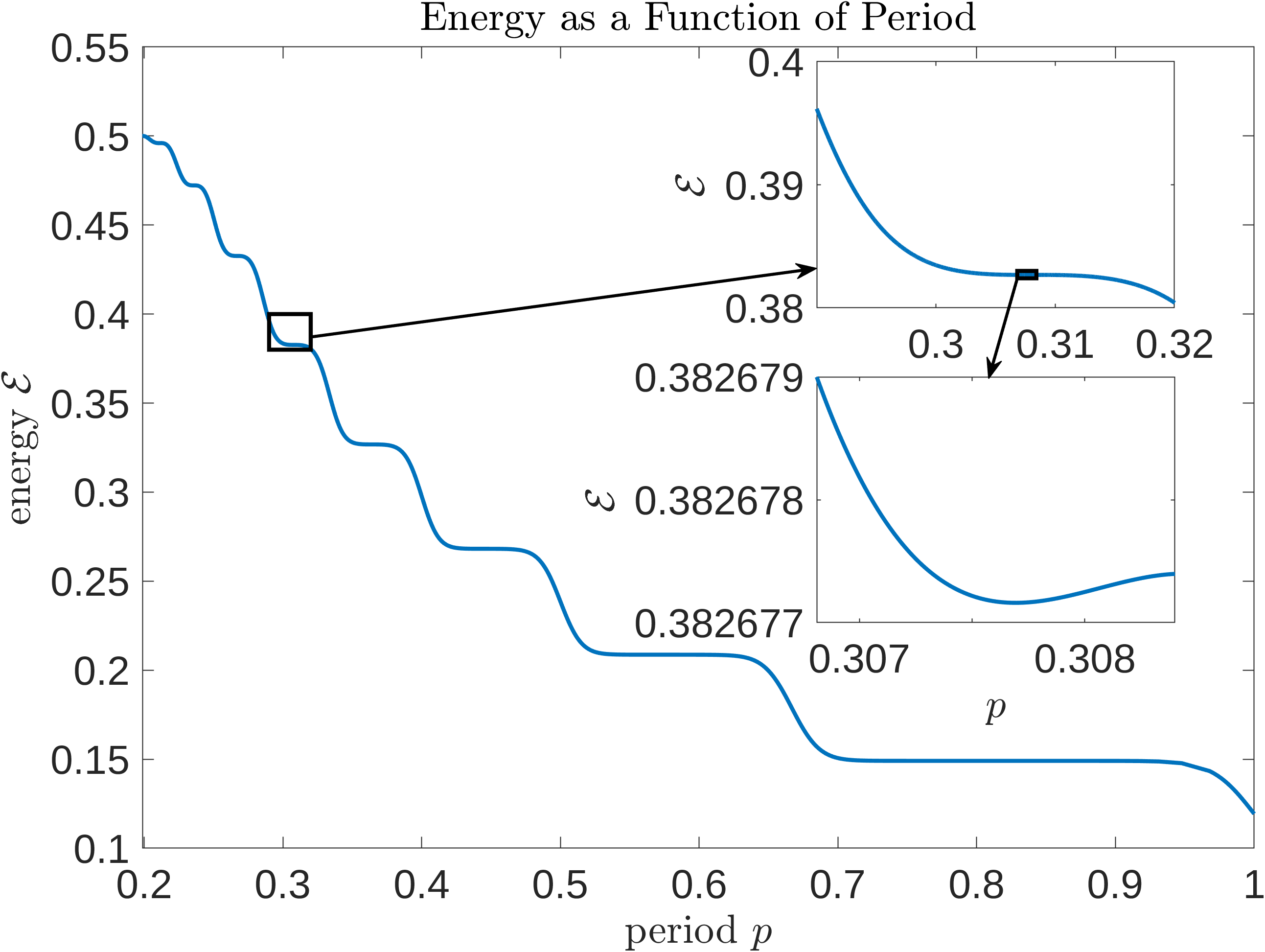}
\caption{\label{EofPzoomed} Plot of $\mathcal{E} (p)$, with zoom-ins at $p \in [0.29, 0.32]$ and $p \in [0.3068, 0.3084]$ inset.  Computed with $\alpha = \beta = 1$, $\kappa = 0.001$, and $L = 1$.}
\end{figure}

We end this section by briefly contrasting our measure of coarsening with 
the measure introduced by Kohn and Otto in \cite{KO02} in the multidimensional
case. For our one-dimensional setting, Definition 1 from p. 383 of \cite{KO02}
can be expressed as follows. 

\begin{definition}
    For any function $\phi (x)$ periodic on $[-L, +L]$ with mean 
    value 0, the Kohn--Otto length is defined to be 
    \begin{equation*}
        \ell_{KO} 
        := \sup \Big{\{} \frac{1}{2L} \int_{-L}^{+L} \phi (x) \zeta (x) dx: \,
        \zeta \in C^1 ([-L, +L]) \,\, \textrm{periodic with}\, \sup_{x \in [-L,+L]} |\zeta' (x)| \le 1\Big{\}}.
    \end{equation*}
\end{definition}

The clear advantage of the Kohn--Otto length over our measure of coarsening 
is that it readily generalizes to the multidimensional setting (the setting for 
which it was introduced in \cite{KO02}). In our one-dimensional setting, 
however, our measure has
several distinct advantages: (1) we can assign a length scale to any 
function $\phi \in C^1 ([-L, +L])$ as long as $E (\phi) \in  (E_{\min}, E_{\max}]$; 
(2) since our measure of length is directly linked to solutions of 
\eqref{ch}, it serves as a more precise indicator of the extent to which 
the associated solution $\phi (x, t)$ has evolved toward its final 
asymptotic state; and (3) our measure is more readily computed (by \eqref{pseudoinverse_eq}).

\subsection{Coarsening rates}
\label{coarsening-rates-section}

We now proceed by comparing coarsening rates for 
solutions of \eqref{ch} computed 
in three different ways: (1) by direct numerical 
integration of \eqref{ch}; (2) by a long-time coarsening 
model due to Langer \cite{L71}; and (3) by a
coarsening model due to one of the authors \cite{H11}. 

{\bf Computation.} For our direct computations, we will
initialize the flow with random perturbations of the 
homogeneous configuration $\phi_0 (x) \equiv 0$. We 
then solve forward until the energy $E(\phi (x, t))$
reaches the spinodal energy $E_s$, and it is from this 
point that we compare energies computed in three 
different ways. 

{\bf Langer's relation.} 
In \cite{L71}, Langer employs a statistical development to capture thermal 
fluctuations driving phase separation, and arrives at a straightforward
equation for the evolution of a coarseness measure $\ell$ as a function
of time. In the current framework, we can interpret equation (6.26)
in \cite{L71} as an equation for the period $p$ as a function of 
time, 
\begin{equation} \label{langer_period}
p (t) = p_0 + \sqrt{\frac{2 \kappa}{\beta}} 
\ln \Big(1 + \frac{16 \beta^2 (t-t_0)}{\kappa} e^{-\frac{p_0}{\sqrt{2\kappa/\beta}}} \Big).
\end{equation}
Here, we take the period to be $p_0$ at time $t_0$.

\begin{remark} \label{Langer626remark}
Equation (6.26) in \cite{L71} is 
\begin{equation} \label{Langer626}
\ell (t) = \ell_0 + \frac{\xi}{2} 
\ln \Big(1 + \frac{32 t}{\tau_0} e^{- 2 \ell_0/\xi} \Big),     
\end{equation}
where $\ell (t) = p(t)/2$ (with $p(t)$ as in the current 
analysis; see Figure 4 in \cite{L71}). In obtaining 
\eqref{langer_period}, we have chosen Langer's constants
$\Gamma$, $a$, $d$, $\kappa_b$, and $T$ so that 
\begin{equation*}
    \frac{\Gamma a^{2+d}}{2 \kappa_b T} = 1,
\end{equation*}
in which case equation (3.4) in \cite{L71} becomes
(in the notation of \cite{L71})
\begin{equation*}
    \bar{\eta}_t = (- \epsilon_0 \xi_0^2 \bar{\eta}_{xx} + F' (\bar{\eta}))_{xx}.
\end{equation*}
In particular, we obtain our equation \eqref{ch} with 
$M(\phi) \equiv 1$ and $\kappa = \epsilon_0 \xi_0^2$. The 
combination $\epsilon_0 \xi_0^2$ always appears together in the 
development of \cite{L71}, and is replaced below with the constant  
$\kappa$.
(Here, we have denoted Boltzmann's constant $\kappa_b$ to 
distinguish it from our $\kappa$; Langer denotes it by 
$\kappa$.) In this case, $\tau_0 = 2 \kappa/\beta^2$ and $\xi = \sqrt{2\kappa/\beta}$,
so that \eqref{Langer626} becomes  
\begin{equation*} 
\ell (t) = \ell_0 + \sqrt{\frac{\kappa}{2 \beta}} 
\ln \Big(1 + \frac{16 \beta^2 t}{\kappa} e^{- 2 \ell_0/\sqrt{2\kappa/\beta}} \Big).    
\end{equation*}
Our \eqref{langer_period} is obtained by multiplying this last expression by 
$2$, replacing $2 \ell_0$ with $p_0$, and allowing the initial time 
to be $t_0 > 0$. (Langer takes the initial time to be $0$ by convention, 
initiating late-stage dynamics at $t = 0$.)
\end{remark}

{\bf The approach of \cite{H11}}. A drawback of Langer's approach 
in \cite{L71} is that approximations are made that require the 
coarsening process to be at an asymptotically late stage. In 
particular, Langer approximates late-stage steady-state periodic solutions 
by piecing together enriched regions with transitions taken from 
kink and antikink solutions. In \cite{H11}, this approach is refined
by the use of exact periodic solutions $\bar{\phi} (x; a)$ in place 
of the approximate solutions, with the further advantage that these 
periodic solutions can serve to approximate the dynamics as early as 
the spinodal time. Following this replacement of the approximate
solutions with exact solutions, the method of \cite{H11} follows
Langer's approach of linearizing and determining the coarsening 
rate from the eigenvalues of the resulting eigenvalue problem.
Before describing the model obtained in this way, we briefly summarize 
an efficient method for computing these eigenvalues. First, if 
\eqref{ch} (with $M \equiv 1$) is linearized
about $\bar{\phi} (x; a)$, with $\phi = \bar{\phi} + v$,
we obtain the perturbation equation 
\begin{equation*}
    v_t = (- \kappa v_{xx} + F'' (\bar{\phi})v)_{xx},
\end{equation*}
and the corresponding eigenvalue problem 
\begin{equation} \label{CH-EV-eqn}
    (- \kappa \psi'' + F''(\bar{\phi}) \psi)'' = \lambda \psi.
\end{equation}

As a starting point, we can explicitly compute the right-most
eigenvalue for the limiting case with amplitude $a = 0$, 
corresponding with $\bar{\phi} (x; 0) \equiv 0$. In this
case, \eqref{CH-EV-eqn} becomes 
\begin{equation*}
- \kappa \psi'''' - \beta \psi'' = \lambda \psi,
\end{equation*}
where we have observed that $F'' (0) = - \beta$. This 
equation only has $L^{\infty} (\mathbb{R})$ eigenvalues, 
and $\lambda$ is such an eigenvalue if and only if 
$\psi (x) = e^{i \xi x}$ is a solution for some 
$\xi \in \mathbb{R}$. Looking for such solutions, we find 
\begin{equation*}
    - \kappa \xi^4 + \beta \xi^2 = \lambda,
\end{equation*}
and so to find the maximum possible value of $\lambda$
we maximize $\lambda (\xi)$. We find that the maximum 
occurs at $\xi = \pm \sqrt{\beta/(2\kappa)}$, with 
\begin{equation*}
    \lambda_{\max}  = \frac{\beta^2}{4\kappa}.
\end{equation*}
For the specific values $\kappa = 0.001$ and $\beta = 1$,
this is $\lambda_{\max} = 250$. 

More generally, in \cite{H11} (which adapts the approach 
of \cite{Gardner1993, Gardner1997} to the current setting), 
Floquet theory is used to characterize the leading eigenvalue associated
with any periodic solution $\bar{\phi} (x; a)$. 
To understand how this works, we first express 
the eigenvalue problem \eqref{CH-EV-eqn}
as a first-order system, setting $y = (y_1, y_2, y_3, y_4)^T$,
with $y_j = \psi^{(j)}$, $j = 0, 1, 2, 3$. Then $y$ solves 
the system 
\begin{equation} \label{ev-ode}
    y' = \mathbb{A} (x; \lambda) y,
    \quad 
    \mathbb{A} (x; \lambda)
    = \begin{pmatrix}
        0 & 1 & 0 & 0 \\
        0 & 0 & 1 & 0 \\
        0 & 0 & 0 & 1 \\
        (b''(x) - \lambda)/\kappa & 2b'(x)/\kappa & b(x)/\kappa & 0
    \end{pmatrix},
\end{equation}
where $b(x) = F''(\bar{\phi} (x; a))$.
We can compute a fundamental matrix for \eqref{ev-ode},
\begin{equation*}
    \Phi' = \mathbb{A} (x; \lambda) \Phi,
    \quad \Phi (0; \lambda) = I_4,
\end{equation*}
and the monodromy matrix is defined to be $M(\lambda; p) := \Phi (p; \lambda)$.
The Evans function is 
\begin{equation}
    D(\lambda, \xi) 
    = \det (M(\lambda; p) - e^{i \xi p} I_4), 
\end{equation}
and in this setting $\lambda$ is an eigenvalue of \eqref{CH-EV-eqn}
if and only if $D (\lambda, \xi) = 0$ for some $\xi \in \mathbb{R}$.
(See \cite{Gardner1993, Gardner1997} for details on the 
Evans function for periodic solutions generally, and \cite{H09, H11} for 
specialization to the current setting.)

In practice, we can use this to compute leading eigenvalues in 
the following relatively efficient way. First, we know that 
for $a = 0$, with corresponding minimum period $p_{\min}$ 
(from \eqref{minimim-period}), the leading eigenvalue is 
$\lambda_{\max}  = \frac{\beta^2}{4\kappa}$. We now increment
the value $a$ by some small step $\Delta a$, and compute the 
corresponding period $p(a+\Delta a)$, using \eqref{a-to-p}.
The leading eigenvalue for $\bar{\phi} (x; a+ \Delta a)$ will
be just below the leading eigenvalue for $\bar{\phi} (x; a)$,
so we search for zeros of $D (\lambda, \xi)$ with values of 
$\lambda$ just below $\lambda_{\max}$. Repeating this process
iteratively we obtain the leading eigenvalues. 
For $\alpha = 1$, $\beta = 1$, and $\kappa = 0.001$, a plot of 
leading eigenvalues as a function of amplitude $a$ is depicted 
in Figure \ref{evplot1} for two different values of $\kappa$. 

\begin{figure}[ht] 
\begin{center}
\includegraphics[width=12cm,height=8.2cm]{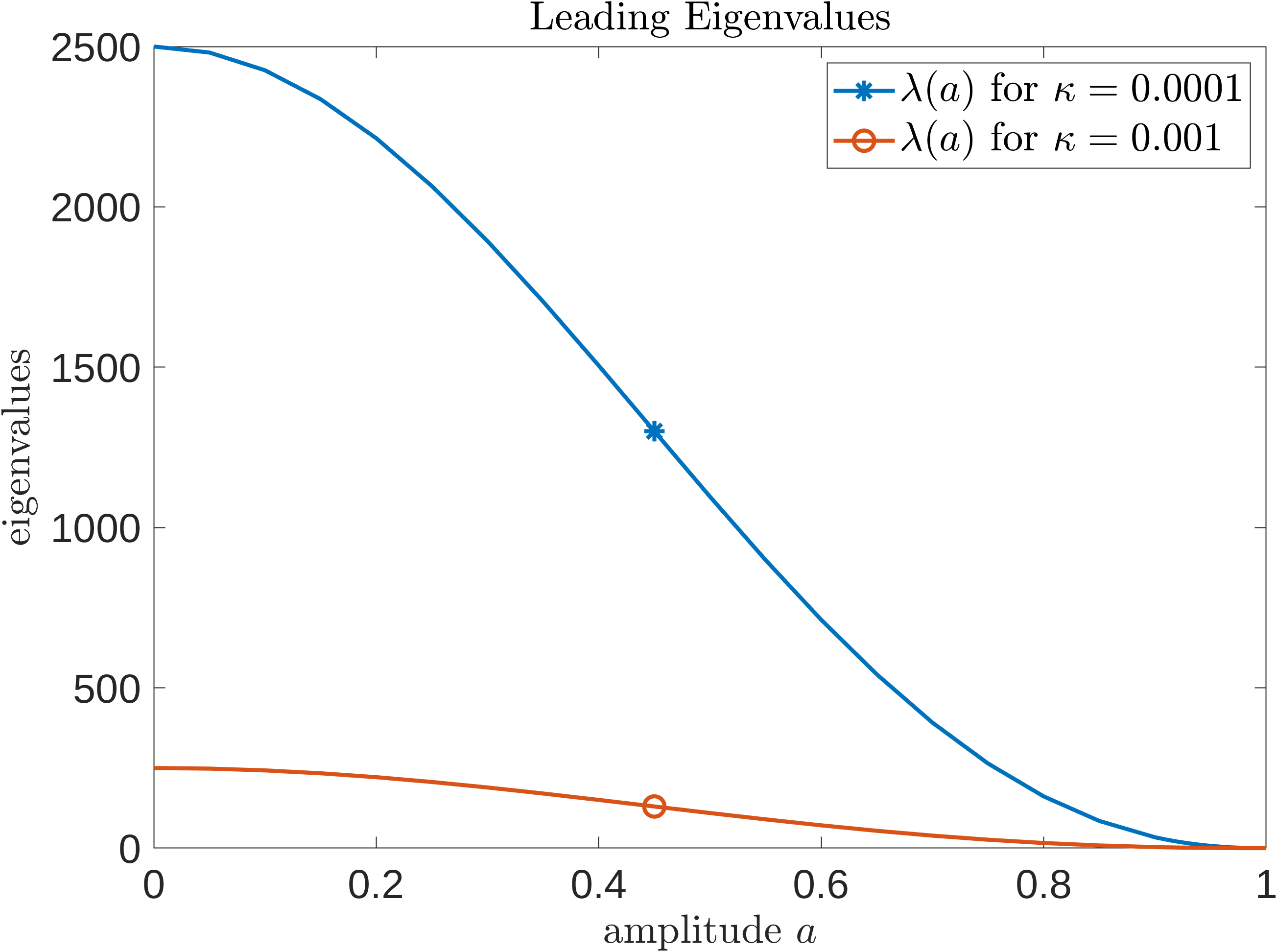}
\end{center}
\caption{Plot of leading eigenvalues versus amplitude, 
computed with $\alpha = \beta = 1$, and $\kappa = 0.001$ and $\kappa = 0.0001$. 
\label{evplot1}}
\end{figure}

According to \cite{H11}, the period $p(t)$ associated with 
an evolving solution of \eqref{ch} evolves approximately 
according to the relation 
\begin{equation} \label{eigenvalue-ode}
    \frac{dp}{dt}
    = \lambda_{\max} (p) p,
    \quad p(t_0) = p_0,
\end{equation}
where $\lambda_{\max} (p)$ denotes the maximum eigenvalue 
associated with the specified periodic solution $\bar{\phi} (x;a(p))$.

\begin{remark}\label{half_ev_remark}
In the derivation of \eqref{eigenvalue-ode} in \cite{H11}, 
a small modification to Langer's framework is employed. Precisely,
the equation arising from a more faithful adaptation of 
Langer's original relations is
\begin{equation} \label{eigenvalue-ode-langer}
    \frac{dp}{dt}
    = \frac12\lambda_{\max} (p) p,
    \quad p(t_0) = p_0.
\end{equation}
As a side note to our analysis, we will compare how these
two possible methods compare with numerically generated 
solutions.

If Langer's method and the eigenvalue method are both initialized 
by the minimum period $p_{\min}$ (from \eqref{minimim-period}), 
the evolution in time proceeds as in Figure \ref{max-coarsening-figure}
(computed for $\alpha = \beta = 1$, $\kappa = 0.001$). 
Using $\mathcal{E} (p)$ to map periods to energies, we can evolve
energy as a function of time for both Langer's method and the 
eigenvalue method. This is depicted in Figures \eqref{max-energy-figure} and \eqref{max-coarsening-figure}. 
In order to see that the primary difference between the two approaches
is a matter of scaling, we create the same two plots using 
\eqref{eigenvalue-ode-langer} in place of \eqref{eigenvalue-ode},
as displayed in Figures \eqref{max-energy-figure-half} and \eqref{max-coarsening-figure-half}
\end{remark}


\begin{figure}[ht] 
\begin{minipage}[T]{0.48\textwidth}
\centering
\begin{subfigure}[T]{1\textwidth}
\includegraphics[width=1\textwidth]{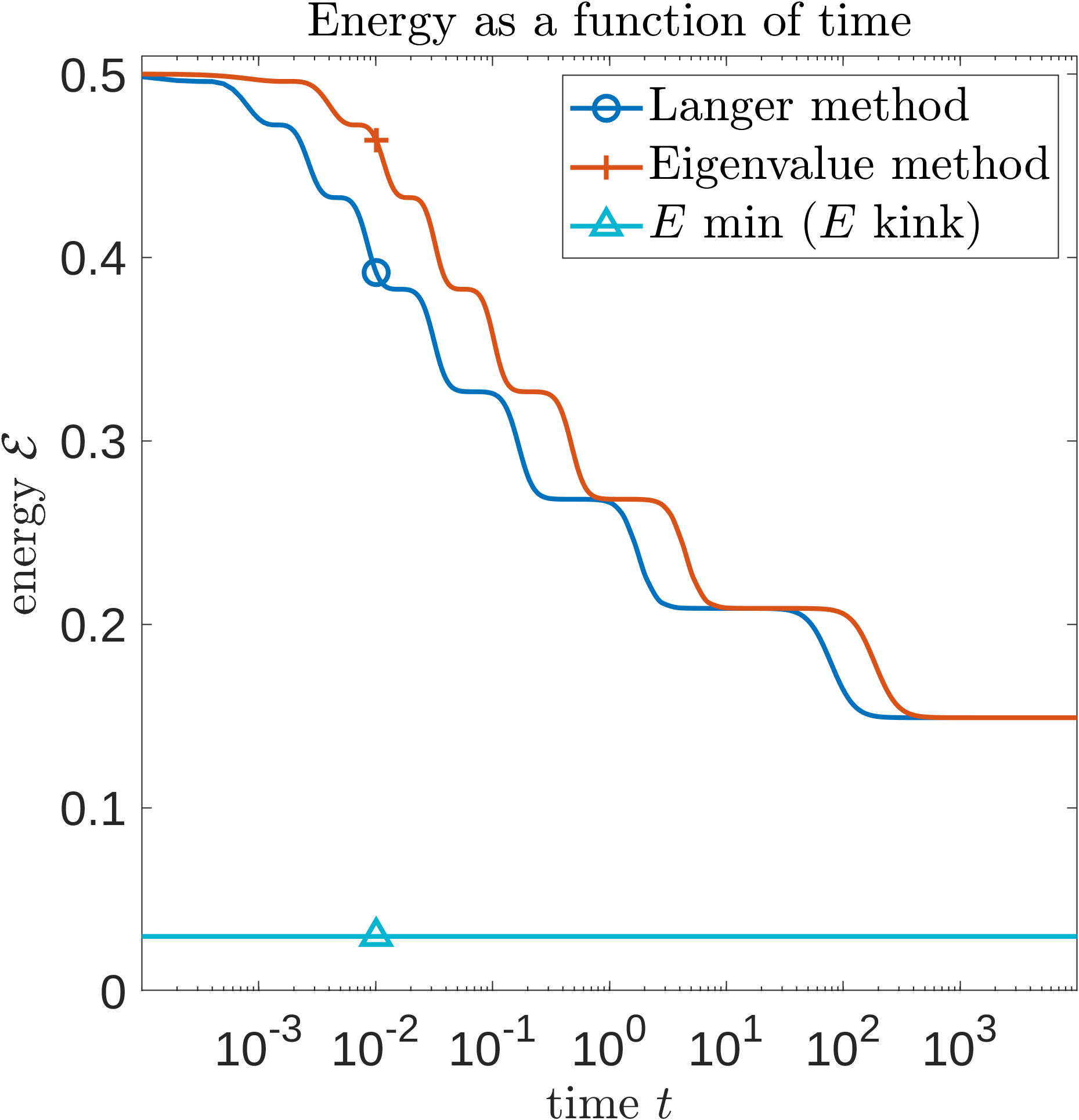}
\caption{\label{max-energy-figure}}
\end{subfigure}
\end{minipage}
\hfill 
\begin{minipage}[T]{0.48\textwidth}
\centering
\begin{subfigure}[T]{1\textwidth}
\includegraphics[width=1\textwidth]{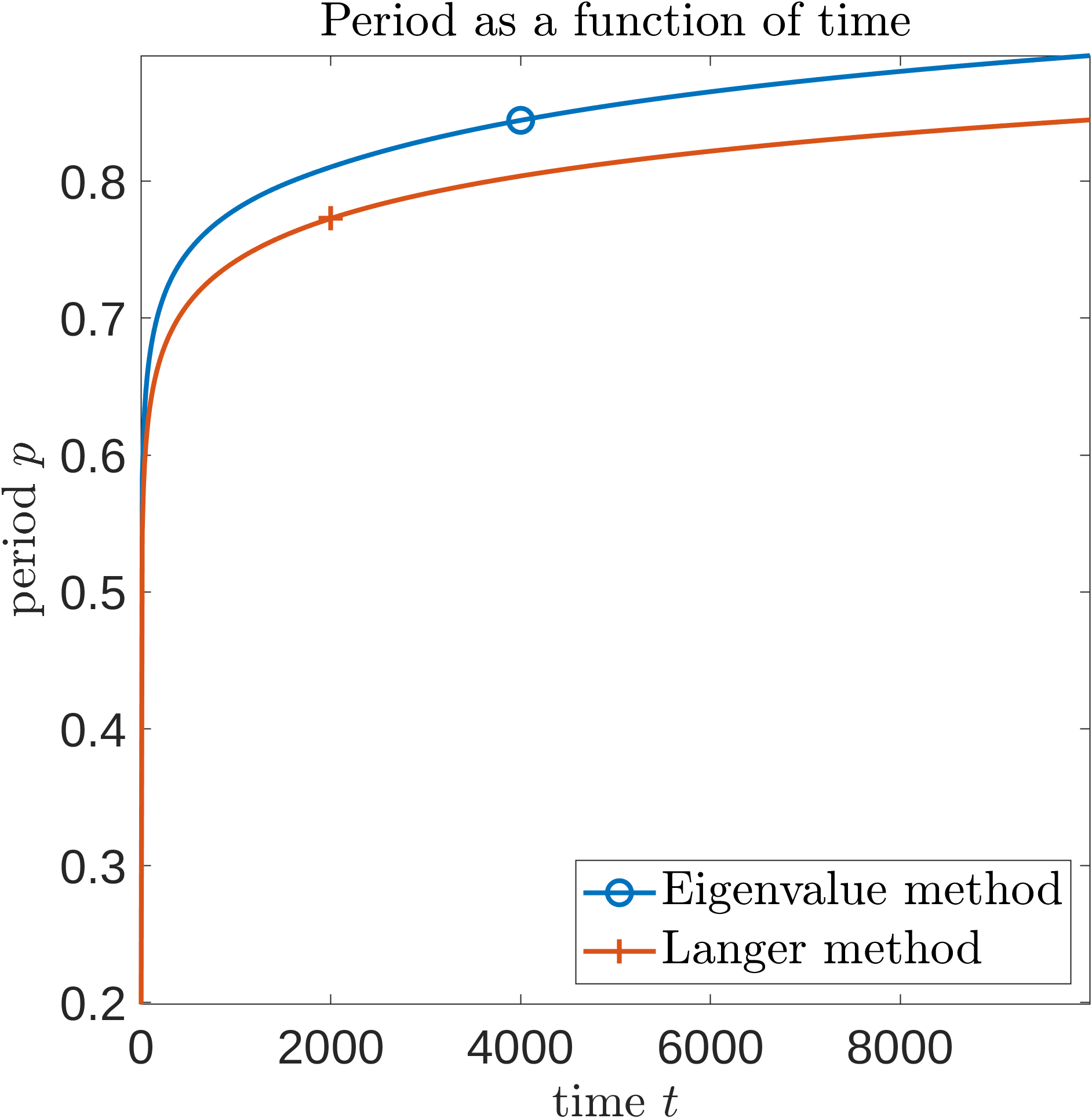}
\caption{\label{max-coarsening-figure}}
\end{subfigure}
\end{minipage}
\begin{minipage}[T]{0.48\textwidth}
\centering
\begin{subfigure}[T]{1\textwidth}
\includegraphics[width=1\textwidth]{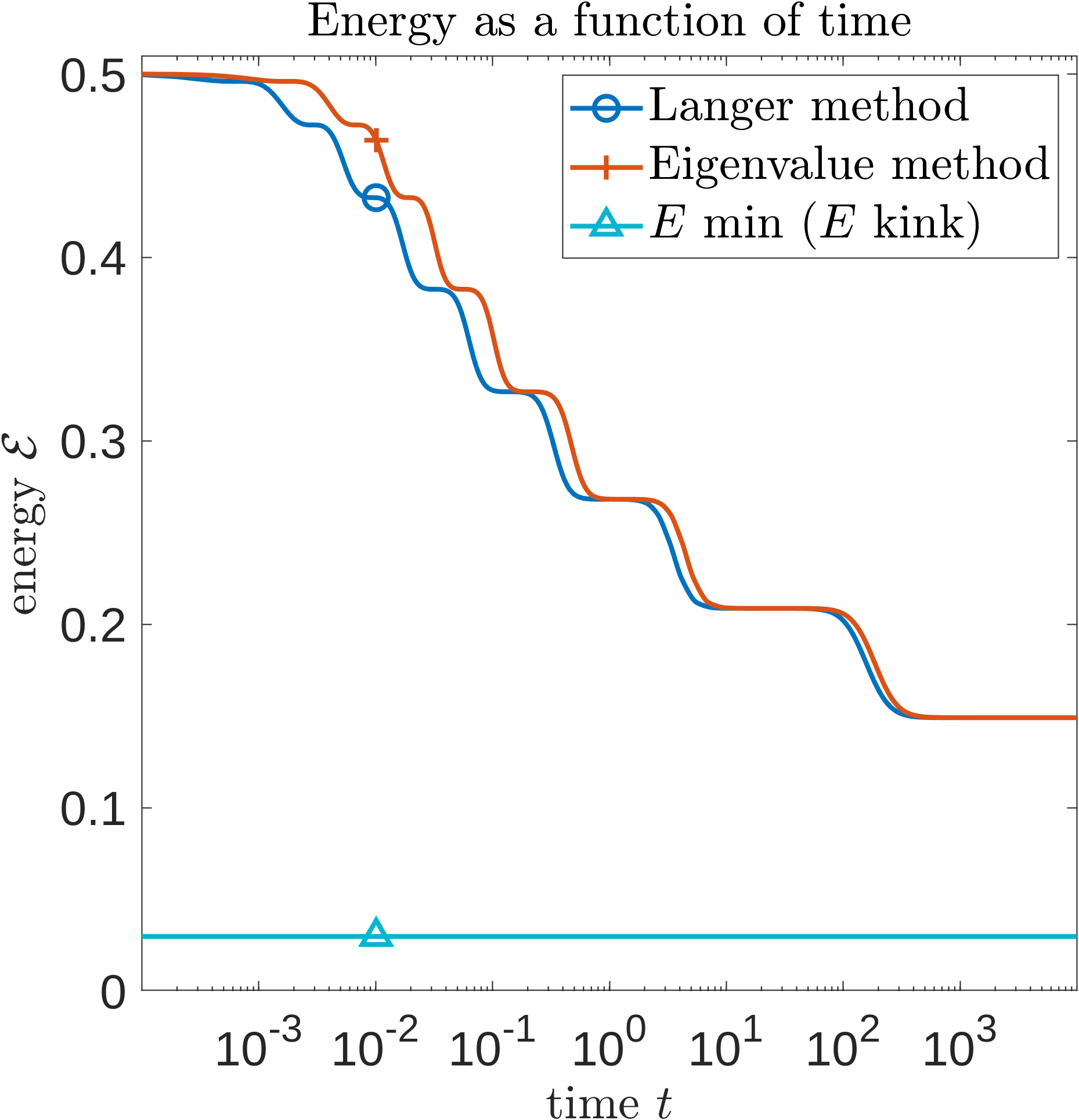}
\caption{\label{max-energy-figure-half}}
\end{subfigure}
\end{minipage}
\hfill 
\begin{minipage}[T]{0.48\textwidth}
\centering
\begin{subfigure}[T]{\textwidth}
\includegraphics[width=1\textwidth]{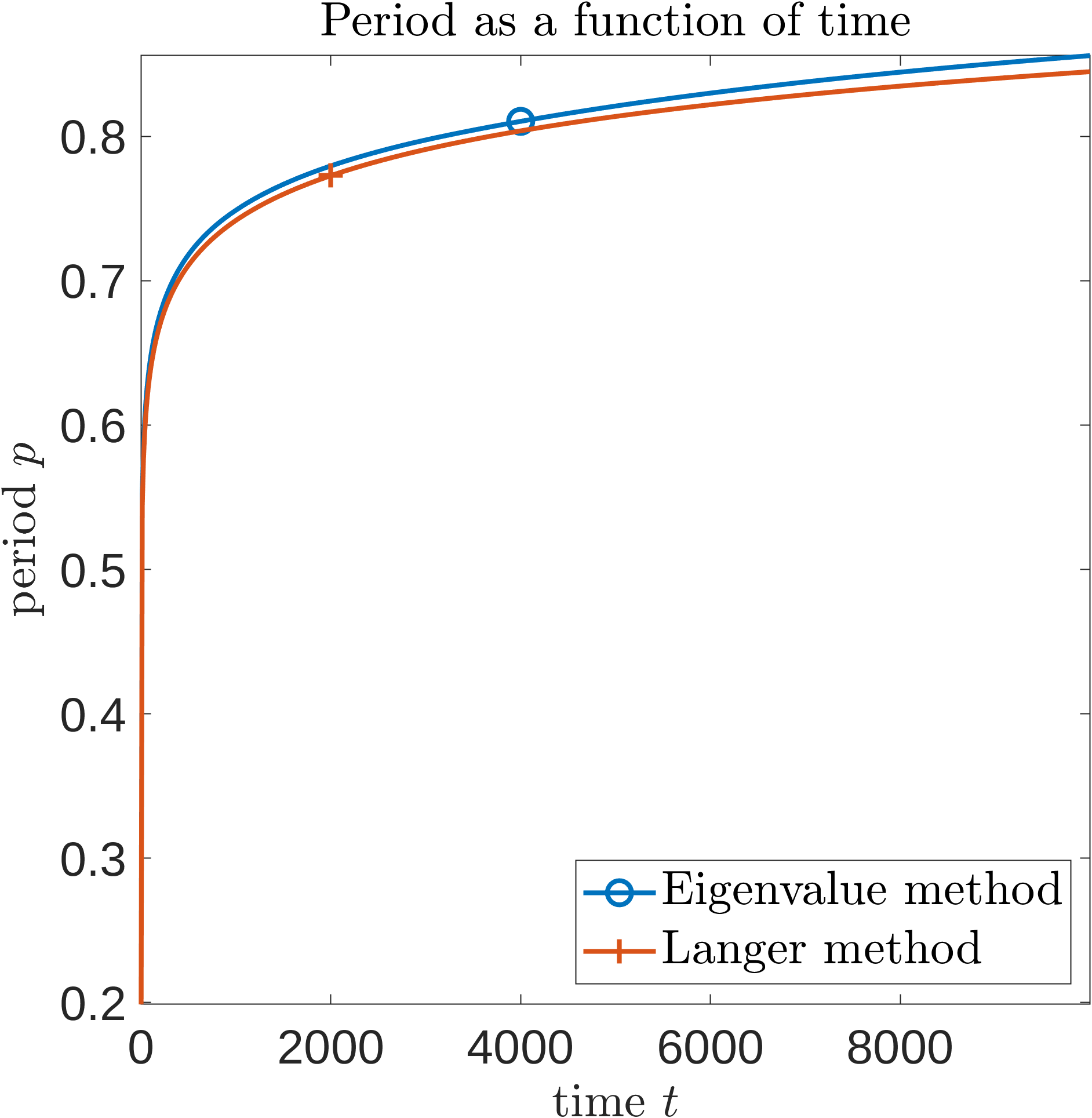}
\caption{\label{max-coarsening-figure-half}}
\end{subfigure}
\end{minipage}
\caption{\label{2x2}Evolution of energies ((a) and (c)) and periods ((b) and (d)) via Langer's method and the eigenvalue method.  For the eigenvalue method, equation \eqref{eigenvalue-ode} was used in (a) and (b), and equation \eqref{eigenvalue-ode-langer} (with the extra $\frac12$-factor) was used in (c) and (d),  cf. Remark \ref{half_ev_remark}.}
\end{figure}

\FloatBarrier

\subsection{Varying \texorpdfstring{$\kappa$}{}}
In this section, we observe that it is straightforward to vary 
$\kappa$ in the above energy calculations. First, we note that 
for a fixed interval $[-L, +L]$ the maximum 
energy specified in \eqref{maximum-energy} does not depend on $\kappa$. 
On the other hand, the minimum energy specified in 
Item (iv) of Proposition \ref{specific-F-proposition} 
is proportional to $\sqrt{\kappa}$. 
For evolution on bounded domains, transition layers are 
typically removed through the boundary in pairs, so we 
expect to see energy drops in steps of size 
$2 E_{\min}$. For $\alpha = 1$, $\beta = 1$, and $L = 1$, 
if $\kappa = 0.001$ then $2 E_{\min} = 0.0596$, 
and likewise if $\kappa = 0.0001$, then $2 E_{\min} = 0.0189$, and if 
$\kappa = 0.00001$, then $2 E_{\min} = 0.0060$. Since the energy 
declines in steps of these sizes, energy plots such as the two 
depicted in Figure \ref{energies_kappa3} have less pronounced steps
for smaller values of $\kappa$. 

For Langer's approach, dependence on $\kappa$ is explicit in 
\eqref{langer_period}, and so changes in $\kappa$ are readily 
accommodated. For the eigenvalue approach, we need to identify 
how the leading eigenvalues vary with $\kappa$. To this end, we 
fix a choice of $F$
with the form \eqref{quarticF}, and we observe 
that if $\bar{\phi} (x)$
denotes a periodic solution of \eqref{ch} obtained with $\kappa = 1$, then 
for any $\kappa > 0$, $\bar{\phi}^{\kappa} (x) := \bar{\phi} (x/\sqrt{\kappa})$ is a 
periodic solution of \eqref{ch} obtained with the value 
$\kappa$. Upon linearization of \eqref{ch} about $\bar{\phi}^{\kappa} (x)$,
we arrive at the eigenvalue problem 
\begin{equation*}
    (-\kappa \psi'' + F'' (\bar{\phi}^\kappa (x)) \psi)'' = \lambda \psi.
\end{equation*}
We can express this equation as 
\begin{equation*}
    (-\kappa \psi'' + b(x/\sqrt{\kappa}) \psi)'' = \lambda \psi,
\end{equation*}
where $b (x/\sqrt{\kappa}) = F'' (\bar{\phi} (x/\sqrt{\kappa}))$, 
and we can also express this as 
\begin{equation*}
    - \kappa \psi'''' + \frac{1}{\kappa} b''(x/\sqrt{\kappa}) \psi 
    + \frac{2}{\sqrt{\kappa}} b' (x/\sqrt{\kappa}) \psi' 
    + b(x/\sqrt{\kappa}) \psi'' = \lambda \psi. 
\end{equation*}
At this point, we make the change of variables
\begin{equation*}
    y = \frac{x}{\sqrt{\kappa}}, \quad
    \Psi (y) = \psi(x) \implies \psi^{(k)} (x) = \frac{1}{\kappa^{k/2}} \Psi^{(k)} (y), 
    \quad k = 1,2,3,\dots.
\end{equation*}
Upon substitution, we see that the equation for $\Psi (y)$ is 
\begin{equation*}
    -(\Psi'' + b(y) \Psi)'' = \kappa \lambda \Psi.
\end{equation*}
We take from this that $\lambda$ is an eigenvalue for 
the equation with general $\kappa$ if and only if $\kappa \lambda$
is an eigenvalue for the equation with $\kappa = 1$. This observation 
allows us to compute leading eigenvalues for any fixed $\kappa > 0$,
and obtain leading eigenvalues associated with all other values 
of $\kappa$ by an appropriate scaling argument. To make this 
precise, suppose that for a specific value $\kappa_0$, we compute
the leading eigenvalue $\lambda_0$ associated with the periodic 
solution with amplitude $a$. Then $\kappa_0 \lambda_0$ will be 
the leading eigenvalue associated with the periodic solution 
with amplitude $a$ arising as a solution to \eqref{ch} with 
$\kappa = 1$. Correspondingly, the leading eigenvalue associated
with the periodic solution with amplitude $a$ arising as a solution
to \eqref{ch} with any general value $\kappa > 0$ will be 
$(\kappa_0/\kappa) \lambda_0$. In fact, we have already seen an 
example of this scaling effect in Figure \ref{evplot1}. 
 
In Figure \ref{energies_kappa3}, we apply these ideas 
to generate for comparison plots of energy as a function 
of time for both Langer's method and the eigenvalue method
for three values of $\kappa$, namely $\kappa = 0.001$, 
$\kappa = 0.0001$, and $\kappa = 0.00001$. In each case, the energy at 
$t = 0$ is the same maximum energy $E_{\max} = 0.5$, 
but the plots are depicted starting with $t = 0.01$,
at which point the energies have already declined at 
varying rates. (These calculations use \eqref{eigenvalue-ode}
rather than \eqref{eigenvalue-ode-langer}.)

\begin{figure}[ht] 
\begin{center}
\includegraphics[width=12cm,height=8.2cm]{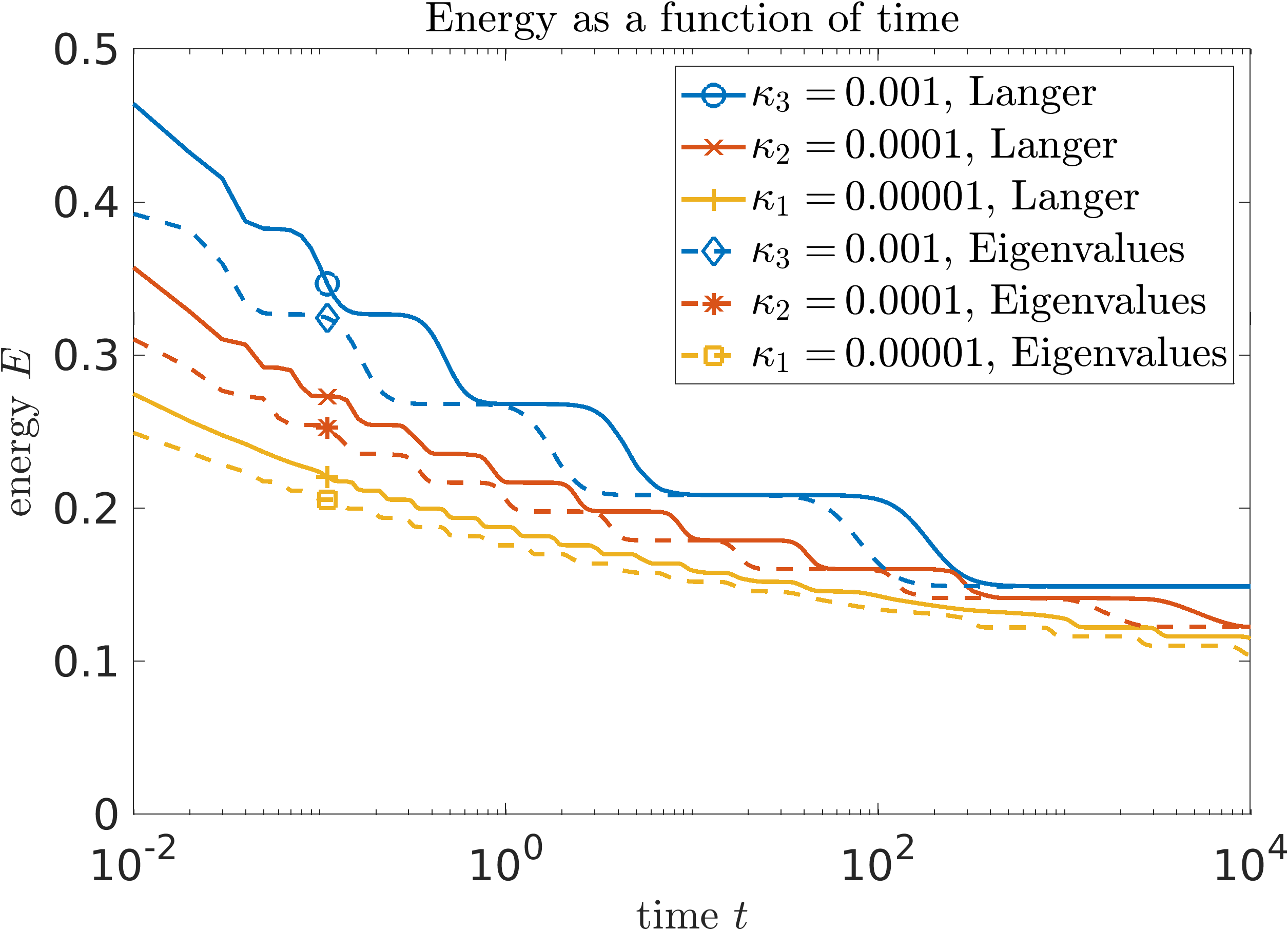}
\end{center}
\caption{Evolution of energies via Langer's method and the eigenvalue method for various values $\kappa$. 
\label{energies_kappa3}}
\end{figure}

\section{Computational Results}
\label{sec-computational-results}

\subsection{Overview of the numerical methods}
\label{overview-numerical-section}

Here, we give a brief summary of the numerical methods used throughout this analysis.  
As much as possible, we have used standard well-established methods, as our goal is not to 
focus on developing numerical methods, but on understanding dynamical phenomena.

All simulations were run using MATLAB version 2023a.  The spatial discretization was done using standard spectral methods, based on MATLAB's \texttt{fft} and \texttt{ifft} (the latter computed using the \texttt{{\textquotesingle}symmetric\textquotesingle} option).  Time-stepping was handled by a semi-implicit Euler method.  Namely, the nonlinear convection and advection terms were computed explicitly in physical space using co-location (i.e., multiplication in physical space).  For the Burgers equation, the linear term was handled implicitly by IMEX-type methods (i.e., simply dividing by the appropriate Helmholtz operator after discretization via Euler's method).  For the Cahn--Hilliard portion, Eyre's convex splitting method \cite{Eyre_1998,Eyre_1997} was employed (in particular, algorithm 5, the ``Linearly Stabilized Splitting Scheme'' proposed in \cite{Eyre_1997}), with the cubic term computed using co-location (i.e., multiplication in physical space).  Due to the presence of the cubic term, the highest half of the wave modes were dealiased (i.e., set to zero before being transformed back to physical space).  The coupling between the equations was handled monolithically, i.e., both equations updated in a single step.  Spatial resolution on the domain $[-1,1)$ was chosen to be $N=8192$, giving a spatial stride of $\Delta x=2\pi/N\approx7.6699\times10^{-4}$.  The time-step $\Delta t\approx 9.7656\times 10^{-5}$ was chosen to respect the advective and convective CFL conditions.  For simplicity, we chose $\alpha=\beta=K=1$. Our smallest choices of viscosity $\nu\geq0.006$ and interfacial energy $\kappa\geq0.00001$, were found experimentally using the resolution criterion that the energy spectra of dealiased modes of both $\phi$ and $v$ must be below machine precision ($\approx2.2204\times10^{-16}$) at all time steps after roughly time $t\approx 0.001$ (obviously, since we are starting with normally-distributed random values at each point, the initial data, and consequentially the first few time steps, are not expected to satisfy this criterion).

We often initialize our systems with ``random'' initial data, which has been evolved until the free energy is just below $0.99 E_{\max}$, our
tolerance taken somewhat arbitrarily to be $10^{-4}$, where $E_{\max}:=\frac{\beta^2 L}{4\alpha}$.  This is carried out in the following way.  First, for the purposes of easy replication of our results, we seed the random number generator (RNG) with seed \verb|0|, then we take normally-distributed mean-zero  data at each of the $N$ points in space with standard deviation \verb|0.1| (i.e., $\mathcal{N}(0,0.1)$).  The result is then transformed with the discrete Fourier transform, where the upper $N/2$ wave modes are removed (for dealiasing the cubic term).  For clarity, the MATLAB code used for this operation is as follows:
\begin{lstlisting}
    rng(0,'twister');
    phi_hat = fft(0.1*randn(1,N));
    phi_hat((N/4+1):(3N/4+1)) = 0;
    phi = ifft(phi_hat,'symmetric');
\end{lstlisting}
Here, we assume $N=2^p$ for some integer $p\geq2$ (we used $p=11$ for pure Cahn--Hilliard tests, and $p=13$ for Burgers--Cahn--Hilliard tests).  The RNG was only seeded by zero on the first trial.  On subsequent trials, the seed was determined by whatever the previous state of the system was.  After the initialization of $\phi$ described above, the simulation was then run until free energy satisfied the criterion stated above.  For this last stage, we used a variable time-step: if the free energy was not within the desired tolerance of $0.99 E_{\max}$, the time step was thrown out and recomputed with half the previous time-step.  Once this process was completed, the result was used to initialize our Cahn--Hilliard simulations.  For the initialization of the Burgers equation, we used either random Fourier coefficients, namely normally distributed $c_k\sim \mathcal{N}(0,1)+i\mathcal{N}(0,1)$ coefficients on wave moves\footnote{The cut-off $|k|\leq32$ was chosen somewhat arbitrarily, but in several tests at different cut-off numbers, no major qualitative differences were observed (except of course when the cut-off was near the Nyquist frequency $N/2$, leading to numerical artifacts than can be fixed by using higher resolution).  This is expected, since the viscosity quickly smooths the high frequencies.  Hence, to focus on important details, we did not display results for different cut-off frequencies for the initial velocity.  Similar remarks hold for other choices of parameters, such as choices for the standard deviations, etc.} $|k|\leq32$ satisfying $c_{-k}=\overline{c_k}$, or the following ``bump'' function times $x$ (for a mean-free compactly supported $C^\infty$ function): 
\begin{equation}\label{bump}
v_0(x)=
\begin{cases}
\frac{C}{L} x\exp\left(\frac{1}{(x/L)^2 - 1}\right)\,&-L<x<L,\\
0&\text{otherwise},
\end{cases}
\end{equation}
where $C=\left(\sqrt{2-\sqrt{3}}\exp(1/(1-\sqrt{3}))\right)^{-1}\approx7.5724$ was chosen so that $\|v_0\|_{L^\infty}=1$ (the max occurs at $x=\sqrt{2-\sqrt{3}}=\frac12(\sqrt{6}-\sqrt{2}))$.  For our simulations, we choose $L=1$.

\subsection{The Cahn--Hilliard Equation}
\label{uncoupled-system-section}
Having developed an analytical baseline for comparison in Section \ref{ch-coarsening}, we begin 
our computations with the uncoupled Cahn--Hilliard equation \eqref{ch}, with an emphasis on the energy 
evolution and associated coarsening dynamics.
In Figure \ref{energies_av}, we depict energies (Figure \ref{energies_av}a) and periods
(Figure \ref{energies_av}b) for 50 trials of the Cahn--Hilliard equation---trials in gray, 
average in thicker black---along with energies and periods computed with the two analytical methods
from Section \ref{ch-coarsening}. For the computationally generated values,  
initial data for each trial was computed pseudo-randomly as described  
in Section \ref{overview-numerical-section}, except with the random number generator seeded randomly 
each time, using \texttt{rng(randi(10000))}.  The free energy was computed for each trial at each 
time-step, and the mean of these curves was computed and displayed in Figure \ref{energies_av}a.  
This is the first direct quantitative comparison that we are aware of for the energies obtained from 
these three approaches, and we interpret the consistency of the results at once as a justification of 
Langer's approach in \cite{L71} (and subsequently as modified in \cite{H11}) and a verification that 
our computations are faithfully capturing dynamics of the energy. In Figure \ref{energies_av}b, 
we depict the evolution of the periods associated with the energies in Figure \ref{energies_av}a. 
For the two analytic methods, these periods are computed directly from \eqref{langer_period} (Langer's 
method) and \eqref{eigenvalue-ode-langer} (eigenvalue method), while for the computational approach the periods
are computed from the energies via the pseudoinverse as specified in Definition \ref{pseudoinverse}
(though see the caption of Figure \ref{energies_av} for a note on this). 

As a final comment to this section, we note that while the two analytically obtained energy curves
in Figure \ref{energies_av}a both remain close to the mean curve obtained by computation, 
the method of \cite{H11} yields a curve that remains moderately closer over intermediate times. 

\begin{figure}[ht] 
\begin{subfigure}[t]{0.49\textwidth}
\includegraphics[width=1\textwidth]{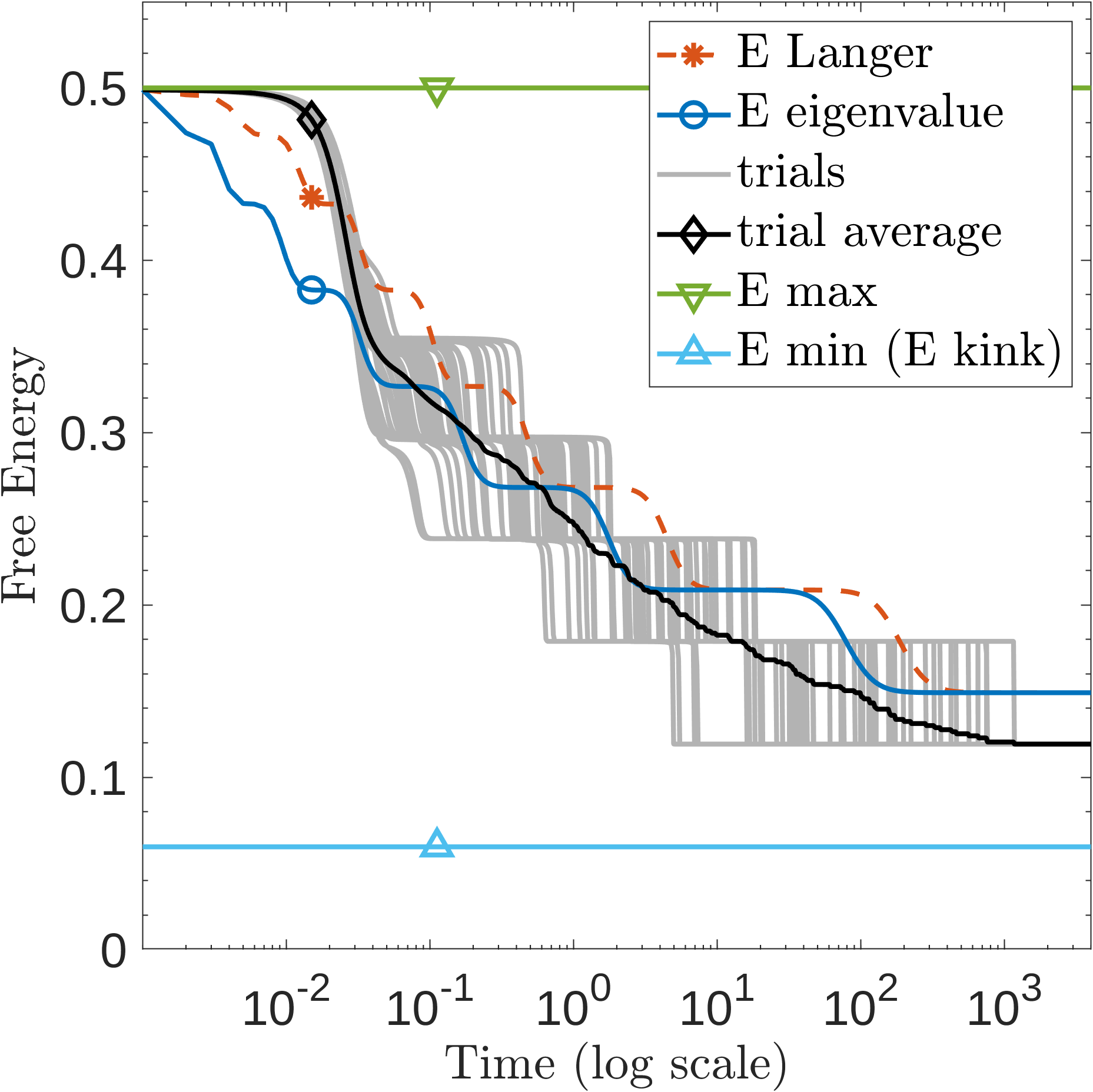}
\caption{\label{fig_En_av} (linear-log scale) Free Energy of $\phi$ vs. time}
\end{subfigure}
\begin{subfigure}[t]{0.48\textwidth}
\includegraphics[width=1\textwidth]{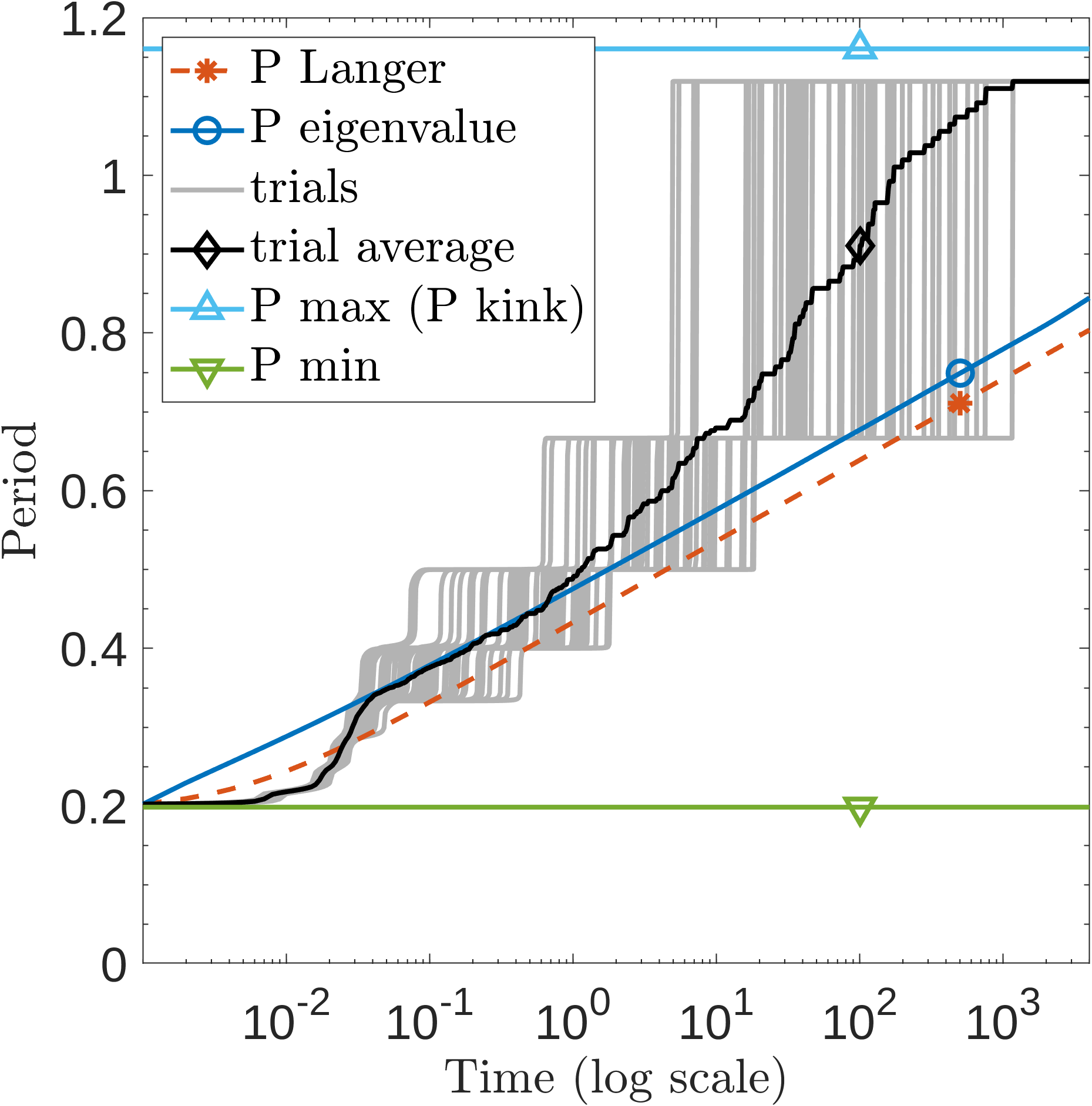}
\caption{\label{fig_per_av}  (linear-log scale) Period of $\phi$ vs. time.}
\end{subfigure}
\caption{\label{energies_av}
Free energy and periods of 50 trials of the (uncoupled) Cahn-Hilliard equation ($\kappa=0.001$, $\alpha=\beta=1$, $N=2^{11}=2048$, Domain: $[-1,1]$, $\Delta t=0.001$, final time $T=4000$), and the average of these, along with the maximum free energy \eqref{maximum-energy}, the kink energy \eqref{E_kink}, and the Langer and eigenvalue predictions.
In principle, the computational plots in Figure (b) should be generated by applying the pseudoinverse map 
from Definition \ref{pseudoinverse} to the plots in Figure (a), but in practice we found it 
more efficient to generate a map by interpolating periods as a function of energies. It is clear 
from our discussion in Section \ref{MeasureofCoursening} that the two approaches must give
nearly identical results.} 
\end{figure}

\subsection{Dynamics of the coupled vs. uncoupled systems}\label{coupled-system-section}

In this section, we examine the dynamics of the coupled system \eqref{pre-main} in comparison with the uncoupled system (that is, \eqref{pre-main} with $K=0$ and the convection term $v\phi_x$ removed).  The purpose is to get a better understanding of how coupling affects the dynamics, and 
especially the coarsening rates. In order to accomplish this, we will numerically generate solutions
to both the coupled and uncoupled systems, and carefully study differences between the resulting dynamics. 
As a starting point, we consider an example in which the phase variable $\phi$ is initiated 
as a small random perturbation of the homogeneous state $\phi \equiv 0$, and velocity is randomly
generated as well. (See the top left panel in Figure \ref{fig_coupled_sol1} for depictions of $\phi$
and $v$ at $t = 0$.) We show the time evolution of these solutions in Figure \ref{fig_coupled_sol1}, 
and the corresponding Fourier spectra in Figure \ref{fig_coupled_spec}. 

\begin{figure}[ht] 
\includegraphics[width=0.48\textwidth]{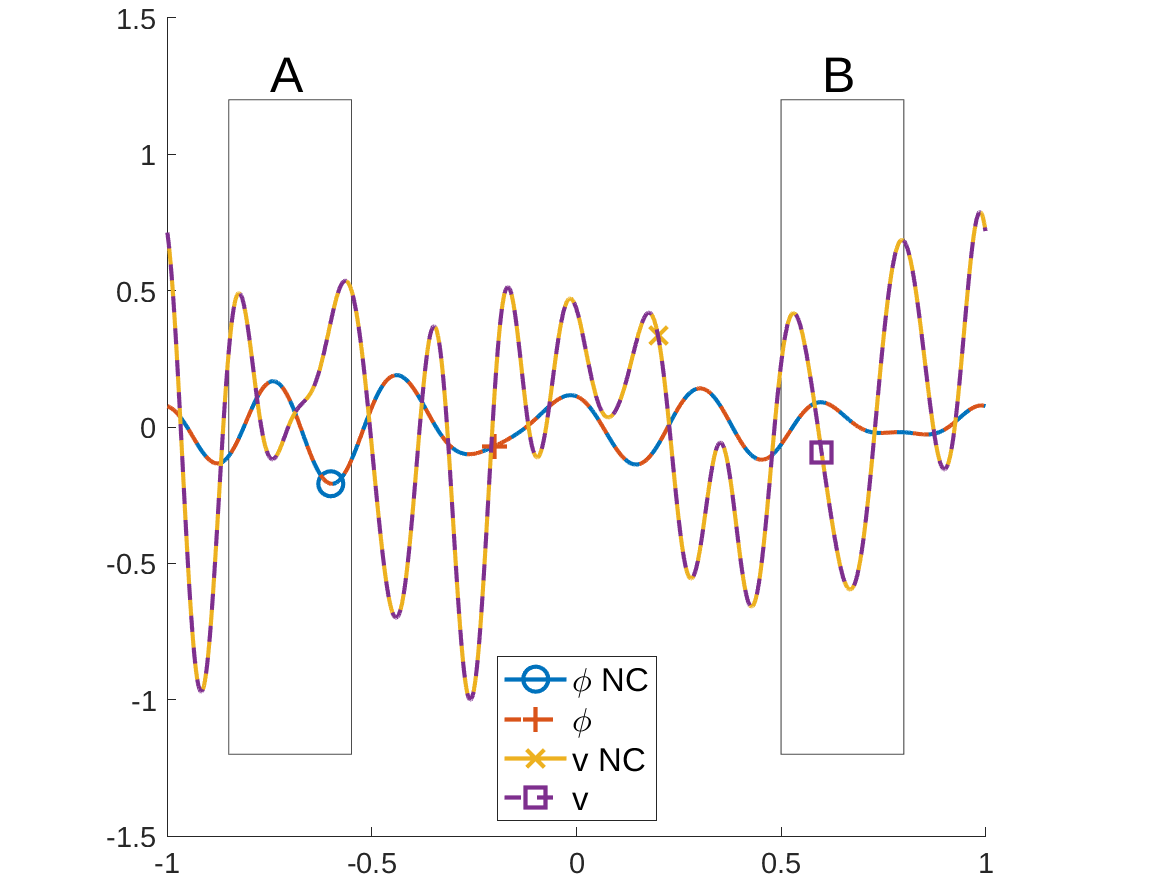}
\includegraphics[width=0.48\textwidth]{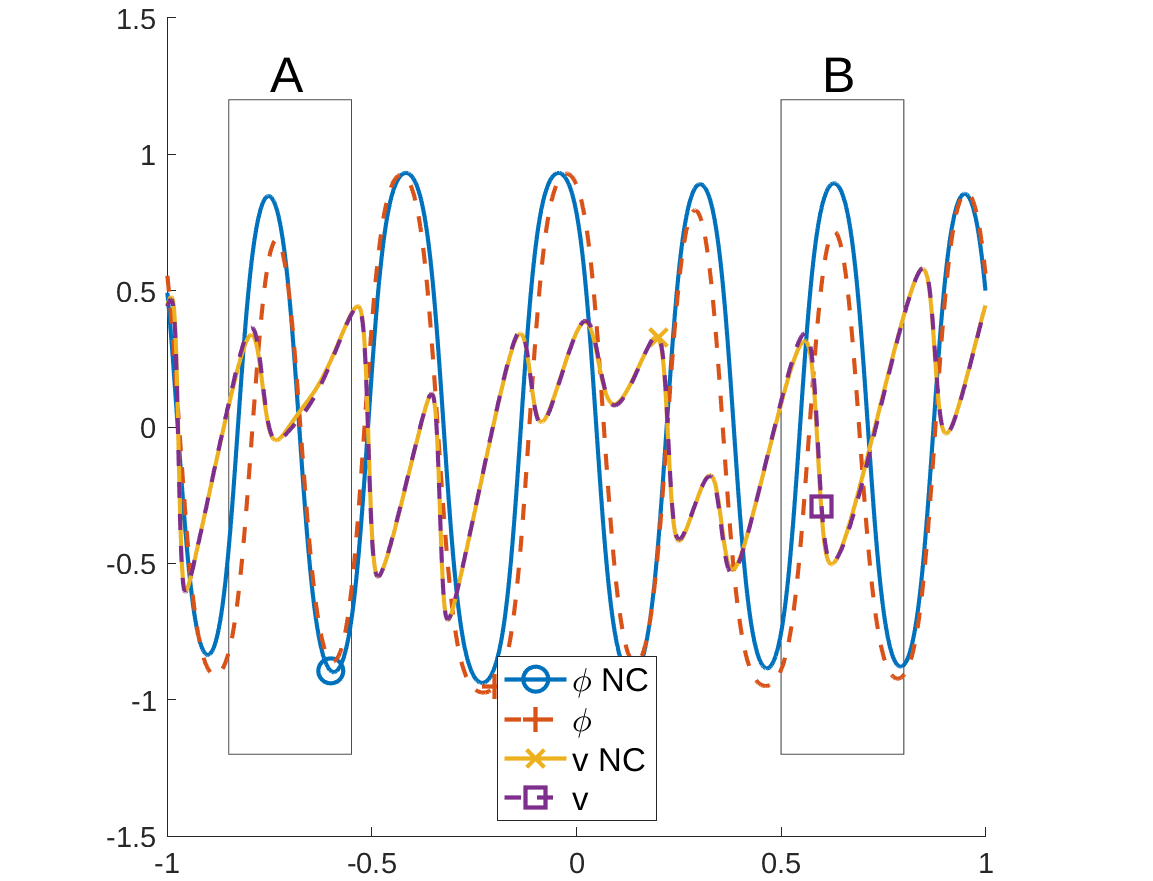}
\includegraphics[width=0.48\textwidth]{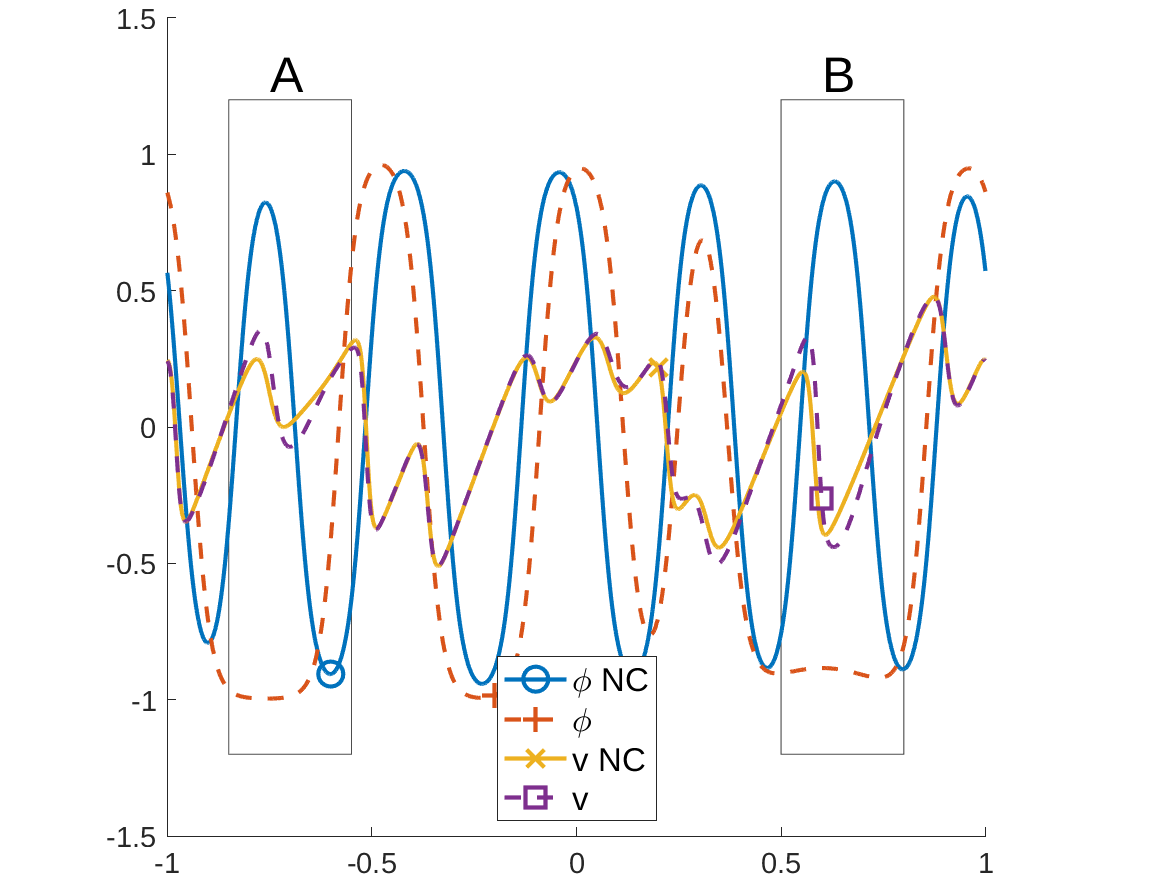}
\includegraphics[width=0.48\textwidth]{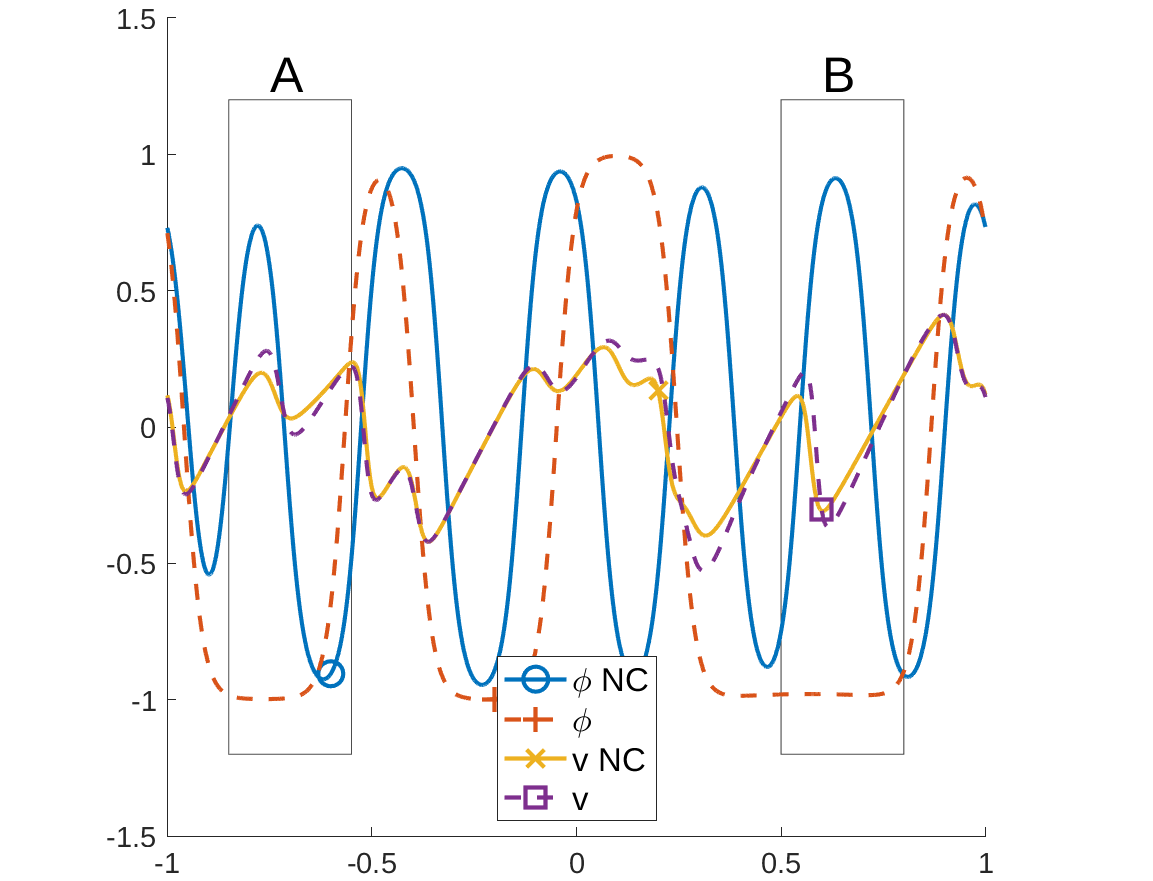}
\includegraphics[width=0.48\textwidth]{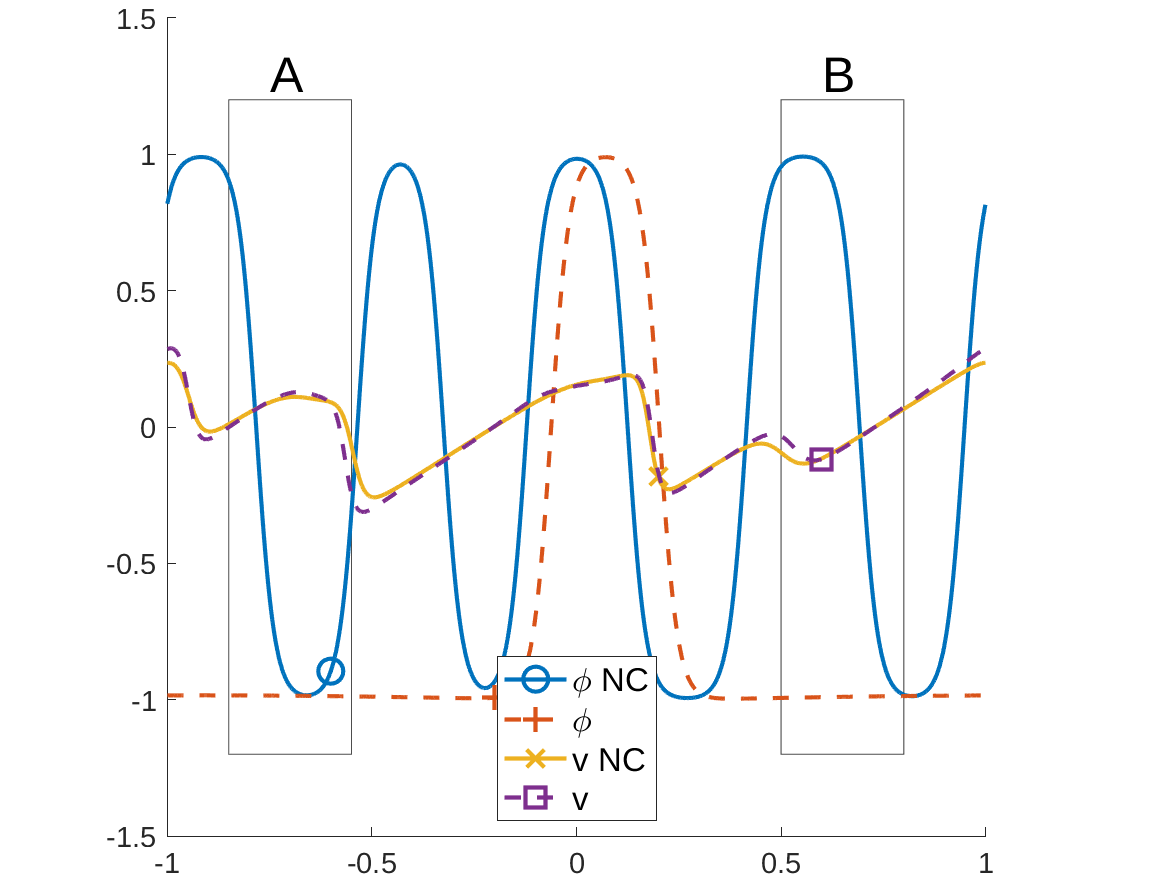}
\includegraphics[width=0.48\textwidth]{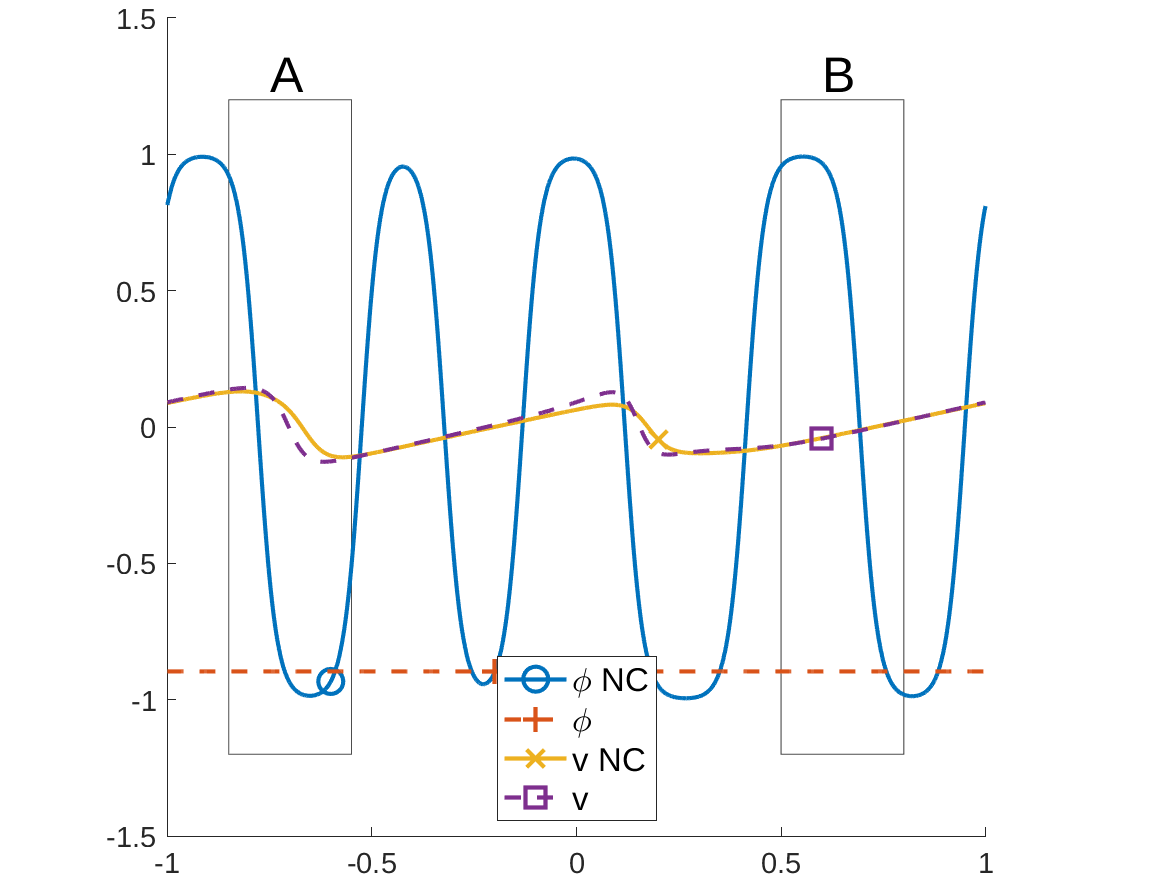}
\caption{\label{fig_coupled_sol1}Solution plots for \eqref{pre-main} over times $t\approx0.0$, $0.1$, $0.2$, $0.3$,  $1.0$, and $3.0$ (left to right and then top to bottom). ``NC'' indicates the system with no coupling (i.e., $K = 0$ and no convection). The horizontal axis is the spatial $x$-axis.
Solutions were computed with parameter values $\kappa = 0.001$, $\alpha = 1$, and $\beta = 1$, 
and in the case of coupling with additionally $\nu = 0.006$ and $K = 1$.}
\end{figure}

Qualitatively, the dynamics that we observe under coupling are driven by two primary mechanisms.
The first of these is the tendency of the phase variable $\phi (x,t)$ to rapidly evolve in such 
a way that enriched regions in which one component of the mixture dominates (i.e., $|\phi (x, t)| \cong 1$)
are separated by steep transition layers, and the second is the convective effect of the coupling term $v \phi_x$ on these
regions, namely if $v$ increases from one transition layer to the next then the trailing transition 
layer moves more slowly than the leading layer so that the layers separate, while if $v$ decreases
from one transition layer to the next a similar effect presses the layers closer together. Ancillary 
to this, the coupling term $K\mu \phi_x$ induces a transfer of energy from the phase variable 
to the velocity, especially pronounced when an enriched region is eliminated. 

Referring specifically now to the evolution depicted in Figure \ref{fig_coupled_sol1}, we will 
refer to the six sub-figures as \ref{fig_coupled_sol1}a through \ref{fig_coupled_sol1}f, labeled
left to right and then top to bottom. We see in Figure \ref{fig_coupled_sol1}b that coarsening
of the phase variable for the uncoupled and coupled systems are quite close for sufficiently 
short times. Nonetheless, if we focus on region A in Figure \ref{fig_coupled_sol1}b we see 
that for the coupled dynamics the negative velocity gradient (higher on the left phase transition 
than the right) compresses 
the positively enriched region to a lower wavelength and correspondingly reduced amplitude. 
At the same time, the positive velocity gradient between the left-most transition layer
and the transition layer at the far right of region A serves to pull the layers apart. In 
Figure \ref{fig_coupled_sol1}c we see that as a consequence of this dynamic, the positively
enriched portion of region A is entirely eliminated. Likewise, a similar dynamic occurs in region B,
and again a positively enriched region is entirely eliminated. In both cases, the elimination 
of transition layers removes energy from the phase variable, and we see from the increased
separation between the coupled and uncoupled velocity profiles that this energy is transferred
to the velocity. 

In Figure \ref{fig_coupled_sol1}c, we observe a similar dynamic, where in this case a negative 
velocity gradient compresses the sides of a negatively enriched region, which has then been 
eliminated in Figure \ref{fig_coupled_sol1}d. This same general dynamic occurs twice more in 
the transition from Figure \ref{fig_coupled_sol1}d to Figure \ref{fig_coupled_sol1}e, and one 
final time in the transition from Figure \ref{fig_coupled_sol1}e to Figure \ref{fig_coupled_sol1}f.
We notice in particular that while the uncoupled dynamics are extremely slow at this point, 
the coupled dynamics have completed their evolution. 

As a companion to Figure \ref{fig_coupled_sol1}, we include in Figure \ref{fig_coupled_spec}
corresponding spectral plots, verifying that the resolution of our discretization is sufficient.  

\begin{figure}[ht] 
\includegraphics[width=0.48\textwidth]{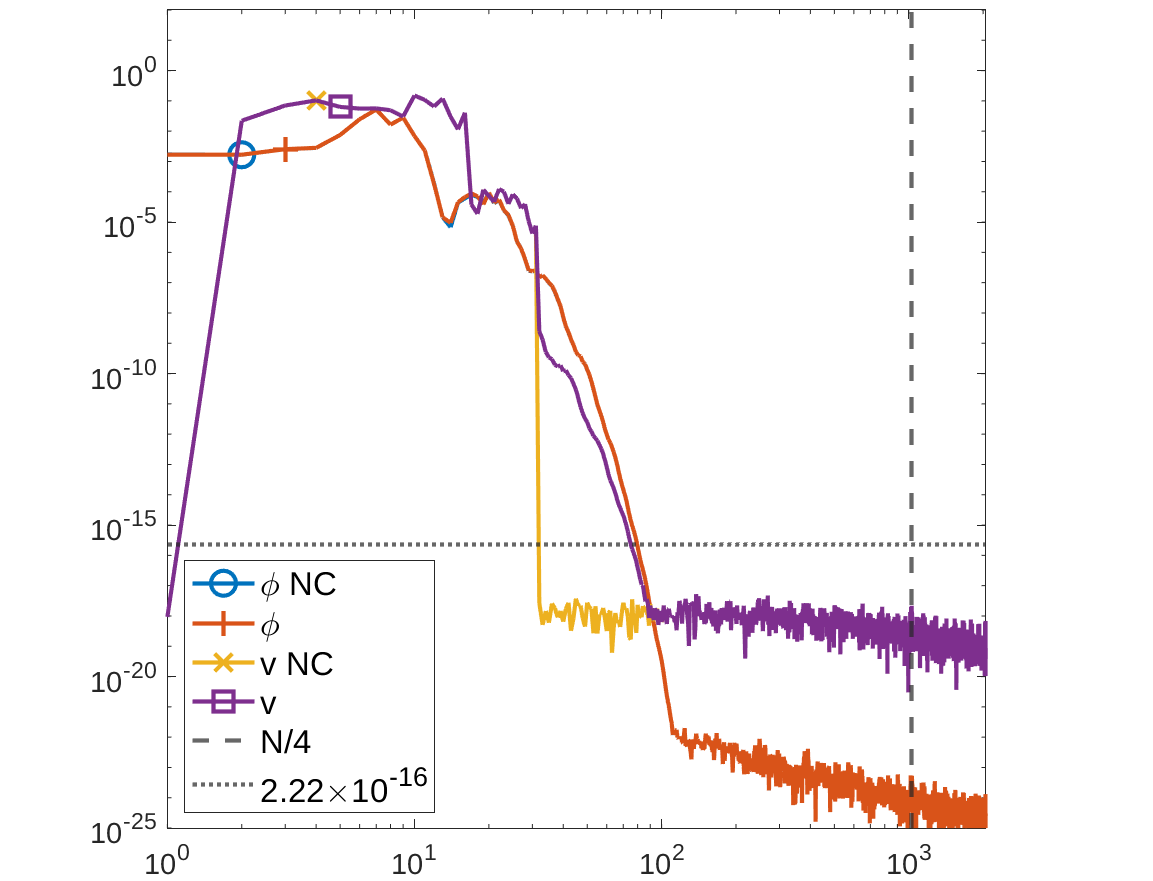}
\includegraphics[width=0.48\textwidth]{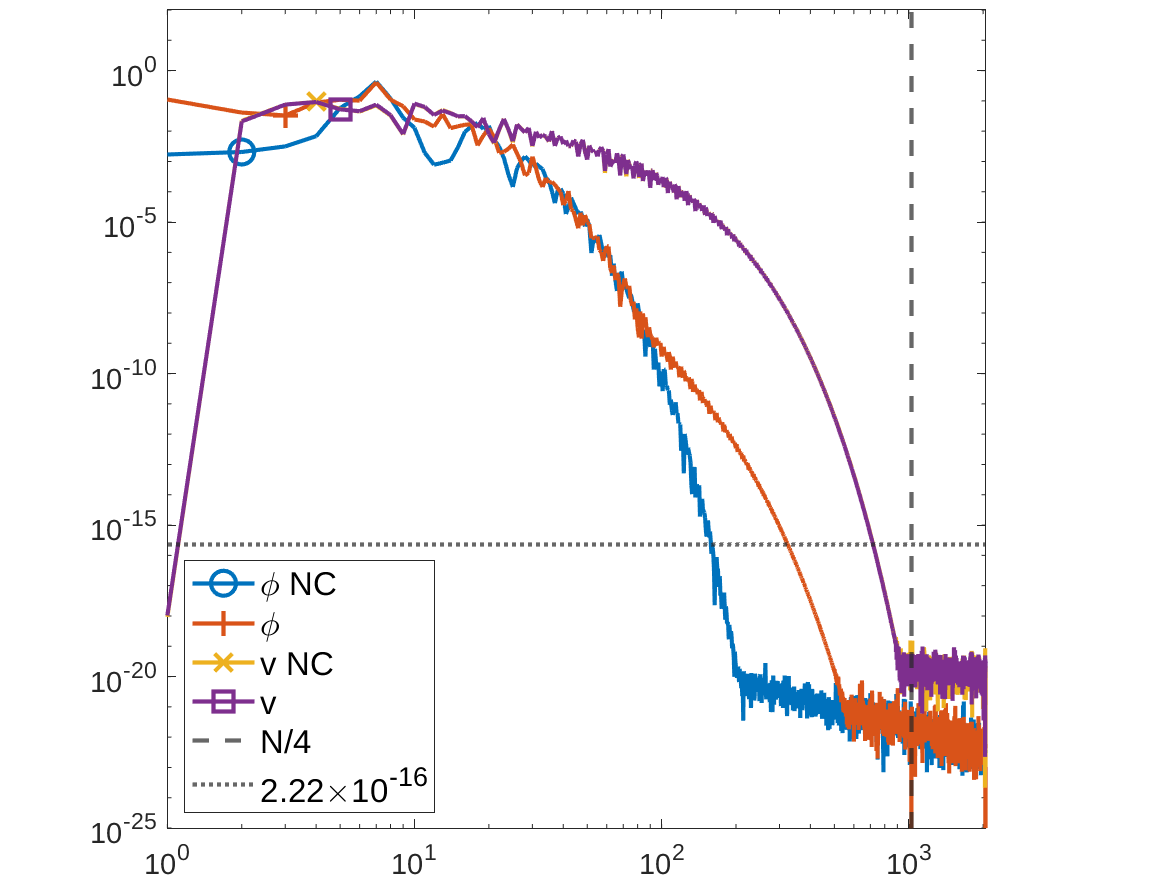}
\includegraphics[width=0.48\textwidth]{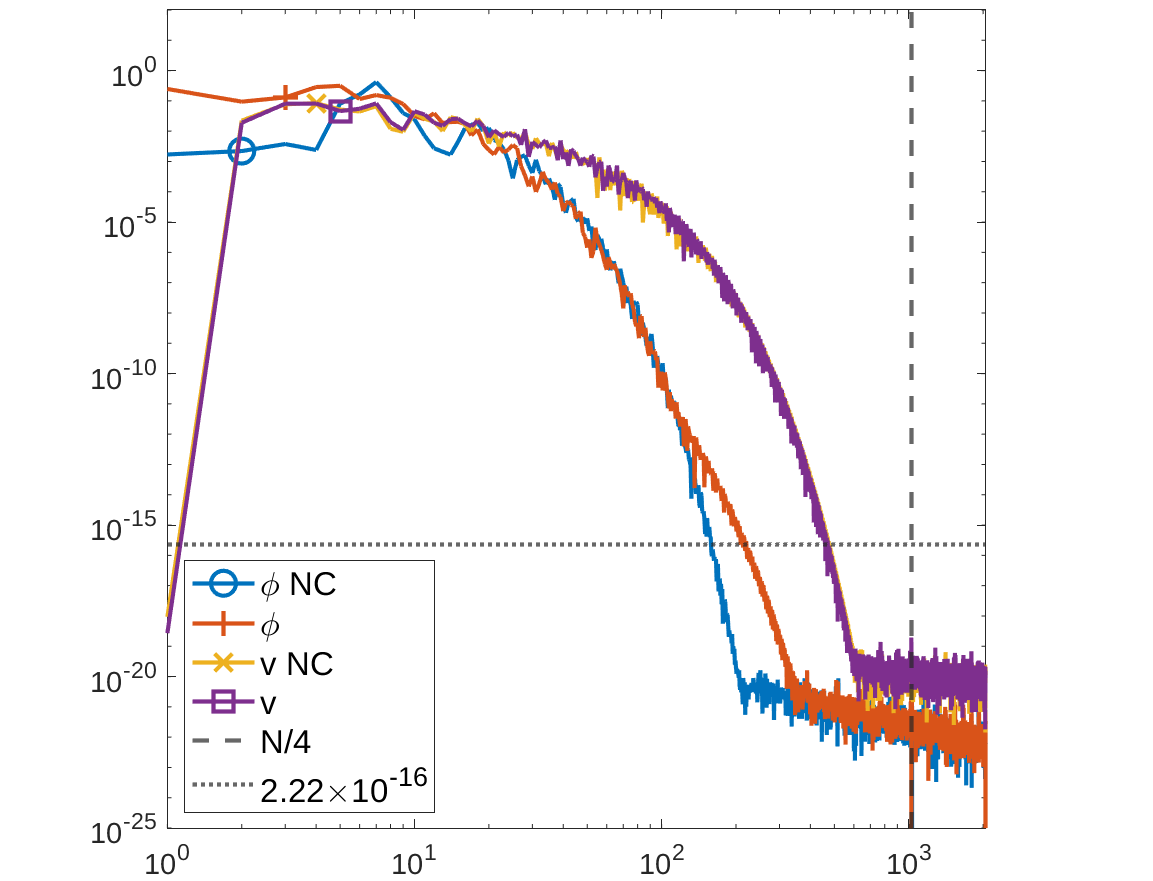}
\includegraphics[width=0.48\textwidth]{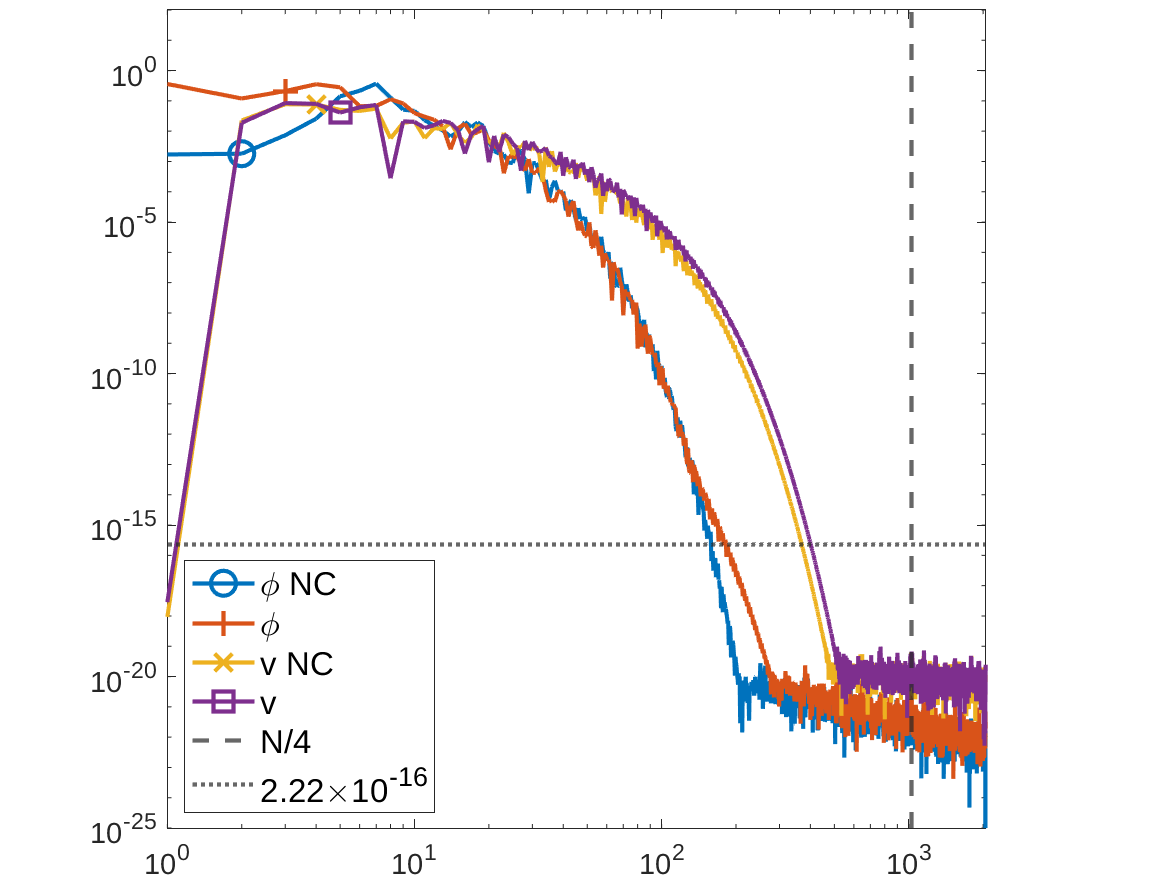}
\includegraphics[width=0.48\textwidth]{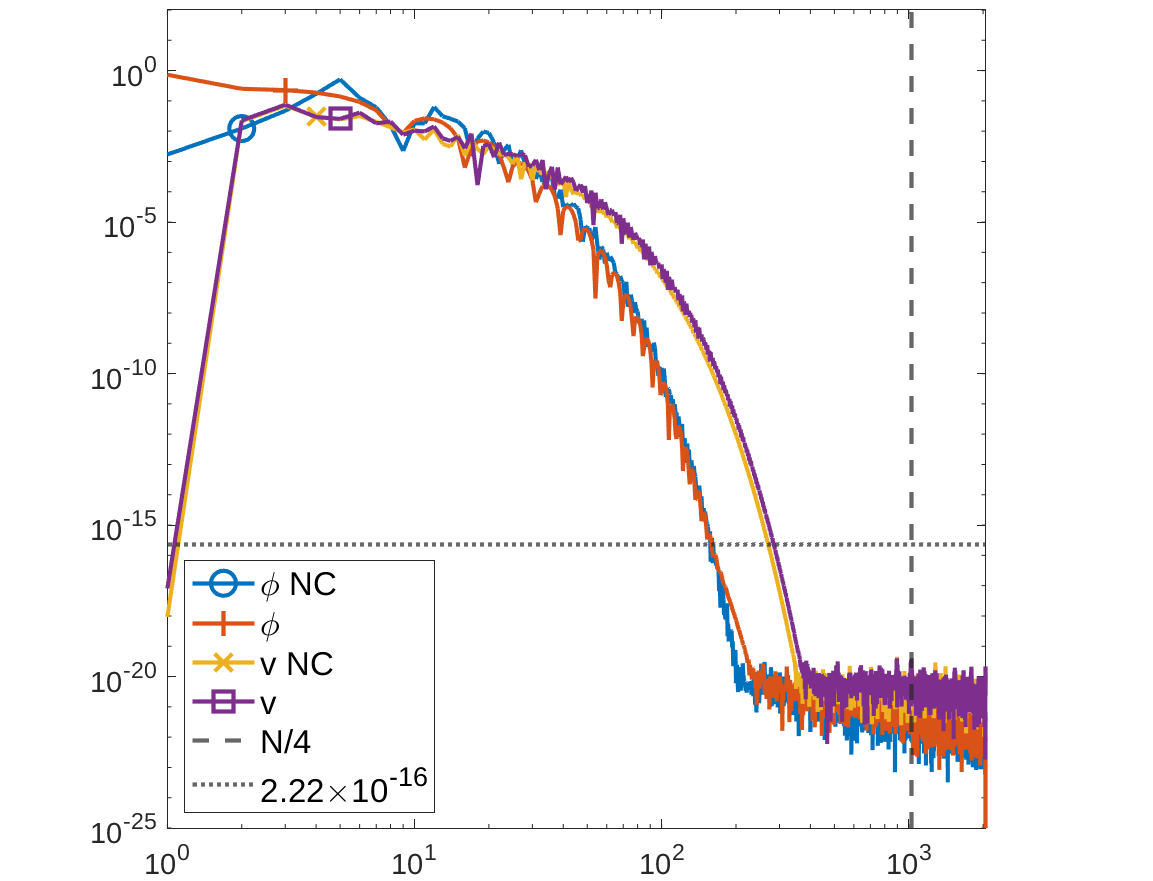}
\includegraphics[width=0.48\textwidth]{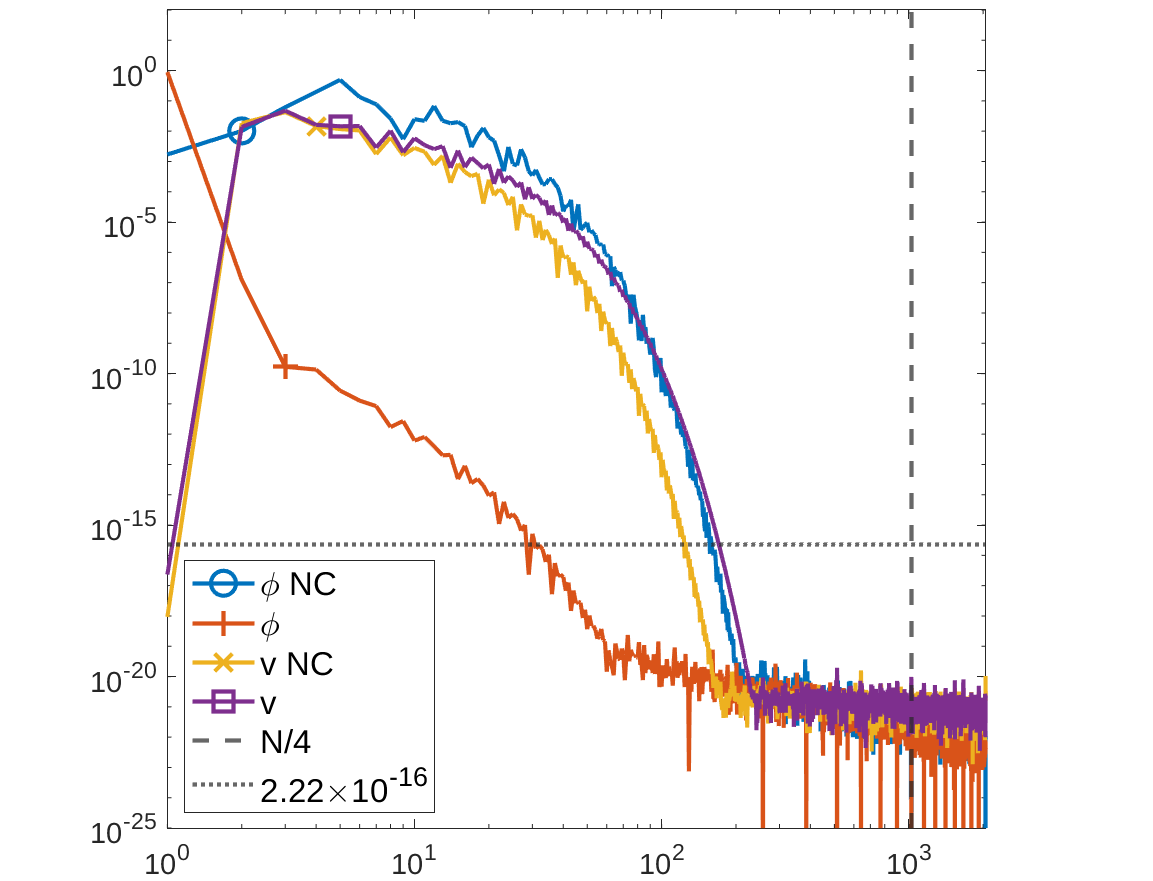}
\caption{\label{fig_coupled_spec}Spectrum plots over times $t\approx0.0$, $0.1$, $0.2$, $0.3$,  $1.0$, 
and $3.0$ (left to right and then top to bottom). ``NC'' indicates the system with no coupling.}
\end{figure}

Turning now to our second example, we again initiate the phase variable with a randomized 
perturbation of the homogeneous state $\phi \equiv 0$, but for this example we initialize the velocity profile 
with a simple zero-mass cycle. In this case, the dynamics are depicted in Figure \ref{fig_coupled_sol2}, 
where again we will use sub-labeling ``a'' through ``f'' to designate the images going left to right and top to 
bottom. As before, the initial dynamics are rapid in both the coupled and uncoupled cases, but in the 
coupled case a consistent positive velocity gradient is created, expanding the enriched regions, and 
since the velocity profile is nearly linear the expansion is greater for wider enriched regions. In 
Figure \ref{fig_coupled_sol2}b, we see four enriched regions, with narrower regions at the far left and 
far right and two wider regions in the middle. The wider regions dominate, and the narrower regions
collapse, leading to Figure \ref{fig_coupled_sol2}c, in which only two wider regions remain. At this point,
the region on the left is substantially wider than the region on the right, and dominates the dynamics
so that in Figure \ref{fig_coupled_sol2}d the enriched region on the right is clearly collapsing. Notably,
the velocity gradient is reduced by this point, and so the dynamics have slowed down, and it is not
until Figure \ref{fig_coupled_sol2}f (at time $t = 6.0$) when the enriched region on the right fully 
collapses and only a single enriched region remains. Again, with each loss of enrichment in the phase
variable, we see a transfer of energy to the velocity, as indicated by ridges forming between the 
coupled and uncoupled velocities. 

\begin{figure}[ht] 
\includegraphics[width=0.48\textwidth]{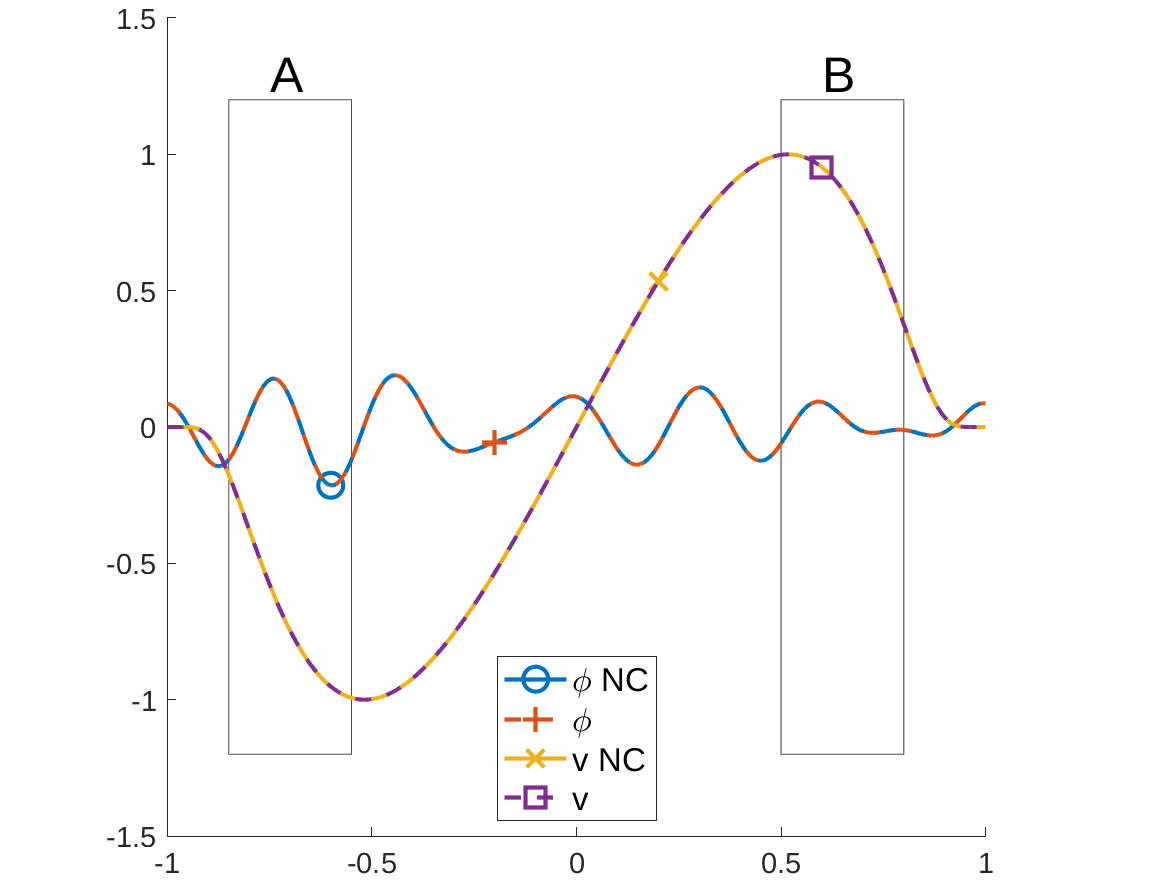}
\includegraphics[width=0.48\textwidth]{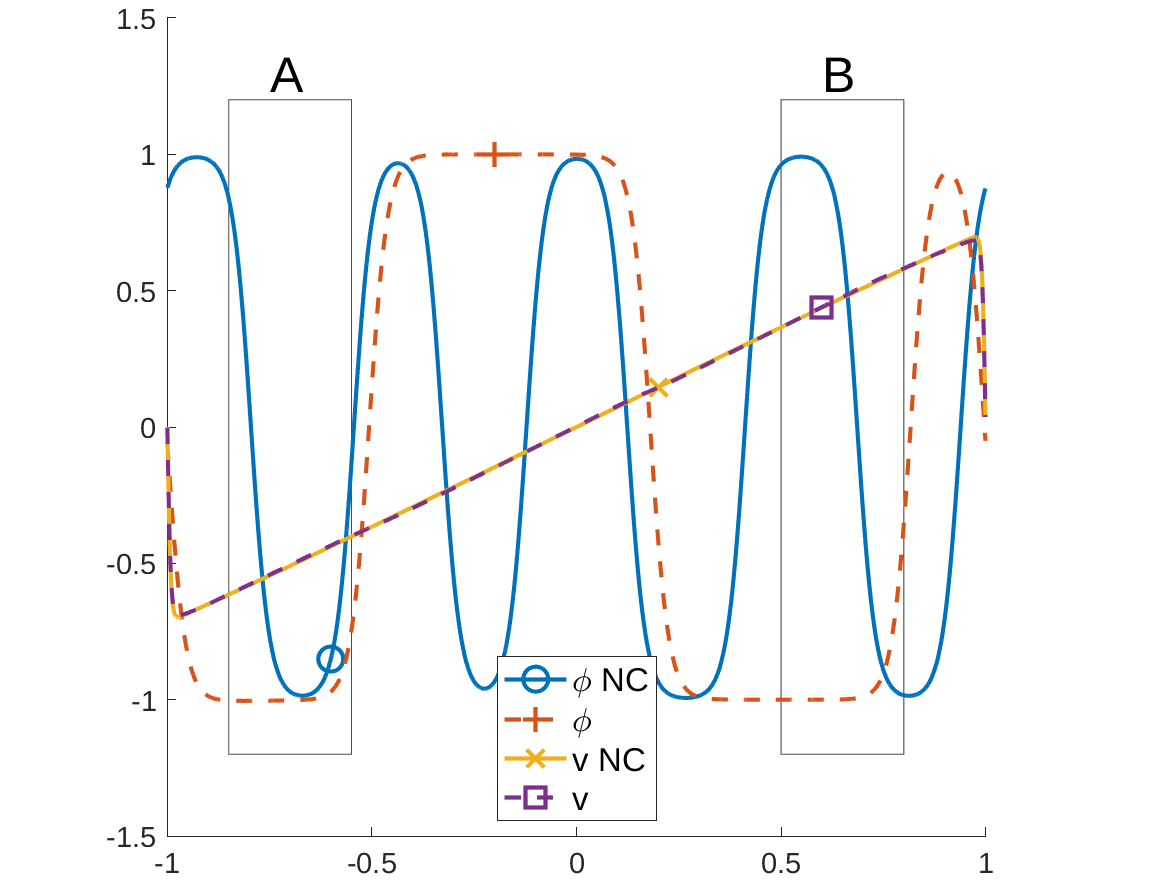}
\includegraphics[width=0.48\textwidth]{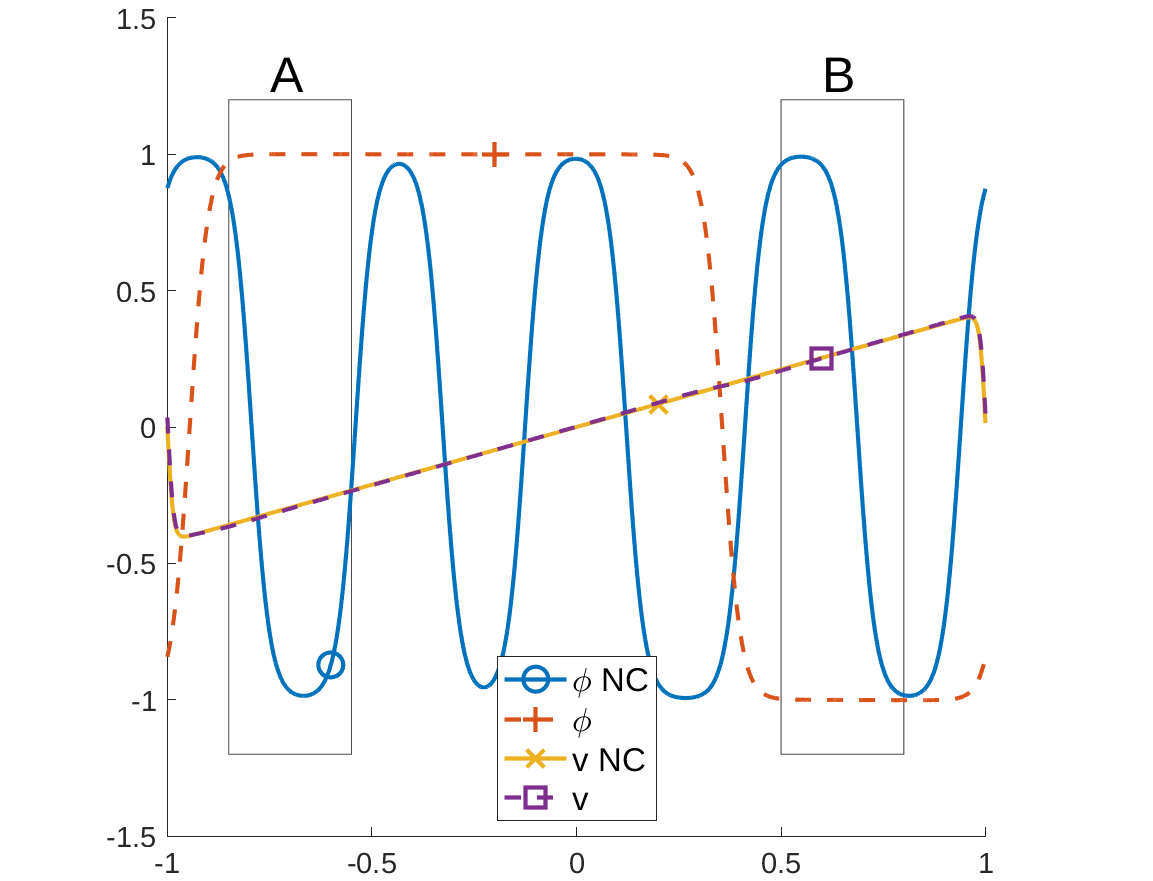}
\includegraphics[width=0.48\textwidth]{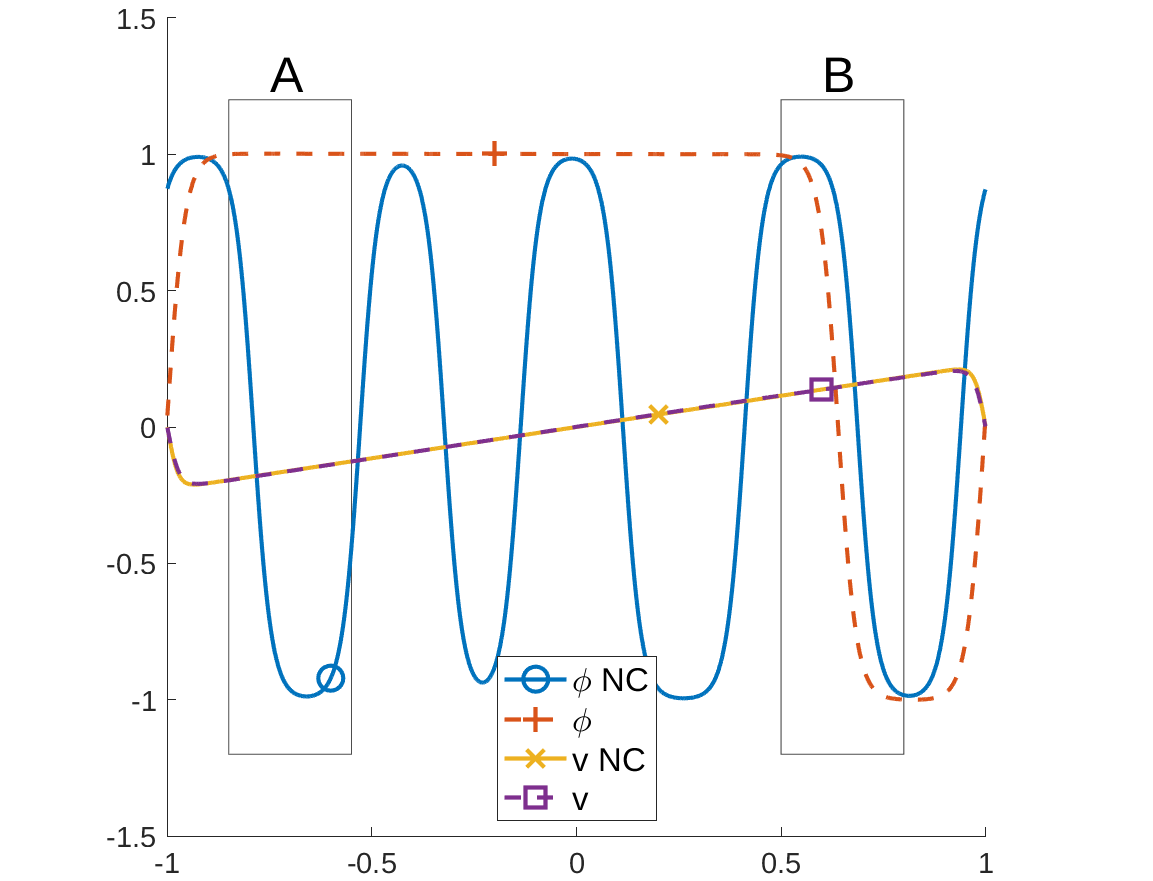}
\includegraphics[width=0.48\textwidth]{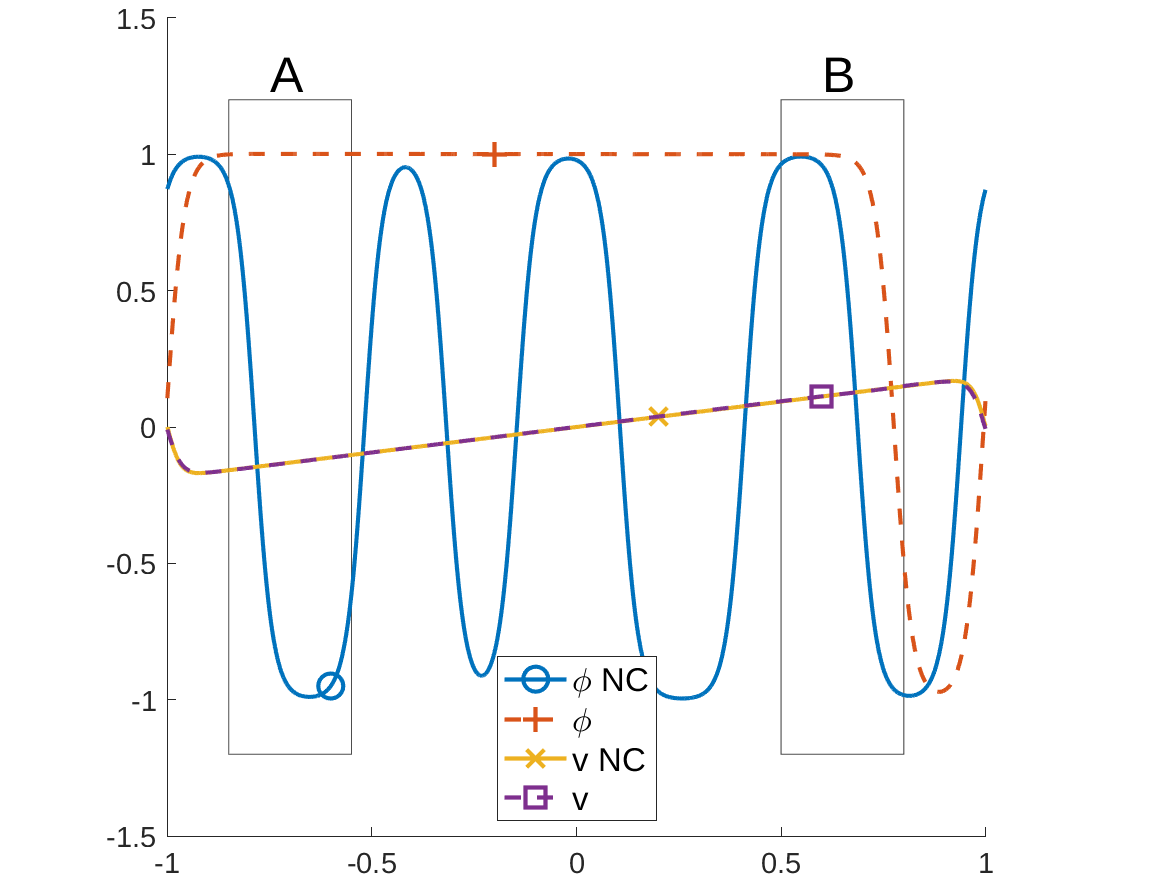}
\includegraphics[width=0.48\textwidth]{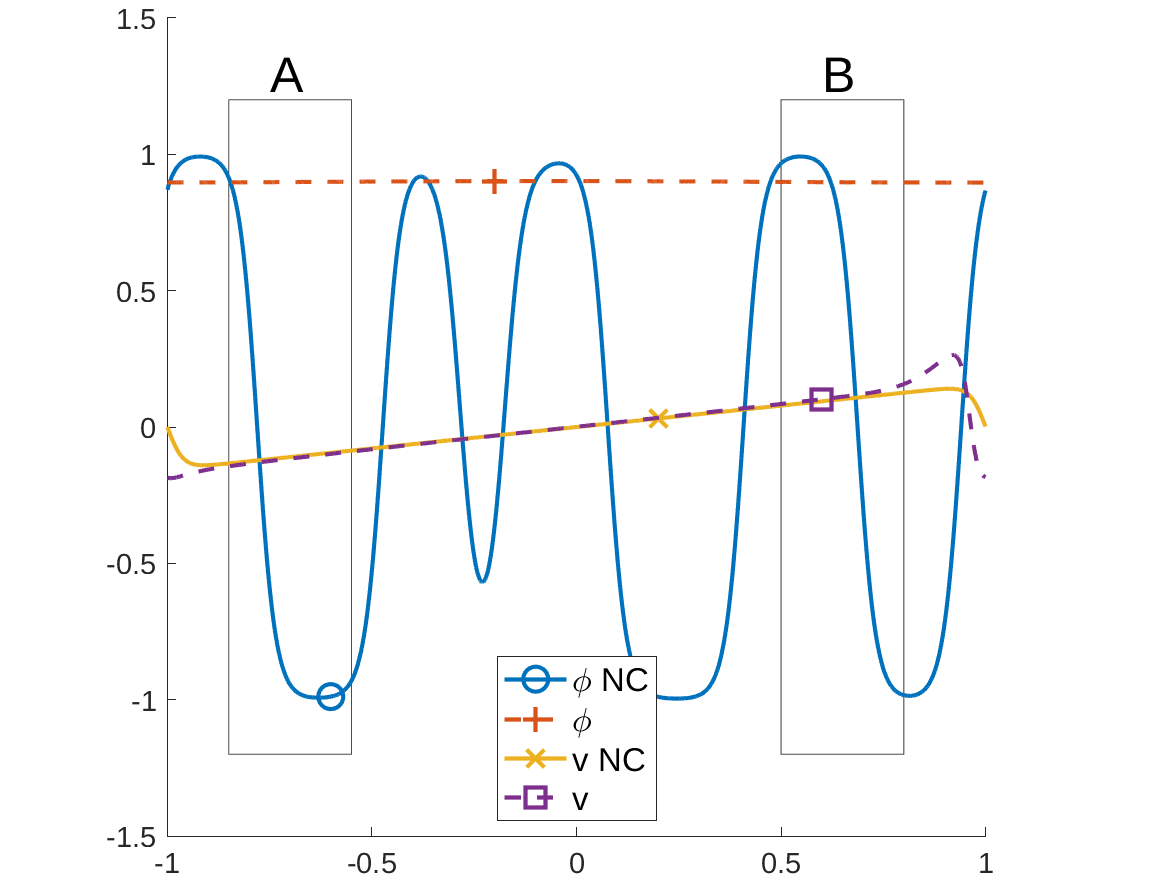}
\caption{\label{fig_coupled_sol2}Solution plots for \eqref{pre-main} with initial velocity $v_0(x)=Cx\exp(\frac{1}{x^2-1})$ over times for $t\approx0.0$, $1.0$, $2.0$, $4.0$,  $5.0$, and $6.0$ (left to right and then top to bottom). ``NC'' indicates the system with no coupling (i.e., $K = 0$ and no convection). The horizontal axis is the spatial $x$-axis. Solutions were computed with parameter values $\kappa = 0.001$, $\alpha = 1$, and $\beta = 1$, 
and in the case of coupling with additionally $\nu = 0.006$ and $K = 1$.}
\end{figure}

The dynamics depicted in Figures \ref{fig_coupled_sol1} and \ref{fig_coupled_sol2} 
suggest that coupling increases the rate of coarsening, and we would next like to 
quantify this effect. In Figure \ref{fig_periods_all}, we plot coarseness values
as a function of time for both the coupled and uncoupled equations, and note from 
the semilog plot on the left that while the rate of coarsening is significantly 
faster in the coupled case, it still appears to be at a logarithmic rate. 
In order to make this observation more precise, we consider a slight adjustment 
to Langer's formula \eqref{langer_period}, keeping the logarithmic form,
\begin{align}\label{P_fit}
P_\text{fit}[c_1,c_2](t) &= p_0 + c_1\sqrt{\frac{2\kappa}{\beta}}\log\left(1+\frac{t}{c_2}\frac{16\beta^2}{\kappa}\exp\left(-\frac{p_0}{\sqrt{2\kappa/\beta}}\right)\right),
\end{align}
where $c_1$ and $c_2$ can be found by optimizing a log-weighted $L^2$ norm of the error on the time interval $[0,20]$; that is,
\begin{align}\label{argmin}
(c_1,c_2) = \argmin_{c_1,c_2}\int_0^{20} |P_\text{fit}[c_1,c_2](t) - P(t)|^2\,\frac{dt}{\log(1+t)}.
\end{align}
For the case of initial velocity given by random Fourier coefficients (as described above), we found $c_1\approx5.70901$, $c_2\approx5.6991$, while for the case of initial velocity given by \eqref{bump}, we found $c_1\approx6.11532$, $c_2\approx4.92613$, indicating a very significant increase in the coarsening rate.  The results are depicted in Figure \ref{fig_periods_all}, where one can see that, for example, a period of roughly $p\approx1.12$ was reached using the bump-function initial velocity by time $t\approx0.25$, by the random Fourier coefficient initial data by time $t\approx0.34$, but the uncoupled equation does not reach this period until time $t\approx76.0$.  Moreover, the next ``merger'' corresponding to a period of $p\approx 1.495$ was reached using the bump-function initial velocity by time $t\approx1.52$, by the random Fourier coefficient initial data by time $t\approx3.43$, but the uncoupled equation does not reach this period even by time $t=2000$ (!).  In fact, extrapolating using \eqref{langer_period}, $p\approx 1.495$ will not be reached by the uncoupled system until roughly $t\approx2.06\times 10^{10}$ (!!).

\begin{figure}[ht] 
\begin{center}
\includegraphics[width=0.48\textwidth]{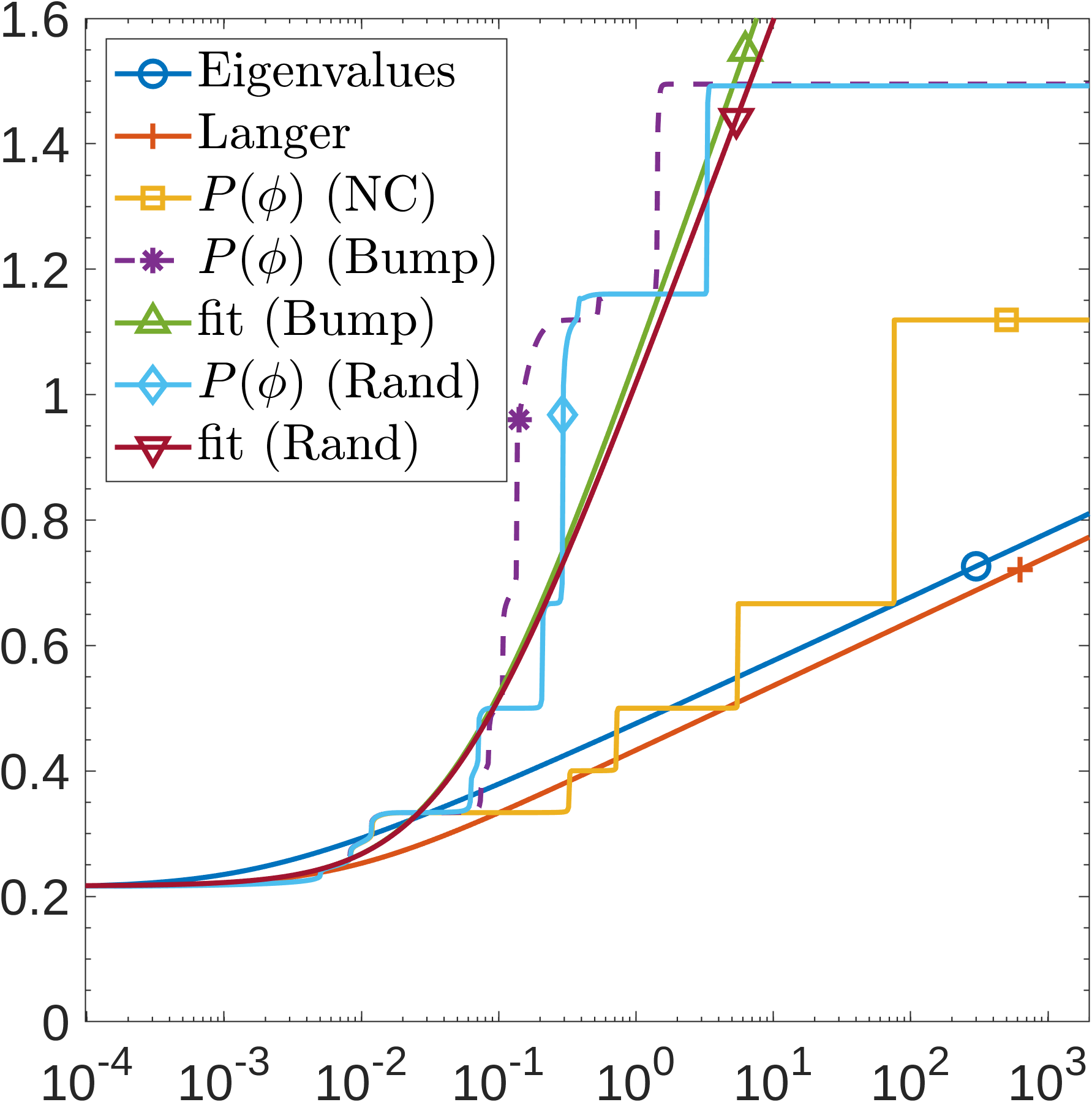}
\includegraphics[width=0.495\textwidth]{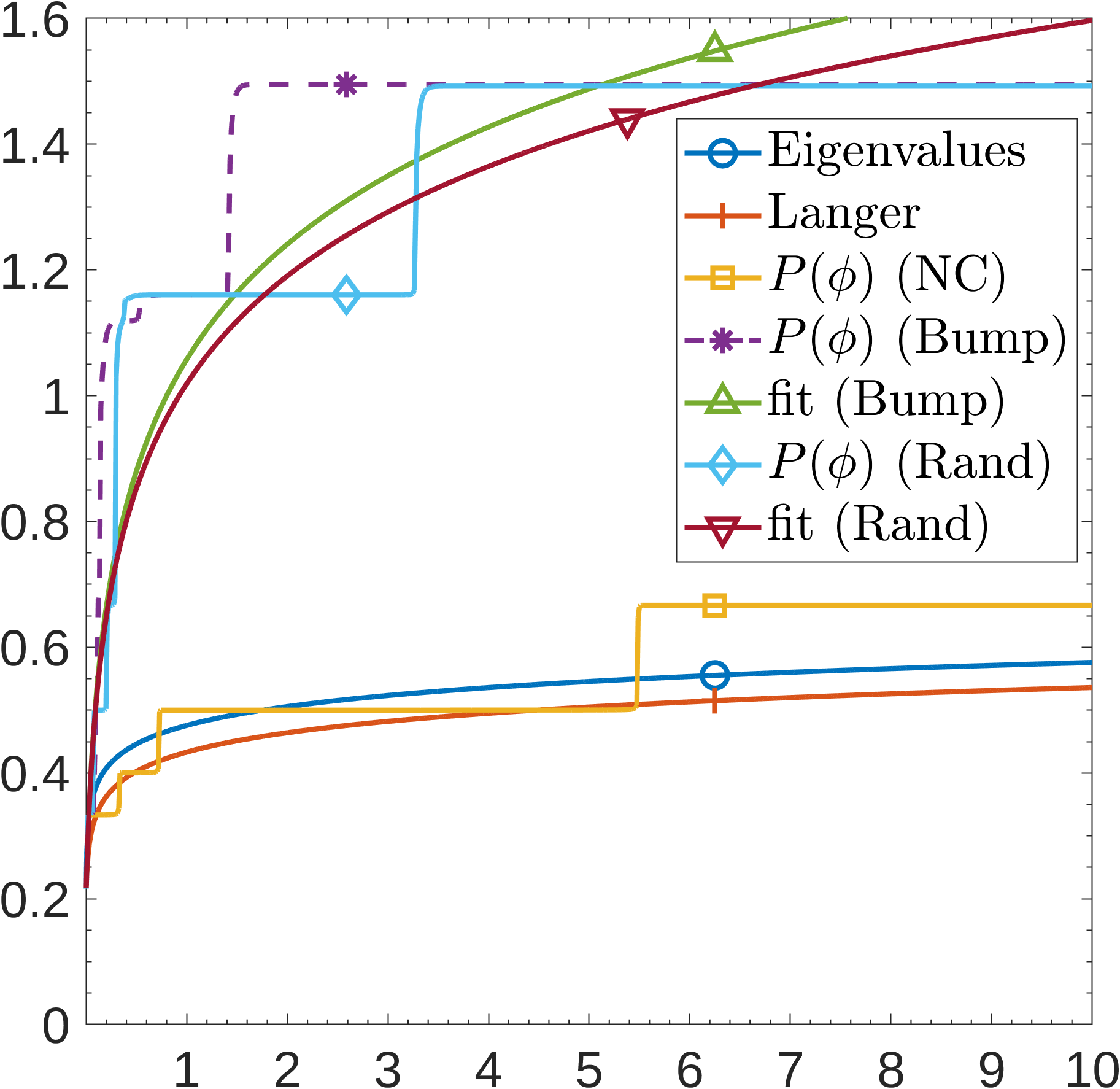}
  \end{center}
\caption{\label{fig_periods_all} Evolution of periods via direct simulation, and also predictions by Langer's method and the eigenvalue method.  The ``fit'' functions are given by \eqref{P_fit} with coefficients $c_1$ and $c_2$ determined by \eqref{argmin}. (left) Time interval $[0,2000]$ on linear-log plot. (right) Zoom in on time interval $[0,10]$ on linear-linear plot.}
\end{figure}

\FloatBarrier

We conclude this section with a comparison of energy dynamics for the coupled and 
uncoupled systems. For this, we compute energies associated with the solutions 
depicted in Figure \ref{fig_coupled_sol1}, along with $\|v\|_{L^2}$, which serves
as a measure of kinetic energy stored in the velocity (see Figure \ref{fig_norms}a). 
As expected, we see that for short times the energies of the coupled and uncoupled 
systems remain quite close, but that they separate substantially starting near $t = 0.3$, which 
corresponds with Figure \ref{fig_coupled_sol1}d. 
After separation, the energy for the coupled system decays much more rapidly than
the energy for the uncoupled equation. We observe that each sharp decline in the 
Cahn--Hilliard energy is accompanied by an incline in the kinetic energy, due to 
the energy transfer. In Figure \ref{fig_norms}b, we include $H^1$ norms of $\phi$
and $v$ in both the coupled and uncoupled cases. 

\begin{figure}[ht] 
\begin{subfigure}[t]{0.49\textwidth}
\includegraphics[width=1\textwidth]{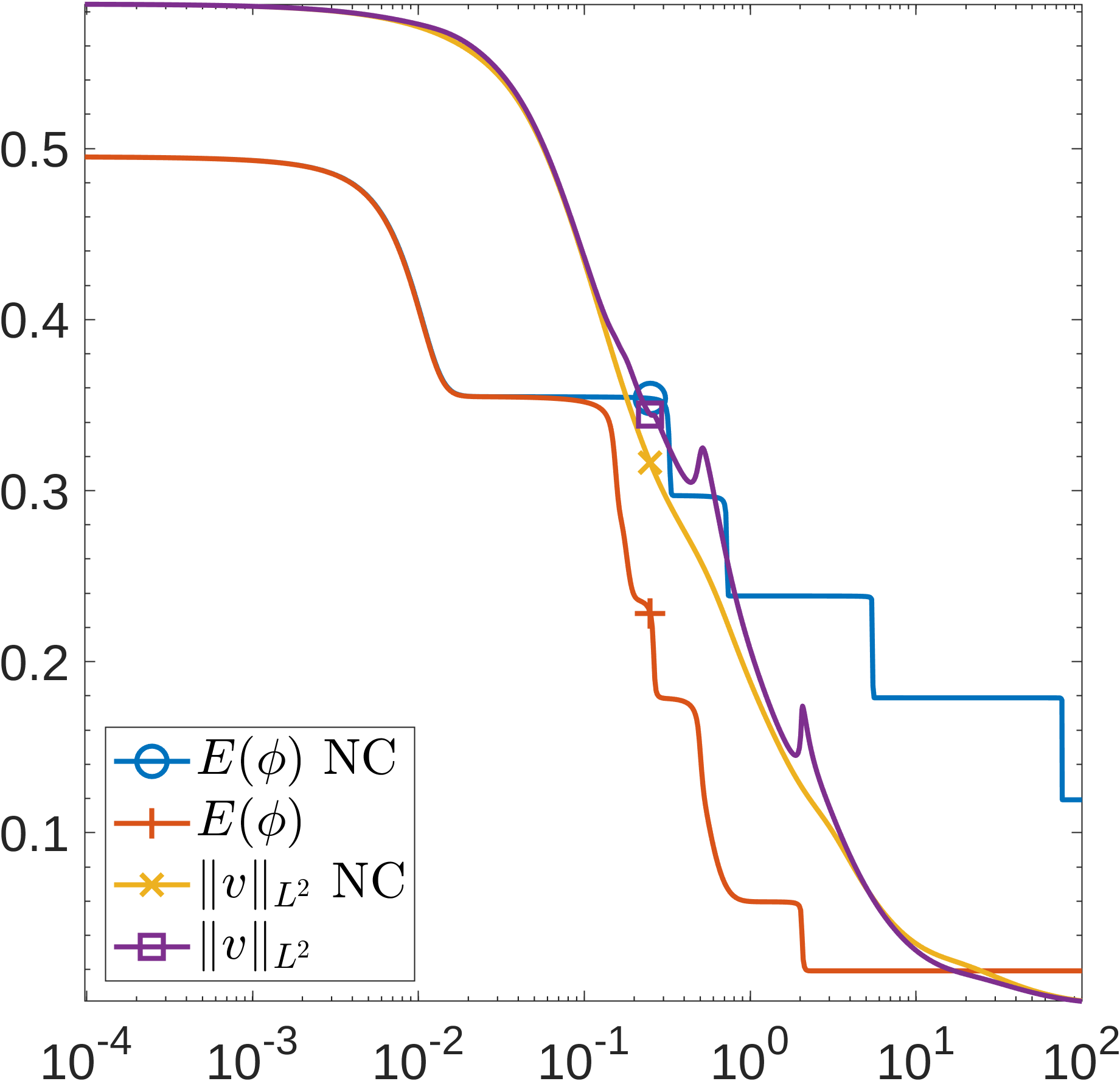}
\caption{\label{fig_En} (linear-log scale) Free Energy of $\phi$ vs. time, and also the kinetic energy of $v$ vs. time. The coupled solution injects energy into the velocity precisely when the free energy decays.}
\end{subfigure}
\hfill
\begin{subfigure}[t]{0.48\textwidth}
\includegraphics[width=1\textwidth]{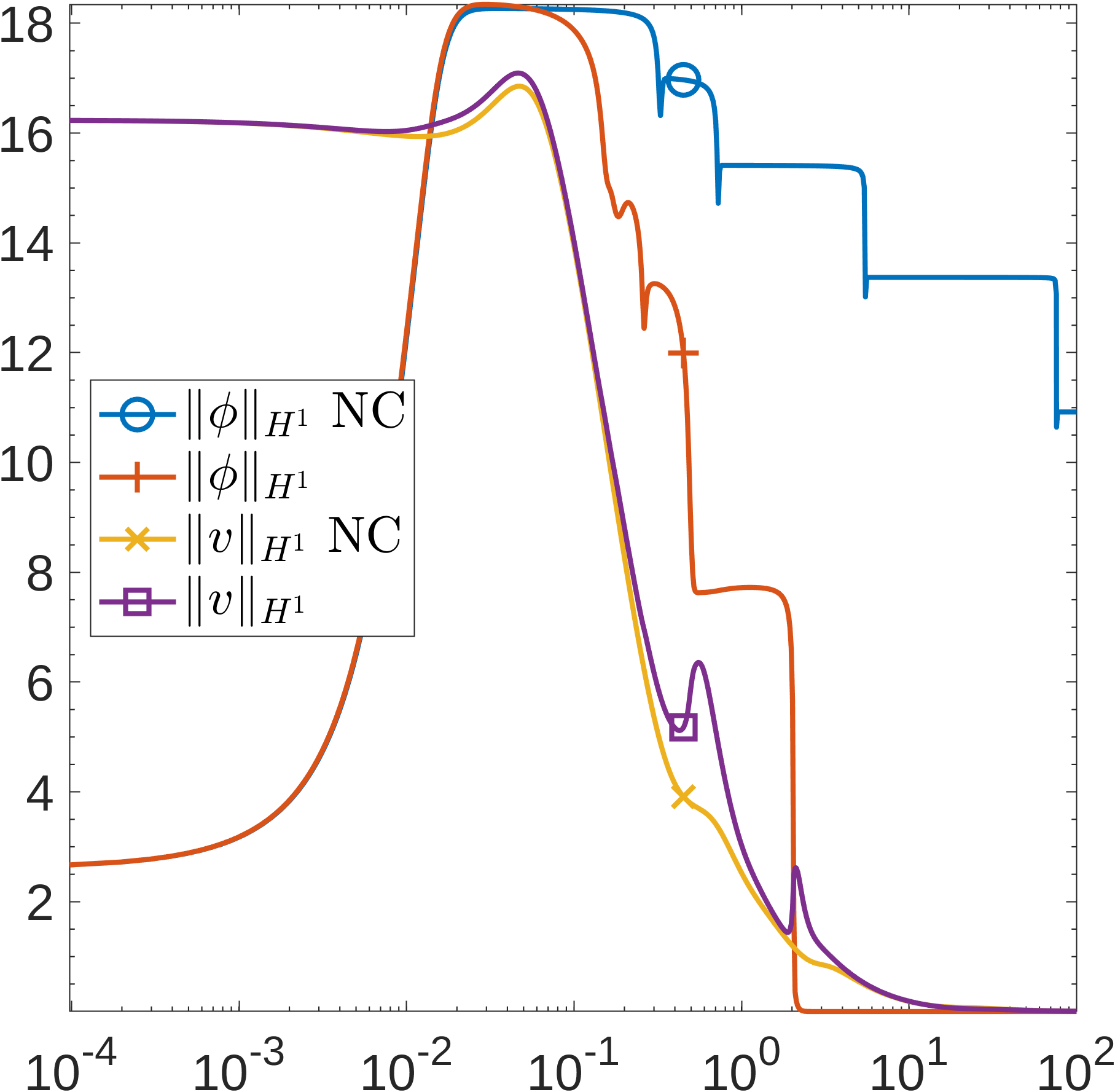}
\caption{\label{fig_H1}  (linear-log scale) $H^1$ semi-norms vs. time. Similar to the free energy in Figure \ref{fig_En}, the $H^1$ of the coupled solution decays much faster than that of the uncoupled solution.}
\end{subfigure}
\caption{\label{fig_norms} Free energy and various norms ($y$-axis) vs. time ($x$-axis) of the system with no coupling (NC) and coupled solutions of \eqref{pre-main} for the simulation described above.} 
\end{figure}

\FloatBarrier

\subsection{Divergence-form coupling}
We briefly consider the divergence forms \eqref{pre-main_div_form1} and \eqref{pre-main_div_form2}.  
Simulations\footnote{These simulations were run with a nearly identical setup to those above, but required slightly higher resolution ($N=4096$), and hence  different randomized initial data for $\phi$, since the random initial data for $\phi$ was generated at each grid point.} of \eqref{pre-main_div_form2} produced qualitatively similar results to those of \eqref{pre-main}, hence we do not show the evolution of the dynamics for the interest of brevity.  However, the decay rates were somewhat slower for solutions to the divergence-form equations as can be seen in Figure \ref{fig_norms_div}.  

\begin{figure}[ht] 
\begin{subfigure}[t]{0.49\textwidth}
\includegraphics[width=1\textwidth]{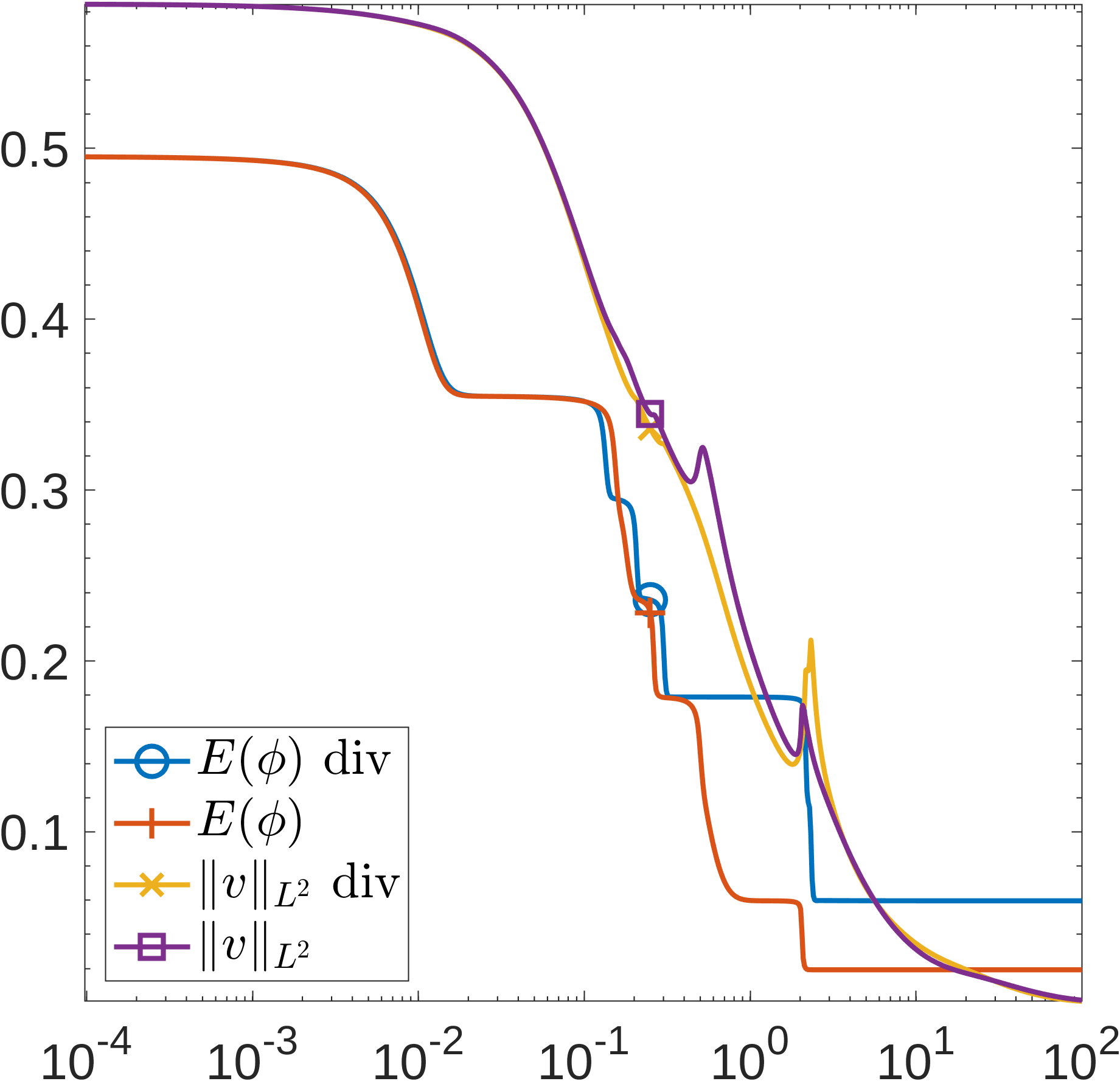}
\caption{\label{fig_En_div} (linear-log scale) Free Energy of $\phi$ vs. time, and also the kinetic energy of $v$ vs. time. The energy injection into the velocity appears to be less pronounced in the divergence-form coupling.}
\end{subfigure}
\hfill
\begin{subfigure}[t]{0.48\textwidth}
\includegraphics[width=1\textwidth]{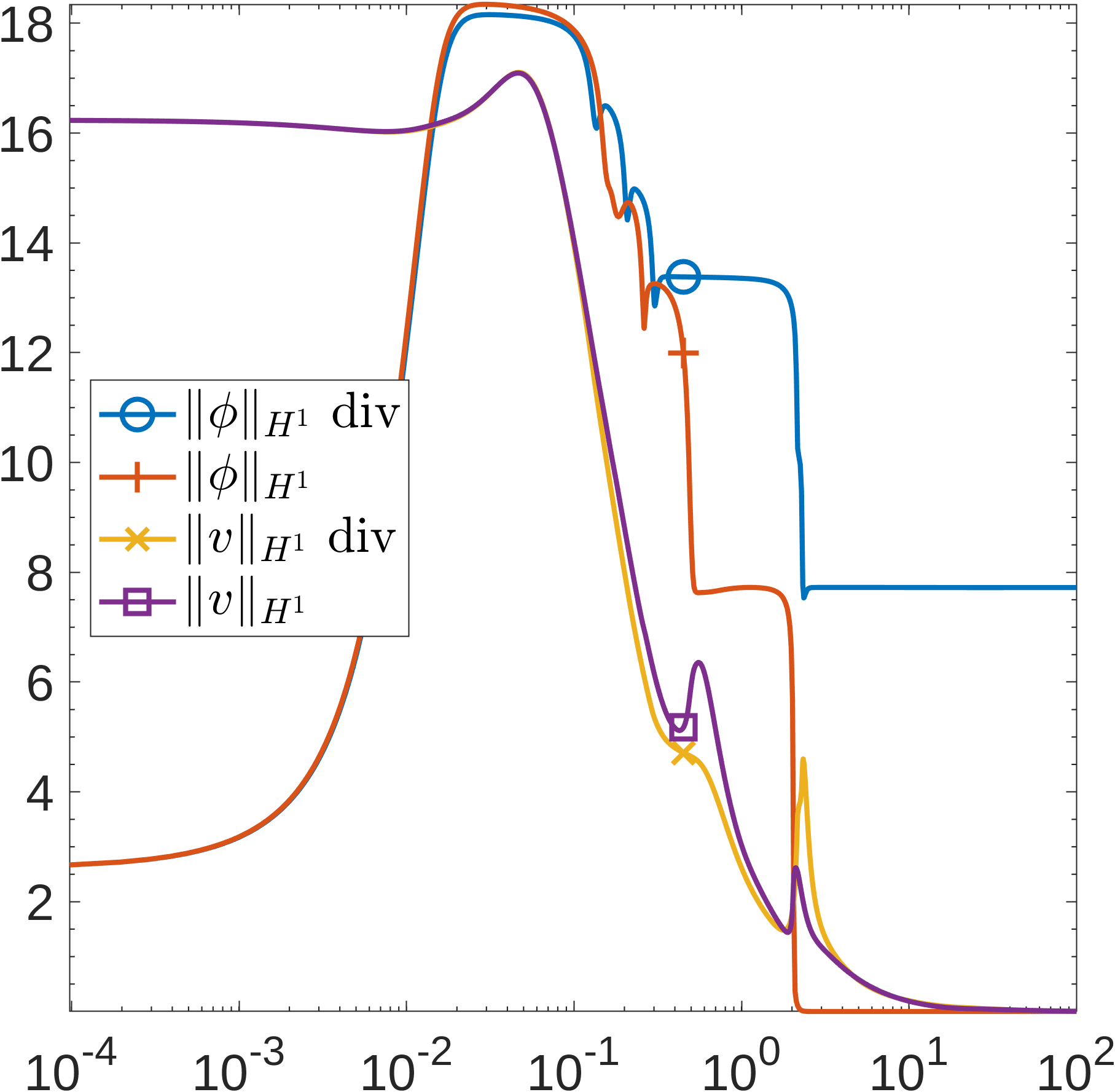}
\caption{\label{fig_H1_div}   (linear-log scale) $H^1$ semi-norms vs. time. Similar to the free energy in Figure \ref{fig_En_div}, the $H^1$ norm of the coupled solution decays much faster than that of the divergence-form coupling.}
\end{subfigure}
\caption{\label{fig_norms_div} Free energy and various norms ($y$-axis) vs. time ($x$-axis) of the divergence form coupling \eqref{pre-main_div_form1} (div) and the advective form \eqref{pre-main} for the simulation described above.}
\end{figure}

\FloatBarrier

\section{Qualitative Properties} \label{chns_section}

We now turn to analysis of qualitative properties of 
\eqref{pre-main}, beginning with a consideration 
of energy. Specifically, we consider $F(\phi)$ in the form \eqref{quarticF}, 
and take $M \equiv 1$, giving the system 
\begin{subequations}
\label{main}
\begin{align}
\label{main_phi}
\phi_t + v \phi_x &= \mu_{xx},
\\\label{main_mu}
\mu &= -\kappa \phi_{xx} + \alpha \phi^3 - \beta \phi,
\\\label{main_v}
v_t + v v_x &= \nu v_{xx} + K \mu \phi_x.
\end{align}
\end{subequations}

In this section, we prove the global existence, uniqueness, and continuous dependence on initial data of the solution to this system. 

\subsection{Energy Balance} \label{energy_balance_section}
To illustrate ideas, we first sketch a basic energy estimate formally before proving our results rigorously in subsection \ref{global_section}.
Taking the (formal) inner product of the momentum equation 
in \eqref{main} with $v$ we obtain 
\begin{equation} \label{from_momentum}
\frac{1}{2} \frac{d}{dt} |v|^2 + \nu |v_x|^2 
=  K\langle \mu \phi_x, v \rangle,
\end{equation}
where $|\cdot|$ and $\langle \cdot, \cdot \rangle$ respectively
denote the standard norm and inner product on $L^2 (-L, L)$, and we have used 
periodicity to compute 
\begin{equation*}
\int_{-L}^{+L} v^2 v_x dx = \int_{-L}^{+L} (\tfrac{1}{3} v^3)_x dx
= 0.
\end{equation*} 

For the energy \eqref{ch_energy}, formally integrating by parts, we compute 
\begin{equation*}
\begin{aligned}
\frac{dE}{dt} &= 
\int_{-L}^{+L} F' (\phi)\phi_t + \kappa \phi_x \phi_{xt} dx
= \int_{-L}^{+L} (-\kappa \phi_{xx} + F'(\phi)) \phi_t dx \\
&= \int_{-L}^{+L} \mu (\mu_{xx} - v \phi_x) dx
= -|\mu_x|^2 - \langle \mu \phi_x, v \rangle.
\end{aligned}
\end{equation*}
(The same result may be formally obtained by taking an inner-product of the phase equation with $\mu$.)
Multiplying by $K$ and combining the above relations, we obtain
\begin{equation} \label{energy_balance}
\frac{d}{dt} \Big(\frac{1}{2} |v|^2 + K E (t) \Big)
+ \nu |v_x|^2 + K|\mu_x|^2 = 0.
\end{equation}
Hence, at least formally, one has a bound on the kinetic energy $\tfrac12|v|^2$ and the chemical energy $E(t)$.
The arguments are made rigorous in the next subsection.

We note in passing that a similar analysis for the divergence-form equation \eqref{pre-main_div_form1} yields, after an application of Leibniz's rule and integration by parts,
\begin{align} \label{energy_balance_div_form}
\frac{d}{dt} \Big(\tfrac{1}{2} |v|^2 + K E (t) \Big)
+ \nu |v_x|^2 + K|\mu_x|^2 = 
-K\langle \mu \phi, v_x\rangle.
\end{align}
Hence there is no clear conserved quantity for this divergence-form equation. Finally, we observe that 
we can recover \eqref{energy_balance} in the case of divergence-form coupling by modifying
the coupling term in the momentum equation to the choice used in \eqref{pre-main_div_form2}.

\subsection{Global well-posedness} \label{global_section}

In this section, we establish global well-posedness for 
system \eqref{main} subject to periodic boundary conditions 
and initial condition
\begin{equation} \label{initial-condition}
(v, \phi)|_{t=0} = (v_0, \phi_0).
\end{equation}
We start with the following definition of what we mean by 
a strong solution to system \eqref{main}.
\begin{definition}\label{definition-strongsolution}
	Suppose that $v_0\in H^1(-L, L)$ and $\phi_0\in H^2(-L,L)$ are $2L-$periodic. Given any positive time $T>0$, we say $(v,\phi)$ is a strong solution to system \eqref{main}, subject to periodic boundary conditions, on the time interval $[0,T]$, if 
 \begin{enumerate}[(i)]
     \item $(v,\phi)$ has the regularities
	\begin{eqnarray*}
		&&\hskip-.8in
		v\in L^{\infty}(0,T;H^1)\cap L^{2}(0,T;H^2)\cap C([0,T]; L^2)\cap C_w([0,T];H^1),\\
		&&\hskip-.8in
		\phi \in L^{\infty}(0,T;H^2)\cap L^{2}(0,T;H^4)\cap C([0,T]; H^{1})\cap C_w([0,T];H^2);
	\end{eqnarray*} 
 \item $(v,\phi)$ solves system \eqref{main} in the following sense:
	\begin{eqnarray*}
		&&\hskip-.68in
		\partial_t \phi  +v\phi_x = \mu_{xx}  \;\;  \text{in} \; \; L^2(0,T; L^2),\\
		&&\hskip-.68in
		\partial_t v  +vv_x = \nu v_{xx} + K\mu \phi_x  \;\;  \text{in} \; \; L^2(0,T; L^2),
	\end{eqnarray*}
	with $\mu$ defined by $\mu = -\kappa \phi_{xx} + F'(\phi)$, and the initial condition \eqref{initial-condition}
    is satisfied.
 \end{enumerate}
\end{definition}

We have the following result concerning global well-posedness of system \eqref{main}:
\begin{theorem} \label{well-posedness-theorem}
	Suppose that $v_0\in H^1(-L, L)$ and $\phi_0\in H^2(-L,L)$ are $2L-$periodic. For any given time $T>0$ there exists a unique strong solution $(v,\phi)$ of system \eqref{main} on $[0,T]$ satisfying Definition \ref{definition-strongsolution}. Moreover, the unique strong solution $(v,\phi)$ depends continuously on the initial data.
\end{theorem}
For the proof, we first establish formal \textit{a priori} estimates for solutions of system \eqref{main} in section \ref{a-priori-estimate}. These estimates can be justified rigorously by deriving them first to the Galerkin approximation system and then passing to the limit using the Aubin-Lions Lemma (see, for example, \cite[Corollary 4]{SIMON}).

The existence of a strong solution is established in section \ref{existence}, while the uniqueness and continuous dependence on initial data is established in section \ref{uniqueness}.

\subsubsection{\textit{A priori} estimates}\label{a-priori-estimate}

We begin the proof of Theorem \ref{well-posedness-theorem} by establishing formal \textit{a priori} estimates for the solutions of system \eqref{main}. Taking the inner product of the phase equation in \eqref{main} with $2\beta K\phi$ and using the relation $\mu = -\kappa \phi_{xx} + \alpha \phi^3 - \beta \phi$, we obtain
\begin{equation} \label{phi}
\frac{d}{dt} (\beta K |\phi|^2) + 2 \kappa\beta K |\phi_{xx}|^2 + 6\alpha\beta K \int_{-L}^{+L} \phi^2 \phi_x^2 dx
= 2\beta^2 K |\phi_x|^2 + \beta K \int_{-L}^{+L} v_x\phi^2 dx,
\end{equation}
where in this calculation we have integrated by parts and used the periodic boundary conditions. 
By adding \eqref{energy_balance} and \eqref{phi} together, we obtain
\begin{equation} \label{L2-1}
\begin{aligned}
&\frac{d}{dt} \Big(\frac{1}{2}|v|^2 + \frac{1}{2}\beta K |\phi|^2 + \frac{1}{4}\alpha K \|\phi\|_{L^4}^4 + \frac{1}{2} \kappa K |\phi_x|^2 \Big) + \nu |v_x|^2 + K |\mu_x|^2 + 2\kappa\beta K |\phi_{xx}|^2 \\
&\quad \leq 2\beta^2 K |\phi_x|^2 + \beta K \left|\int_{-L}^{+L} v_x\phi^2 dx\right|. 
\end{aligned}
\end{equation}
Using H\"older's inequality and Young's inequality, we obtain
\begin{equation*}
\left|\int_{-L}^{+L} v_x\phi^2 dx \right| \leq  |v_x| \|\phi\|_{L^4}^2 \leq \frac{\nu}{2} |v_x|^2 + \frac{1}{2\nu}\|\phi\|_{L^4}^4.
\end{equation*}
By virtue of this, \eqref{L2-1} becomes
\begin{equation} \label{L2-2}
\begin{aligned}
&\frac{d}{dt} \Big(\frac{1}{2}|v|^2 + \frac{1}{2}\beta K |\phi|^2 + \frac{1}{4}\alpha K \|\phi\|_{L^4}^4 + \frac{1}{2} \kappa K |\phi_x|^2 \Big) + \frac{\nu}{2} |v_x|^2 + K |\mu_x|^2 + 2\kappa\beta K |\phi_{xx}|^2 \\
&\quad \leq  C\Big(\frac{1}{4}\alpha K \|\phi\|_{L^4}^4 + \frac{1}{2} \kappa K|\phi_x|^2\Big),
\end{aligned}
\end{equation}
where $C = C(\nu,\alpha,\beta,\kappa,K)$.\footnote{Note that in \eqref{L2-2}, the two summands in the right-hand side no longer have matching dimensions as we have factored out the constants for the sake of clarity. However, it is straightforward to reproduce these estimates in a dimensionally consistent manner.  For notational simplicity, we continue similarly throughout the paper without further comment.}
Using Gr\"onwall's inequality, we obtain
\begin{equation} \label{L2-3}
\begin{aligned}
&\Big(\frac{1}{2}|v|^2 +  \frac{1}{2}\beta K|\phi|^2 + \frac{1}{4}\alpha K\|\phi\|_{L^4}^4 + \frac{1}{2} \kappa K |\phi_x|^2 \Big)(t) 
\\
&+ \int_0^t \Big( \frac{\nu}{2} |v_x|^2 + K|\mu_x|^2 +  2\kappa\beta K |\phi_{xx}|^2\Big)(s) ds \\
\leq & \Big(\frac{1}{2}|v(0)|^2 + \frac{1}{2}\beta K |\phi(0)|^2 + \frac{1}{4}\alpha K\|\phi(0)\|_{L^4}^4 +  \frac{1}{2} \kappa K|\phi_x(0)|^2 \Big) e^{Ct} ,
\end{aligned}
\end{equation}
for arbitrary time $t\geq 0$. 
Therefore, for arbitrary time $T>0$ the following bounds hold:
\begin{equation} \label{L2-4}
\begin{aligned}
&v \in L^\infty(0,T; L^2) \cap L^2(0,T;H^1); \\
&\phi \in L^\infty(0,T; H^1) \cap L^2(0,T; H^2); \\
&\mu \in L^2(0,T; H^1).
\end{aligned}
\end{equation}

Next, we take the inner product of the phase equation in \eqref{main} with $-\phi_{xx}$. By integration by parts and periodic boundary condition, and 
also using the relation $(F')_x = 3\alpha \phi^2 \phi_x - \beta \phi_x$, we obtain
\begin{equation} \label{H1-1}
\begin{aligned}
\frac{1}{2}\frac{d}{dt} |\phi_x|^2  +\kappa|\phi_{xxx}|^2 &\leq  \int_{-L}^{+L} \Big| v\phi_x \phi_{xx} \Big|+ \Big|(F')_x \phi_{xxx} \Big| dx  \\
&\leq |v| \|\phi_x\|_{L^\infty}  |\phi_{xx}| + C_1(\|\phi\|_{L^\infty}^2 |\phi_{x}| + |\phi_x|) |\phi_{xxx}| \\
&\leq C_2 |v| \|\phi\|_{H^2}^2 + \frac{\kappa}{2}|\phi_{xxx}|^2 + C_3(\|\phi\|_{H^1}^6 + \|\phi\|_{H^1}^2),
\end{aligned}  
\end{equation}
where we have used H\"older's inequality, Young's inequality and a Sobolev inequality, and 
introduced fixed sufficiently large constants $C_1$, $C_2$, and $C_3$. 
Therefore, 
\begin{equation} \label{H1-2}
\frac{d}{dt} |\phi_x|^2  +\kappa|\phi_{xxx}|^2 \leq   2C_2|v| \|\phi\|_{H^2}^2 + 2C_3(\|\phi\|_{H^1}^2 + \|\phi\|_{H^1}^6)=: K_1(t).
\end{equation}
From \eqref{L2-4}, we know $K_1 \in L^1 (0,T)$ for arbitrary $T>0$. Integrating \eqref{H1-2} from $0$ to $t$, we obtain
\begin{equation} \label{H1-3}
|\phi_x|^2(t)  + \int_0^t \kappa|\phi_{xxx}|^2(s) ds \leq  |\phi_x|^2(0) + \int_0^t K_1(s) ds,
\end{equation}
and this implies that 
\begin{equation} \label{H1-4}
\phi \in L^\infty(0,T; H^1) \cap L^2(0,T; H^3).
\end{equation}

Next, taking the inner product of the phase equation in \eqref{main} with $\phi_{xxxx}$, and of the momentum equation in \eqref{main} with $-v_{xx}$, 
we integrate by parts and use the periodic boundary conditions to obtain the 
inequality
\begin{equation} \label{H2-1}
\begin{aligned}
&\frac{d}{dt} \Big(\frac{1}{2}|v_x|^2 + \frac{1}{2} |\phi_{xx}|^2 \Big) + \nu |v_{xx}|^2 + \kappa |\phi_{xxxx}|^2 \\
&= \int_{-L}^{+L} \Big[ vv_xv_{xx} - K\mu \phi_x v_{xx} + (F')_{xx} \phi_{xxxx} - v\phi_x \phi_{xxxx} \Big] dx\\
&=:I_1 + I_2 + I_3+ I_4.
\end{aligned}
\end{equation}
Owing again to H\"older's inequality, Young's inequality and a Sobolev inequality, we obtain the following estimates for $I_1,\ldots,I_4$,
\begin{equation*} 
\begin{aligned}
|I_1| &\leq  |v_{xx}| |v_x| \|v\|_{L^\infty} \leq C |v_{xx}| |v_x| \|v\|_{H^1} \leq \frac{\nu}{4} |v_{xx}|^2 + \frac{C^2}{\nu} \|v\|_{H^1}^2  |v_x|^2,\\
|I_2| & \leq K|\mu| \|\phi_x\|_{L^\infty} |v_{xx}|\leq \frac{\nu}{4} |v_{xx}|^2 + C |\mu|^2 (\|\phi\|^2_{H^1} + |\phi_{xx}|^2),\\
|I_3| & \leq |(F')_{xx}| |\phi_{xxxx}| \leq \frac{\kappa}{4} |\phi_{xxxx}|^2 + C|6\alpha \phi \phi_x^2 + 3\alpha \phi^2 \phi_{xx} -\beta\phi_{xx}|^2 \\
&\leq \frac{\kappa}{4} |\phi_{xxxx}|^2  + \widetilde{C} (\|\phi\|_{H^1}^4 +1) \|\phi\|_{H^2}^2, \\
|I_4| & \leq |v|\|\phi_x\|_{L^\infty}|\phi_{xxxx}| \leq \frac{\kappa}{4} |\phi_{xxxx}|^2 + C|v|^2\|\phi\|_{H^2}^2.
\end{aligned}
\end{equation*}
Plugging these estimates back into \eqref{H2-1}, we obtain the inequality
\begin{equation} \label{H2-2}
\begin{aligned}
&\quad\frac{d}{dt} \Big(|v_x|^2 +  |\phi_{xx}|^2 \Big) + \nu |v_{xx}|^2 + \kappa |\phi_{xxxx}|^2 \\
& \leq C_1(\|v\|_{H^1}^2 + |\mu|^2)(|v_x|^2 +  |\phi_{xx}|^2) + C_2|\mu|^2\|\phi\|_{H^1}^2 + C_3 (\|\phi\|_{H^1}^4 + |v|^2 +1) \|\phi\|_{H^2}^2.
\end{aligned}
\end{equation}
At this point, we introduce the quantities $$K_2(t) :=C_1(\|v\|_{H^1}^2 + |\mu|^2),$$ $$K_3(t):= C_2|\mu|^2\|\phi\|_{H^1}^2 + C_3 (\|\phi\|_{H^1}^4 + |v|^2 +1) \|\phi\|_{H^2}^2,$$
noting from \eqref{L2-4} that $K_2, K_3 \in L^1(0,T)$ for arbitrary time $T>0$. Therefore, applying Gr\"onwall's inequality to \eqref{H2-2} we obtain the inequality
\begin{equation} \label{H2-3}
\begin{aligned}
&\Big(|v_x|^2 +  |\phi_{xx}|^2 \Big)(t) + \int_0^t (\nu |v_{xx}|^2 + \kappa |\phi_{xxxx}|^2)(s)ds \\
& \leq \Big(|v_x(0)|^2 +  |\phi_{xx}(0)|^2 + \int_0^t K_3(s)ds\Big) \exp\Big(\int_0^t K_2(s)ds\Big).
\end{aligned}
\end{equation}
This together with \eqref{L2-4} implies that for any arbitrary time $T>0$,
\begin{equation} \label{H2-4}
\begin{aligned}
&v \in L^\infty(0,T; H^1) \cap L^2(0,T;H^2); \\
&\phi \in L^\infty(0,T; H^2) \cap L^2(0,T; H^4); \\
&\mu \in L^\infty(0,T; L^2) \cap L^2(0,T; H^2).
\end{aligned}
\end{equation}
In order to use Aubin-Lions Compactness Theorem, we should additionally establish bounds for $v_t$ and $\phi_t$. Taking the $L^2$ inner product in both space and time of the phase equation and momentum equation in \eqref{main} with some test function $\psi\in L^2(0,T;L^2)$, we use H\"older's inequality and a Sobolev inequality similarly as before to show that 
\begin{equation}\label{phit}
\begin{aligned}
   |\langle \phi_t, \psi \rangle| &\leq |\langle v\phi_x , \psi \rangle| + |\langle \mu_{xx} , \psi \rangle| \\
   &\leq C_1\|v\|_{L^\infty(0,T; L^2)} \|\phi\|_{L^2(0,T; H^2)} \|\psi\|_{L^2(0,T; L^2)} + C_2 \|\mu\|_{L^2(0,T; H^2)} \|\psi\|_{L^2(0,T; L^2)},
\end{aligned}
\end{equation}
and
\begin{equation}\label{vt}
\begin{aligned}
|\langle v_t, \psi \rangle| &\leq |\langle vv_x , \psi \rangle| + |\langle \nu v_{xx} , \psi \rangle| + |\langle K\mu \phi_x , \psi \rangle| \\
&\leq C_1\|v\|_{L^\infty(0,T; L^2)} \|v\|_{L^2(0,T; H^2)} \|\psi\|_{L^2(0,T; L^2)} + C_2 \|v\|_{L^2(0,T; H^2)} \|\psi\|_{L^2(0,T; L^2)}  \\
&+ C_3 \|\mu\|_{L^\infty(0,T; L^2)} \|\phi\|_{L^2(0,T; H^2)} \|\psi\|_{L^2(0,T; L^2)}.
\end{aligned}
\end{equation}
Thanks to \eqref{H2-4} and since $\psi\in L^2(0,T;L^2)$, we see the right hand sides of \eqref{phit} and \eqref{vt} are finite. Therefore, we conclude that 
\begin{equation} \label{vtphit}
v_t, \phi_t \in L^2(0,T;L^2).
\end{equation}

\subsubsection{Existence of strong solutions}\label{existence}
We now employ the standard Galerkin approximation procedure to establish the existence of strong solutions of system \eqref{main}. For any $k\in \mathbb{Z}$ and $m \in \mathbb{N}$, we set
\begin{equation} \label{base}
\psi_k := e^{k\frac{\pi}{L} i},
\end{equation}
and 
\begin{equation} \label{basespace}
V_m := \bigg\{ \psi \in L^2(-L,L)\; : \; \psi= \sum\limits_{|k|\leq m} a_k \phi_k, \; a_{-k}=a_{k}^{*}  \bigg \},
\end{equation}
observing that $\{\psi_k\}_{k\in \mathbb{Z}}$ forms an orthonormal basis of $L^2(-L, L)$, and $V_m$ is a finite-dimensional subspace of $L^2(-L, L)$. For any function $f\in L^2(-L, L)$, we set $\hat{f_j}=\langle f,\psi_j\rangle$, and we let 
$P_m$ denote the orthogonal projection from $L^2(-L, L)$ to $V_m$, namely $P_m f := \sum_{|k|\leq m} \hat{f_k} \psi_k$. 
We note here that for any $f\in H^1(-L,L)$ it holds that $(P_m f)_x = P_m(f_x)$. Letting now
\begin{equation} \label{basefunction}
\phi_m:=\sum_{|k|\leq m}a_k(t)\psi_k(x), \;\; v_m:=\sum_{|k|\leq m}b_k(t)\psi_k(x),
\end{equation}
we consider the following Galerkin approximation system for \eqref{main}:
\begin{equation} \label{main-Galerkin}
\begin{aligned}
\partial_t\phi_m + P_m(v_m \partial_x\phi_m) &= P_m \partial_{xx}\mu_{m}, \\
\mu_m &= -\kappa \partial_{xx}\phi_{m} + P_m F' (\phi_m), \\
\partial_t v_m + P_m(v_m \partial_x v_m) &= \nu \partial_{xx}v_{m} + K P_m(\mu_m \partial_x\phi_m),
\end{aligned}
\end{equation}
subject to periodic boundary conditions and the following initial conditions:
\begin{equation} \label{initial-condition-Galerkin}
(\phi_m(0), v_m(0)) = (P_m\phi_0,  P_m v_0).
\end{equation}
For each $m\geq 1$, the Galerkin approximation system \eqref{main-Galerkin} corresponds to a first-order system of ordinary differential equations in the coefficients $a_k$ and $b_k$ for $0\leq |k| \leq m$, with at most quadratic nonlinearity. Therefore, by the Picard-Lindel\"of Theorem, there exists some $t_m > 0$ such that system \eqref{main-Galerkin} admits a unique solution $(\phi_m, v_m)$ on the interval $[0, t_m]$. Since $(\phi_m, v_m)$ have finitely many modes, they are smooth functions, and therefore if one repeats the arguments concerning the \textit{a priori} estimates for the solution, one obtains the same estimates as in the previous section for the Galerkin approximate solution $(\phi_m, v_m)$. More precisely, we first establish an estimate similar as \eqref{L2-3},
\begin{equation} \label{L2-3-Galerkin}
\begin{aligned}
&\Big(\frac12|v_m|^2 + \frac12\beta K|\phi_m|^2 + \frac14\alpha K\|\phi_m\|_{L^4}^4 +  \frac12\kappa K|\partial_x\phi_m|^2 \Big)(t) 
\\
&+ \int_0^t \Big(\frac\nu 2|\partial_xv_m|^2 + K|\partial_x\mu_m|^2 +  2\kappa\beta K|\partial_{xx}\phi_{m}|^2\Big)(s) ds \\
\leq&\Big(\frac12|v_m(0)|^2 + \frac12\beta K|\phi_m(0)|^2 + \frac14\alpha K\|\phi_m(0)\|_{L^4}^4 +  \frac12\kappa K|\partial_x\phi_m(0)|^2 \Big) e^{Ct} \\
\leq &\Big(\frac12|v_0|^2 + \frac12\beta K|\phi_0|^2 + \frac14C\alpha K\|\phi_0\|_{H^1}^4 + \frac12\kappa K |\partial_x\phi_0|^2 \Big) e^{Ct},
\end{aligned}
\end{equation}
where we have applied Sobolev embedding $H^1\hookrightarrow L^4$ in the last inequality.
Notice that the right hand side of \eqref{L2-3-Galerkin} is uniformly bounded with respect to $m$. Therefore, for arbitrary time $T>0$ it holds that $\|\phi_m\|_{L^\infty(0,T;L^2)}$ and $\|v_m\|_{L^\infty(0,T;L^2)}$ remain bounded. This implies that for any $m\in \mathbb{N}$, the solution $(\phi_m, v_m)$ of \eqref{main-Galerkin} exists globally in time. Next, following the same process as in the previous section, for arbitrary time $T>0$, we obtain the following uniform bounds with respect to $m$ for the sequence of Galerkin approximate solutions $(\phi_m, v_m)$ and corresponding $\mu_m$, and their time derivative $(\partial_t\phi_m, \partial_t v_m)$:
\begin{equation} \label{H2-4-Galerkin}
\begin{aligned}
&\{v_m\}_{m\in\mathbb N}\;\;\text{are uniformly bounded in}\;\; L^\infty(0,T; H^1) \cap L^2(0,T;H^2); \\
&\{\phi_m\}_{m\in\mathbb N} \;\;\text{are uniformly bounded in}\;\; L^\infty(0,T; H^2) \cap L^2(0,T; H^4); \\
&\{\mu_m\}_{m\in\mathbb N} \;\;\text{are uniformly bounded in}\;\; L^\infty(0,T; L^2) \cap L^2(0,T; H^2); \\
&\{\partial_t\phi_m\}_{m\in\mathbb N}, \; \{\partial_t v_m\}_{m\in\mathbb N} \;\;\text{are uniformly bounded in}\;\; L^2(0,T; L^2).
\end{aligned}
\end{equation}
By the Banach--Alaoglu Theorem (see, e.g., \cite{Folland99}), there exist a subsequence, which we also denote by $(\phi_m, v_m, \mu_m, \partial_t\phi_m, \partial_t v_m)$, and corresponding
limits, $(\phi, v, \mu, \partial_t \phi, \partial_t v)$, respectively, such that
\begin{eqnarray} \label{weak-convergence}
\begin{aligned}
&v_m \overset{*}{\rightharpoonup} v  \;\text{ in} \; L^\infty(0,T;H^1) \; \text{and }v_m \rightharpoonup v\text{ in} \; L^2(0,T;H^2); \\
&\phi_m \overset{*}{\rightharpoonup} \phi \; \text{ in} \; L^\infty(0,T;H^2) \; \text{and } \phi_m \rightharpoonup \phi\text{ in}\; L^2(0,T;H^4); \\
&\mu_m \overset{*}{\rightharpoonup} \mu \; \text{ in} \; L^\infty(0,T;L^2)  \;\text{and } \mu_m \rightharpoonup \mu\text{ in} \; L^2(0,T;H^2); \\
&\partial_t\phi_m \rightharpoonup \partial_t\phi, \; \partial_t v_m \rightharpoonup \partial_t v \;  \text{ in} \; L^2(0,T;L^2).
\end{aligned}
\end{eqnarray}
Combining the Aubin-Lions Lemma with \eqref{H2-4-Galerkin}, we obtain the following strong convergence:
\begin{equation} \label{strong-convergence}
\begin{aligned}
&v_m \rightarrow v \;  \text{ in} \; L^2(0,T;H^{1})\cap C([0,T]; L^2) ; \\
&\phi_m \rightarrow \phi \;  \text{ in} \; L^2(0,T;H^{3})\cap C([0,T]; H^{1}).
\end{aligned}
\end{equation}
As $v\in L^\infty(0,T;H^1) \cap C([0,T]; L^2)$ and $\phi \in L^\infty(0,T;H^2) \cap C([0,T]; H^1)$, according to \cite[Theorem 2.1]{Strauss_1966_PJM_Weak_continuity} (see also \cite[Chapter 3, Lemma 1.4]{temam2001navier}), it follows that 
\begin{equation}
    v\in C_w([0,T];H^1),\quad \phi\in C_w([0,T];H^2).
\end{equation}
As $v_m(0)\to v_0$ in $H^1$ and $\phi_m(0)\to \phi_0$ in $H^2$, the above implies that $v(0)=v_0$ and $\phi(0)=\phi_0$, i.e., the limits $(v,\phi)$ satisfy the desired initial conditions.

In order to verify that the limit function $\mu$ is the same as its counterpart defined in \eqref{main}, 
we begin with the inequality
\begin{equation*} 
\begin{aligned}
&\|\mu_m - (-\kappa\phi_{xx}+ F'(\phi))\|_{L^2(0,T;L^2)}= \|-(\kappa\partial_{xx}\phi_m - \kappa \phi_{xx}) + (P_m F' (\phi_m) - F'(\phi))\|_{L^2(0,T;L^2)} \\
&\leq \kappa\|\partial_{xx}\phi_m - \phi_{xx}\|_{L^2(0,T;L^2)} + \|P_m F' (\phi_m) - F'(\phi)\|_{L^2(0,T;L^2)} =: I_m + J_m.
\end{aligned}
\end{equation*}
For $I_m$, we can use \eqref{strong-convergence} to see that as $m$ tends to $\infty$, $I_m\rightarrow 0$. For $J_m$, using H\"older's inequality and a Sobolev inequality, we estimate
\begin{equation} \label{mu-2}
\begin{aligned}
J_m &\leq \alpha\| P_m(\phi_m^3-\phi^3)\|_{L^2(0,T;L^2)} + \alpha\|P_m(\phi^3) - \phi^3 \|_{L^2(0,T;L^2)} + \beta \|\phi_m - \phi\|_{L^2(0,T;L^2)} \\
& \leq \alpha \| \phi_m-\phi\|_{L^\infty(0,T;H^1)} \| \phi_m^2 + \phi_m \phi + \phi^2\|_{L^2(0,T;L^2)} \\
&\hspace{1cm}+ \alpha\|P_m(\phi^3) - \phi^3 \|_{L^2(0,T;L^2)} + \beta \|\phi_m - \phi\|_{L^2(0,T;L^2)}.
\end{aligned}
\end{equation}
Using the estimates \eqref{H2-4} and \eqref{H2-4-Galerkin}, along with the strong convergence \eqref{strong-convergence}
and the fact that $P_m f \rightarrow f$ in $L^2$ for any function $f\in L^2$, we can pass $m\rightarrow \infty$ to see that the right-hand side of \eqref{mu-2} approaches to $0$. Therefore, 
\begin{equation} \label{mu-convergence}
\mu_m \rightarrow -\kappa\phi_{xx} + F'(\phi)  \; \; \text{in} \;\; {L^2(0,T;L^2)}.
\end{equation}
By uniqueness of the limit, we conclude that the limit $\mu$ is the same as the function $\mu$ specified in \eqref{main}. 

Next, we show that $(\phi, v)$ is a solution of \eqref{main}. Taking the $L^2$ inner product in both space and time of the phase equation in \eqref{main-Galerkin} with test function $\psi\in L^2(0,T;H^2)$, we obtain the relation 
\begin{equation} \label{phase-convergence}
\begin{aligned}
 0&=\langle\partial_t\phi_m + P_m(v_m \partial_x\phi_m) - P_m \partial_{xx}\mu_{m}, \psi \rangle \\
 &= \langle\partial_t\phi_m , \psi \rangle - \langle \mu_{m}, \psi_{xx} \rangle + \langle P_m(v_m \partial_x\phi_m) , \psi \rangle 
:= I_m + J_m + K_m,
\end{aligned}
\end{equation}
where in obtaining the second line we integrated by parts and employed the periodic boundary conditions, along with the observation that $P_m$ and $\partial_x$ commute.   
From \eqref{weak-convergence}, we know $I_m \rightarrow \langle\partial_t\phi , \psi \rangle $ as $m \to \infty$, and since $\psi_{xx}\in L^2(0,T;L^2)$, we also see that $J_m \rightarrow \langle -\mu , \psi_{xx} \rangle = \langle -\mu_{xx} , \psi \rangle$. For $K_m$, using H\"older's inequality and a Sobolev inequality, we obtain
\begin{equation} \label{phase-convergence-nonlinear}
\begin{aligned}
&|K_m - \langle v\phi_x,\psi\rangle| \leq |\langle (v_m - v) \partial_x \phi_m, P_m\psi\rangle| + |\langle v (\partial_x \phi_m - \phi_x), P_m\psi\rangle| + |\langle P_m(v\phi_x) - v\phi_x, \psi\rangle|\\
& \leq \|v_m-v\|_{L^\infty(0,T;L^2)} \|\phi_m\|_{L^2(0,T;H^2)}  \|P_m \psi\|_{L^2(0,T;L^2)} + \|P_m(v\phi_x) - v\phi_x\|_{L^2(0,T;L^2)} \|\psi\|_{L^2(0,T;L^2)} \\
& \hspace{1cm}+ \|v\|_{L^\infty(0,T;H^1)} \|\phi_m - \phi\|_{L^2(0,T;H^1)} \|P_m \psi\|_{L^2(0,T;L^2)}.
\end{aligned}
\end{equation}
Appealing now to \eqref{H2-4}, \eqref{H2-4-Galerkin}, and the strong convergence \eqref{strong-convergence}, we find that the right-hand side of \eqref{phase-convergence-nonlinear} approaches $0$ as $m\rightarrow \infty$. Therefore, $K_m \rightarrow \langle v\phi_x,\psi\rangle$ as $m \to \infty$. Now we can pass the limit $m\rightarrow \infty$ in \eqref{phase-convergence} to get
\begin{equation*} 
\begin{aligned}
\langle\partial_t\phi+ v \phi_x - \partial_{xx}\mu, \psi \rangle = 0.
\end{aligned}
\end{equation*}
This last relation is true for arbitrary $\psi\in L^2(0,T;H^2)$, and therefore $\phi_t+v\phi_x=\mu_{xx}$ in $L^2(0,T;H^{-2})$. Moreover, since all terms in the phase equation are actually in $L^2(0,T;L^2)$, we can further conclude that
\begin{equation} \label{phase-convergence-result}
\phi_t+v\phi_x=\mu_{xx} \;\; \text{in} \;\; L^2(0,T;L^2).
\end{equation}
Similarly, taking the $L^2$ inner product in both space and time of the momentum equation in \eqref{main-Galerkin} with test function $\psi\in L^2(0,T;H^2)$, we obtain
\begin{equation} \label{momentum-convergence}
\begin{aligned}
0&=\langle\partial_t v_m + P_m(v_m \partial_x v_m)- \nu \partial_{xx}v_{m} - K P_m(\mu_m \partial_x\phi_m),  \psi \rangle \\
&= \langle\partial_t v_m , \psi \rangle - \nu\langle  v_m, \psi_{xx} \rangle + \langle P_m(v_m \partial_x v_m) , \psi \rangle - K \langle P_m(\mu_m\partial_x \phi_m) , \psi \rangle\\
&=: I_m+J_m+K_m+L_m.
\end{aligned}
\end{equation}
For $I_m$, $J_m$, and $K_m$, the steps used above for the phase equation can be repeated to show that as $m \to \infty$,  $I_m \rightarrow \langle\partial_t v, \psi \rangle$, $J_m\rightarrow \langle -\nu v_{xx}, \psi \rangle$, and $K_m \rightarrow \langle vv_x, \psi \rangle$. For $L_m$, we can proceed similarly as with $K_m$, except now using \eqref{mu-convergence}, and we find that $L_m \rightarrow \langle -K \mu\phi_x, \psi \rangle$ as $m \to \infty$. Since all terms in the momentum equation are in $L^2(0,T;L^2)$, we conclude that
\begin{equation} \label{momentum-convergence-result}
v_t+vv_x=\nu v_{xx} + K\mu\phi_x \;\; \text{in} \;\; L^2(0,T;L^2).
\end{equation}
Therefore, $(v,\phi)$ is a strong solution to system \eqref{main}.

\subsubsection{Uniqueness of solutions and continuous dependence on initial data}\label{uniqueness}

We now conclude the proof of Theorem \ref{well-posedness-theorem} by establishing uniqueness of 
the strong solution. For this, we let $(\phi_1, v_1)$ and $(\phi_2, v_2)$ denote two strong solutions 
of system \eqref{main}, with initial data $((\phi_0)_1, (v_0)_1)$ and $((\phi_0)_2, (v_0)_2)$, respectively,
and we set $\phi := \phi_1 - \phi_2$, $v := v_1 - v_2$, $\phi_0 := (\phi_0)_1 - (\phi_0)_2$, and 
$v_0 := (v_0)_1 - (v_0)_2$. Taking the difference between equations for $(\phi_1, v_1)$ and for $(\phi_2, v_2)$, 
we find that 
\begin{equation} \label{uniqueness-equation}
\begin{aligned}
&\phi_t + v_1 \phi_x + v\partial_x \phi_2 + \kappa \phi_{xxxx} = \partial_{xx}\Big[\alpha \phi(\phi_1^2 + \phi_1 \phi_2 + \phi_2^2) - \beta \phi\Big],\\
&v_t + v_1 v_x + v\partial_x v_2 +\frac{1}{2}\kappa K \partial_x\Big(\phi_x (\partial_x\phi_1 + \partial_x\phi_2)\Big) - \nu v_{xx} \\
&= \frac{\alpha K}{4}\partial_x\Big(\phi (\phi_1 + \phi_2)(\phi_1^2 + \phi_2^2)\Big) - \frac{\beta K}{2} \partial_x\Big(\phi (\phi_1 + \phi_2)\Big).
\end{aligned}
\end{equation}
Taking the $L^2$ inner product of the phase equation in \eqref{uniqueness-equation} with $\phi$ and $-\phi_{xx}$, and 
taking the $L^2$ inner product of the momentum equation in \eqref{uniqueness-equation} with $v$, we integrate by parts 
and use the periodic boundary conditions to see that 
\begin{equation} \label{uniqueness-pf1}
\begin{aligned}
&\frac{1}{2}\frac{d}{dt}(|\phi|^2 + |\phi_x|^2 + |v|^2) + \kappa(|\phi_{xx}|^2 + |\phi_{xxx}|^2) + \nu |v_{x}|^2 \\
&= \int_{-L}^{+L}\Bigg[ \Big(v_1 \phi_x + v\partial_x \phi_2\Big)(-\phi + \phi_{xx}) + \partial_{xx}\Big[\alpha \phi(\phi_1^2 + \phi_1 \phi_2 + \phi_2^2) - \beta \phi\Big] (-\phi + \phi_{xx}) \\
&\quad- \Big(v_1 v_x + v\partial_x v_2\Big)v - \frac{1}{2}\kappa K \partial_x\Big(\phi_x (\partial_x\phi_1 + \partial_x\phi_2)\Big)v + \frac{\alpha K}{4}\partial_x\Big(\phi (\phi_1 + \phi_2)(\phi_1^2 + \phi_2^2)\Big)v \\
&\quad- \frac{\beta K}{2} \partial_x\Big(\phi (\phi_1 + \phi_2)\Big)v\Bigg] dx =: I_1 + I_2 + I_3 + I_4 + I_5 + I_6.
\end{aligned}
\end{equation}
Using H\"older's inequality, Young's inequality and Sobolev inequalities, we obtain
\begin{equation*} 
\begin{aligned}
|I_1| &= \Big|\int_{-L}^{+L} \frac{1}{2}\partial_x v_{1}(\phi^2 - \phi_x^2) + v\partial_x \phi_2(-\phi + \phi_{xx}) dx\Big| \\
&\leq C_1 \|v_1\|_{H^2}(|\phi|^2 + |\phi_x|^2) + C_2\|\phi_2\|_{H^2}(|\phi|^2 + |v|^2) + C_3\|\phi_2\|_{H^2}^2 |v|^2 + \frac{\kappa}{4}|\phi_{xx}|^2\\
&\leq C(\|v_1\|_{H^2} + \|\phi_2\|_{H^2} + \|\phi_2\|_{H^2}^2)(|\phi|^2 + |\phi_x|^2 + |v|^2)+ \frac{\kappa}{4}|\phi_{xx}|^2, \\
|I_2| &= \Big|\int_{-L}^{+L} \Big[\alpha \phi(\phi_1^2 + \phi_1 \phi_2 + \phi_2^2) - \beta \phi\Big] \phi_{xx} +  \partial_x \Big[\alpha \phi(\phi_1^2 + \phi_1 \phi_2 + \phi_2^2) - \beta \phi\Big] \phi_{xxx} dx\Big|\\
&\leq C(|\phi|^2+|\phi_x|^2) (\|\phi_1^2 + \phi_1 \phi_2 + \phi_2^2\|_{H^2}^2 + 1)  + \frac{\kappa}{4}(|\phi_{xx}|^2 + |\phi_{xxx}|^2) \\
&\leq C(1 + \|\phi_1\|_{H^2}^4 + \|\phi_2\|_{H^2}^4)(|\phi|^2+|\phi_x|^2) + \frac{\kappa}{4}(|\phi_{xx}|^2 + |\phi_{xxx}|^2),
\end{aligned}
\end{equation*}
and likewise,
\begin{equation*} 
\begin{aligned}
|I_3| &= \Big|\int_{-L}^{+L} \frac{1}{2}\partial_x v_{1}v^2  - v^2\partial_x \phi_2 dx\Big| \leq C(\|v_1\|_{H^2} + \|\phi_2\|_{H^2})|v|^2,\\
|I_4| &\leq C|v_x||\phi_x|\|\partial_x\phi_1 + \partial_x\phi_2\|_{L^\infty} \leq C(\|\phi_1\|_{H^2}^2 + \|\phi_2\|_{H^2}^2) |\phi_x|^2 + \frac{\nu}{6} |v_x|^2 ,\\
|I_5| &\leq C|v_x||\phi|\|(\phi_1+\phi_2)(\phi_1^2+\phi_2^2)\|_{L^\infty} \leq C(\|\phi_1\|_{H^1}^6 + \|\phi_2\|_{H^1}^6)  |\phi|^2 + \frac{\nu}{6} |v_x|^2, \\
|I_6| &\leq C|v_x||\phi|\|(\phi_1+\phi_2)\|_{L^\infty} \leq C(\|\phi_1\|_{H^1}^2 + \|\phi_2\|_{H^1}^2)  |\phi|^2 + \frac{\nu}{6} |v_x|^2.
\end{aligned}
\end{equation*}
From above, using Young's inequality we obtain
\begin{equation} \label{uniqueness-pf3}
\begin{aligned}
&\quad\frac{d}{dt}(|\phi|^2 + |\phi_x|^2 + |v|^2) + \kappa(|\phi_{xx}|^2 + |\phi_{xxx}|^2) + \nu |v_{x}|^2 \\
&\leq C(1 + \|v_1\|_{H^2} + \|\phi_1\|_{H^2}^6 + \|\phi_2\|_{H^2}^6) (|\phi|^2 + |\phi_x|^2 + |v|^2).
\end{aligned}
\end{equation}
Thanks to \eqref{H2-4}, we know $(1 +  \|v_1\|_{H^2} + \|\phi_1\|_{H^2}^6 + \|\phi_2\|_{H^2}^6) \in L^1(0,T)$ for arbitrary time $T>0$. By Gr\"onwall's inequality, it the follows that
\begin{equation} \label{uniqueness-result}
\begin{aligned}
\Big(|\phi|^2 + |\phi_x|^2 + |v|^2\Big)(t) \leq  &\Big(|\phi_0|^2 + |(\phi_0)_x|^2 + |v_0|^2\Big) 
\\
&\hspace{1cm}\times\exp\Big(\int_0^t C(1 +  \|v_1\|_{H^2} + \|\phi_1\|_{H^2}^6 + \|\phi_2\|_{H^2}^6) ds\Big).
\end{aligned}
\end{equation}
The above inequality proves the continuous dependence of the
solutions on the initial data, and in particular, when
$\phi_0 \equiv0$ and $ v_0 \equiv0$, it follows that $\phi(t)=v(t)=0$ for all $t\geq 0$.
Therefore, the strong solution is unique.

\section{Conclusion} \label{conclusions-section}

In this paper, we carried out a detailed exploration of coarsening rates for a one-dimensional (1D) model of two-phase 
flow. Our primary goal was to take advantage of analytic results that are only currently available in the 1D
setting to build a framework for analyzing and evaluating coarsening rates for models of two-phase flow. A key 
step in this program was the introduction of a natural measure of coarsening applicable to any solution 
of the Cahn--Hilliard equation. This coarsening measure directly links the free energy of a solution to a 
corresponding periodic solution with the same free energy, taking then the period of this solution to be 
a gauge for coarseness. Notably, this method circumvents the need for any \textit{a priori} assumptions about the 
structure of the solution, such as near-periodicity. Using this measure of coarseness, we compared known analytic
rates with rates obtained by numerical simulations, obtaining the first such direct comparisons. 

Our research also extended to examining coarsening dynamics within the context of the Cahn--Hilliard equation 
when coupled with Burgers equation, a natural one-dimensional analog of the Cahn--Hilliard--Navier--Stokes 
system. A key focus was on comparing the coarsening rates between the coupled and uncoupled systems through 
computational analysis, and we found, at least for the computations carried out, that coupling has the effect 
of substantially increasing coarsening rates. The primary mechanism for this increase seems to be the appearance
of gradients in the velocity profile, which expand some enriched regions and compress others.  

In addition to these findings, we proposed what we call the Burgers--Cahn--Hilliard System, and provided the first proof of global well-posedness for this system. The proof is distinct from proofs of well-posedness for the Cahn--Hilliard--Navier--Stokes System, as BCH has   no divergence-free condition, and hence the nonlinear terms must be handled differently.

In future work, we would like to extend elements of these results to multi-dimensional Cahn--Hilliard Systems, the Cahn--Hilliard--Navier--Stokes system, and other related equations. Although our 1D investigations point toward promising directions to pursue in such studies, we expect new phenomena to emerge in those cases, and leave further commentary to those future works.

\section*{Acknowledgements}
  A.L. was partially supported by NSF grants DMS-2206762 and CMMI-1953346, and USGS Grant No. G23AC00156-01. Q.L. was partially supported by an AMS--Simons travel grant.

\bibliographystyle{abbrv}
\bibliography{refs}

\begin{thebibliography}{10}

\bibitem{AF08}
H.~Abels and E.~Feireisl.
\newblock On a diffuse interface model for a two-phase flow of compressible
  viscous fluids.
\newblock {\em Indiana U. Math. J.}, 57:659--698, 2008.

\bibitem{Bateman_1915_burgers}
H.~Bateman.
\newblock Some recent researches on the motion of fluids.
\newblock {\em Monthly Weather Review}, 43(4):163--170, 1915.

\bibitem{BLMPQ2010}
F.~Boyer, C.~Lapuerta, S.~Minjeaud, B.~Piar, and M.~Qintard.
\newblock {C}ahn--{H}illiard/{N}avier--{S}tokes model for the simulation of
  three-phase flows.
\newblock {\em Transp. Porous Media.}, 82:463--483, 2010.

\bibitem{Bray94}
A.~J. Bray.
\newblock Theory of phase-ordering kinetics.
\newblock {\em Adv. Phys.}, 43:357--459, 1994.

\bibitem{BKT1999}
J.~Bricmont, A.~Kupiainen, and J.~Taskinen.
\newblock Stability of {C}ahn--{H}illiard fronts.
\newblock {\em Comm. Pure Appl. Math.}, LII:839--871, 1999.

\bibitem{Brummelen2015}
E.~H.~v. Brummelen, M.~Shokrpour-Roudbari, and G.~J.~v. Zwieten.
\newblock Elasto-capillarity simulations based on the
  {N}avier--{S}tokes--{C}ahn--{H}illiard equations.
\newblock {\em Advances in Computational Fluid-Structure Interaction and Flow
  Simulation: New Methods and Challenging Computations}, pages 451--462, 2016.

\bibitem{Burgers_1948}
J.~M. Burgers.
\newblock {\em A mathematical model illustrating the theory of turbulence},
  pages 171--199.
\newblock Academic Press, Inc., New York, N. Y., 1948.

\bibitem{Cahn1961}
J.~W. Cahn.
\newblock On spinodal decomposition.
\newblock {\em Acta. Metall.}, 9:795--801, 1961.

\bibitem{CH58}
J.~W. Cahn and J.~E. {H}illiard.
\newblock Free energy of a nonuniform system {I}: {I}nterfacial free energy.
\newblock {\em J. Chem. Phys.}, 28:258--267, 1958.

\bibitem{CCO2001}
E.~Carlen, M.~Carvalho, and E.~Orlandi.
\newblock A simple proof of stability of fronts for the {C}ahn--{H}illiard
  equation.
\newblock {\em Commun. Math. Phys.}, 224:323--340, 2001.

\bibitem{Chen2020}
L.~Chen and J.~Zhao.
\newblock A novel second-order linear scheme for the
  {C}ahn--{H}illiard--{N}avier--{S}tokes equations.
\newblock {\em J. Comput. Phys.}, 419:109782, 2020.

\bibitem{diegel2017convergence}
A.~E. Diegel, C.~Wang, X.~Wang, and S.~M. Wise.
\newblock Convergence analysis and error estimates for a second order accurate
  finite element method for the {C}ahn--{H}illiard--{N}avier--{S}tokes system.
\newblock {\em Numer. Math.}, 137:495--534, 2017.

\bibitem{NIST:DLMF}
{\it NIST Digital Library of Mathematical Functions}.
\newblock \url{https://dlmf.nist.gov/}, Release 1.2.0 of 2024-03-15.
\newblock F.~W.~J. Olver, A.~B. {Olde Daalhuis}, D.~W. Lozier, B.~I. Schneider,
  R.~F. Boisvert, C.~W. Clark, B.~R. Miller, B.~V. Saunders, H.~S. Cohl, and
  M.~A. McClain, eds.

\bibitem{EB1996}
C.~L. Emmott and A.~J. Bray.
\newblock Coarsening dynamics of a one-dimensional driven {C}ahn--{H}illiard
  system.
\newblock {\em Phys. Rev. E}, 54:4568--4575, 1996.

\bibitem{Eyre_1998}
D.~J. Eyre.
\newblock Unconditionally gradient stable time marching the {C}ahn--{H}illiard
  equation.
\newblock {\em Comm. Math. Models Microstruct. Evol.}, 53:1686--1712, 1998.

\bibitem{Eyre_1997}
D.~J. Eyre.
\newblock An unconditionally stable one-step scheme for gradient systems, 1998.
\newblock Unpublished article.
  http://www.math.utah.edu/$\sim$eyre/research/methods/stable.ps.

\bibitem{Folland99}
G.~B. Folland.
\newblock {\em Real {A}nalysis: {M}odern {T}echniques and {T}heir
  {A}pplications}.
\newblock John Wiley \& Sons, 2nd edition, 2013.

\bibitem{gal2010asymptotic}
C.~G. Gal and M.~Grasselli.
\newblock Asymptotic behavior of a {C}ahn--{H}illiard--{N}avier--{S}tokes
  system in 2{D}.
\newblock {\em Ann. Inst. H. Poincar\'{e} Anal. Non Lin\'{e}aire},
  27(1):401--436, 2010.

\bibitem{gal2011instability}
C.~G. Gal and M.~Grasselli.
\newblock Instability of two-phase flows: a lower bound on the dimension of the
  global attractor of the {C}ahn--{H}illiard--{N}avier--{S}tokes system.
\newblock {\em Phys. D}, 240(7):629--635, 2011.

\bibitem{GNRW03}
H.~Garcke, B.~Niethammer, M.~Rumpf, and U.~Weikard.
\newblock Transient coarsening behavior in the {C}ahn--{H}illiard model.
\newblock {\em Acta Materialia}, 51:2823--2830, 2003.

\bibitem{Gardner1993}
R.~A. Gardner.
\newblock On the structure of the spectra of periodic travelling waves.
\newblock {\em J. Math. Pures Appl.}, 72:415--439, 1993.

\bibitem{Gardner1997}
R.~A. Gardner.
\newblock Spectral analysis of long wavelength periodic waves and applications.
\newblock {\em J. Reine Angew. Math.}, 491:149--181, 1997.

\bibitem{giorgini2019uniqueness}
A.~Giorgini, A.~Miranville, and R.~Temam.
\newblock Uniqueness and regularity for the
  {N}avier--{S}tokes--{C}ahn--{H}illiard system.
\newblock {\em SIAM J. Math. Anal.}, 51(3):2535--2574, 2019.

\bibitem{Grant91}
C.~P. Grant.
\newblock {\em The dynamics of pattern selection for the {C}ahn--{H}illiard
  equation}.
\newblock PhD thesis, The University of Utah, 1991.

\bibitem{Grant93}
C.~P. Grant.
\newblock Spinodal decomposition for the {C}ahn--{H}illiard equation.
\newblock {\em Comm. Partial Differential Equations}, 18:453--490, 1993.

\bibitem{Guo2014}
Z.~Guo, P.~Lin, and J.~Lowengrub.
\newblock A numerical method for the quasi-incompressible
  {C}ahn--{H}illiard--{N}avier--{S}tokes equations for variable density flows
  with a discrete energy law.
\newblock {\em J. Comput. Phys.}, 274:611--625, 2014.

\bibitem{GPV96}
M.~E. Gurtin, D.~Polignone, and J.~Vi\~nals.
\newblock Two-phase binary fluids and immiscible fluids described by an order
  parameters.
\newblock {\em Math. Models Methods Appl. Sci.}, 6:8--15, 1996.

\bibitem{HW2015}
D.~Han and X.~Wang.
\newblock A second order in time uniquely solvable unconditionally stable
  numerical scheme for {C}ahn--{H}illiard--{N}avier--{S}tokes equations.
\newblock {\em J. Comput. Phys.}, 290:139--156, 2015.

\bibitem{hintermuller2018goal}
M.~Hinterm{\"u}ller, M.~Hinze, C.~Kahle, and T.~Keil.
\newblock A goal-oriented dual-weighted adaptive finite element approach for
  the optimal control of a nonsmooth {C}ahn--{H}illiard--{N}avier--{S}tokes
  system.
\newblock {\em Optim. Eng.}, 19:629--662, 2018.

\bibitem{HH77}
P.~C. Hohenberg and B.~I. Halperin.
\newblock Theory of dynamical critical phenomena.
\newblock {\em Rev. Modern Phys.}, 49:435--479, 1977.

\bibitem{H07}
P.~Howard.
\newblock Asymptotic behavior near transition fronts for equations of
  generalized {C}ahn--{H}illiard form.
\newblock {\em Comm. Math. Phys.}, 269:765--808, 2007.

\bibitem{H09}
P.~Howard.
\newblock Spectral analysis of stationary solutions of the {C}ahn--{H}illiard
  equation.
\newblock {\em Adv. Diff. Eqns.}, 14:87--120, 2009.

\bibitem{H11}
P.~Howard.
\newblock Spectral analysis for periodic solutions of the {C}ahn--{H}illiard
  equation on $\mathbb{R}$.
\newblock {\em Nonlinear Differ. Equ. Appl.}, 18:1--26, 2011.

\bibitem{kay2007efficient}
D.~Kay and R.~Welford.
\newblock Efficient numerical solution of
  {C}ahn--{H}illiard--{N}avier--{S}tokes fluids in 2d.
\newblock {\em SIAM J. Sci. Comput.}, 29(6):2241--2257, 2007.

\bibitem{Kim2007}
J.~Kim.
\newblock Phase field computations for ternary fluid flows.
\newblock {\em Comput. Models Appl. Mech. Eng.}, 196:45--48, 2007.

\bibitem{Kim2005}
J.~Kim.
\newblock A continuous surfact tension force formulation for diffuse-interface
  models.
\newblock {\em J. Comput. Phys.}, 204:784--804, 2015.

\bibitem{KL2005}
J.~Kim and J.~Lowengrub.
\newblock Phase field modeling and simulation of three-phase flows.
\newblock {\em Interfaces Free Bound.}, 7:435--466, 2005.

\bibitem{KO02}
R.~V. Kohn and F.~Otto.
\newblock Upper bounds on coarsening rates.
\newblock {\em Comm. Math. Phys.}, 229:375--395, 2002.

\bibitem{L71}
J.~S. Langer.
\newblock Theory of spinodal decomposition in alloys.
\newblock {\em Ann. Phys.}, 65:53--86, 1971.

\bibitem{Leung1988}
K.~Leung.
\newblock Interfacial properties of a driven diffusive system.
\newblock {\em J. Statist. Phys.}, 50:405--437, 1988.

\bibitem{Leung1990}
K.~Leung.
\newblock Theory of morphological instability in driven systems.
\newblock {\em J. Statist. Phys.}, 61:345--364, 1990.

\bibitem{LS03}
C.~Liu and J.~Shen.
\newblock A phase field model for the mixture of two incompressible fluids and
  its approximation by a {F}ourier-spectral method.
\newblock {\em Phys. D}, 179:211--228, 2003.

\bibitem{LT98}
J.~Lowengrub and L.~Truskinovsky.
\newblock Quasi-incompressible {C}ahn--{H}illiard fluids and topological
  transitions.
\newblock {\em Proc. R. Soc. Lond. A}, 454:2617--2654, 1998.

\bibitem{M-P2016}
P.~Mchedlov-Petrosyan.
\newblock The convective viscous cahn-hilliard equation: exact solutions.
\newblock {\em European J. Appl. Math.}, 27(1):42--65, 2016.

\bibitem{M-PK2013}
P.~Mchedlov-Petrosyan and D.~Kopiychenko.
\newblock Exact solutions for some modifications of the nonlinear
  {C}ahn--{H}illiard equation.
\newblock {\em Rep. Natl. Acad. Sci. Ukr.}, 27(12):2013, 2013.

\bibitem{Minjeaud2013}
S.~Minjeaud.
\newblock An unconditionally stable uncoupled scheme for a triphase
  {C}ahn--{H}illiard--{N}avier--{S}tokes model.
\newblock {\em Numer. Methods Partial Differ. Equ.}, 29:584--618, 2012.

\bibitem{NovickCohen_2000}
A.~Novick-Cohen.
\newblock Triple-junction motion for an {A}llen-{C}ahn/{C}ahn-{H}illiard
  system.
\newblock {\em Phys. D}, 137(1-2):1--24, 2000.

\bibitem{OW2014}
F.~Otto and M.~G. Westdickenberg.
\newblock Relaxation to equilibrium in the one-dimensional {C}ahn--{H}illiard
  equation.
\newblock {\em SIAM J. Math. Anal.}, 46(1):TODO--TODO, 2014.

\bibitem{SHH76}
E.~Siggia, B.~Halperin, and P.~Hohenberg.
\newblock Renormalization-group treatment of the critical dynamics of the
  binary-fluid and gas-liquid transitions.
\newblock {\em Phys. Rev. B}, 13:2110, 1976.

\bibitem{SIMON}
J.~Simon.
\newblock Compact sets in the space $l^p(0,t;b)$.
\newblock {\em Ann. Mat. Pure Appl.}, 146:65--96, 1987.

\bibitem{Strauss_1966_PJM_Weak_continuity}
W.~Strauss.
\newblock On continuity of functions with values in various {B}anach spaces.
\newblock {\em Pacific journal of Mathematics}, 19(3):543--551, 1966.

\bibitem{temam2001navier}
R.~Temam.
\newblock {\em Navier-Stokes equations: theory and numerical analysis}, volume
  343.
\newblock American Mathematical Soc., 2001.

\bibitem{watson2003coarsening}
S.~J. Watson, F.~Otto, B.~Y. Rubinstein, and S.~H. Davis.
\newblock Coarsening dynamics of the convective {C}ahn--{H}illiard equation.
\newblock {\em Phys. D}, 178(3-4):127--148, 2003.

\bibitem{Witelski1996}
T.~P. Witelski.
\newblock The structure of internal layers for unstable nonlinear diffusion
  equations.
\newblock {\em Stud. Appl. Math.}, 96:277--300, 1996.

\bibitem{you2022continuous}
B.~You and Q.~Xia.
\newblock Continuous data assimilation algorithm for the two dimensional
  {C}ahn--{H}illiard--{N}avier--{S}tokes system.
\newblock {\em Appl. Math. Optim.}, 85(2):5, 2022.

\bibitem{Zhao2021}
J.~Zhao.
\newblock Second-order decoupled energy-stable schemes for
  {C}ahn--{H}illiard--{N}avier--{S}tokes equations.
\newblock {\em J. Comput. Phys.}, 435:110536, 2021.

\end{thebibliography}
\end{document}